\documentclass[reqno,11pt]{amsart}

\usepackage[margin=1in]{geometry}
\usepackage{graphicx}
\usepackage{amssymb}
\usepackage{amsmath}
\usepackage{mathtools}
\usepackage{amsthm}
\usepackage{cite}
\usepackage{xcolor}
\usepackage{slashed}
\usepackage[parfill]{parskip}
\usepackage{bbm,comment}
\usepackage{physics}
\usepackage{hyperref}
\usepackage{mathtools}
\hypersetup{
    colorlinks=true,
    linkcolor=black,
    filecolor=magenta,
    citecolor= black
    }

\usepackage{mathtools}
\DeclarePairedDelimiter\ceil{\lceil}{\rceil}

 \newtheorem{assump}{Assumption}[section]
 \newtheorem{theorem}{Theorem}[section]

\newtheorem{lemma}{Lemma}[section]
\newtheorem{proposition}{Proposition}[section]

\newtheorem{remark}{Remark}[section]

\newcommand{\p}{\partial}
\newcommand{\po}{\mathbb{P}_{\mathbbm{o}}} 
\newcommand{\pn}{\mathbb{P}_{\neq}}
\newcommand{\R}{\mathbb{R}}
\newcommand{\lb}{\langle}
\newcommand{\rb}{\rangle}
\newcommand{\TT}{\mathbb{T}}
\newcommand{\G}{\mathcal{G}_k}

\numberwithin{equation}{section}

\newcommand{\rr}{\mathbb{R}}

\newcommand{\lan}{\langle}
\newcommand{\ran}{\rangle}
\newcommand{\be}{\begin{eqnarray*}}
\newcommand{\bel}{\begin{eqnarray}}
\newcommand{\lf}{\left}
\newcommand{\rg}{\right}
\newcommand{\ba}{\begin{aligned}}
\newcommand{\ea}{\end{aligned}}
\newcommand{\de}{\Delta}

\newcommand{\na}{\nabla}
\newcommand{\ep}{\epsilon}

\newcommand{\pa}{\partial}
\newcommand{\nq}{{\neq}}
\newcommand{\n}{\nonumber}
\newcommand{\wh}{\widehat}
\newcommand{\wt}{\widetilde}

\newcommand{\Torus}{\mathbb{T}}
\newcommand{\om}{\Omega}
\newcommand{\W}{\mathcal{W}}

\newcommand{\mf}[1]{\mathfrak{#1}}
\newcommand{\myb}[1]{} 
\newcommand{\myr}[1]{{#1}} 

\newcommand{\oms}{\Omega^*}

\mathtoolsset{showonlyrefs}

\allowdisplaybreaks

\begin{document}
\title{ Transition Threshold for Strictly Monotone Shear Flows in Sobolev Spaces}

\author{Rajendra Beekie}
\address{Duke University}
\email{rajendra.beekie@duke.edu}
\author{Siming He}
\address{University of South Carolina}
\email{siming@mailbox.sc.edu}

\begin{abstract}
 We study the stability of spectrally stable, strictly monotone, and smooth shear flows  in the 2D Navier-Stokes equations on $\TT \times \R$ with small viscosity $\nu$. We establish nonlinear stability in $H^s$ for $s \geq 2$ with a threshold of size $\epsilon \nu^{1/3}$ for time smaller than $c_*\nu^{-1}$ with $\epsilon, c_* \ll 1$. Additionally, we demonstrate nonlinear inviscid damping and enhanced dissipation. 
\end{abstract}




\maketitle
\setcounter{tocdepth}{1}
\pagestyle{plain}




\section{Introduction}
Consider the 2D incompressible Navier-Stokes equations 
\begin{align}
\label{nse:velo}
    \begin{cases}&\p_t v + v \cdot \nabla v - \nu \Delta v + \nabla p = 0,\\ &\quad \nabla \cdot  v = 0,\\
    &v(t = 0) = v_{\rm in},
  \end{cases}  
\end{align}
on the domain $ (t, x,y) \in [0, \infty) \times \TT \times \R $ where $\nu > 0$ is the viscosity (or inverse Reynolds number), $p$ the pressure, and $v_{ \rm in}$ the initial data. An important class of solutions of \eqref{nse:velo} are shear flows of the form $v(t,x ,y) = (b(t,y) , 0)$ where $b(t,y)$ denotes a solution of the equation
\begin{equation}
\label{heat}
    \p_t b(t,y) - \nu \p_y^2 b(t,y)=0, \quad b(0, y) = b(y).
\end{equation}
We are interested in the stability of monotone shears, namely $\p_y b(y) > 0$, which is preserved for $b(t, y)$. See Assumption \ref{assum:b} and the following remarks for a more detailed discussion of the assumptions we have on $b(y)$. To study the stability of these shears, we consider solutions of \eqref{nse:velo} of the form
\begin{align}
    v(t,x ,y) = (b(t, y), 0) + u(t, x, y),
\end{align}
where the perturbation $u(t,x,y) = (u^1(t,x ,y) , u^2(t, x ,y))$ satisfies 
\begin{align}
    \label{velocity:shear}
\begin{cases}&\p_t u + b(t,y) \p_x u  + (\p_y b(t,y)u^2, 0) - \nu \Delta u + \nabla p = 0,\\
    &\quad\nabla \cdot u = 0,\\
    &u(0, x, y) = u_{\rm in} (x,y),
    \end{cases}
\end{align}
and $u_{\rm in}(x,y)$ is small in a sense to be made precise in Theorem \ref{main}. It is convenient to work with the perturbation vorticity $\omega := -\p_y u^1 + \p_x u^2$ which satisfies the equation
\begin{align}
\label{shear}\begin{cases}
    &\p_t \omega + b(t,y) \p_x \omega  - b''(t,y)\p_x \psi - \nu \Delta \omega  +  \nabla^{\perp} \psi \cdot \nabla \omega = 0, \\
    &\quad\Delta\psi  = \omega,\\
    &\omega(0, x, y) = \omega_{\rm in}(x,y),
\end{cases}
\end{align}
where $\nabla^{\perp} = (-\p_y, \p_x)$. To solve for the stream function $\psi$, we impose that $\psi$ has zero mean over $\TT \times \R$.  We can recover $u$ from $\omega$ via the Biot-Savart law $u = \nabla^{\perp} \Delta^{-1}\omega$.

\subsection{Statement and Discussion of Main Result}
The main focus of this work is to precisely study the stability properties of  \eqref{shear} as $\nu \to 0$. Specifically, we are interested in the following stability threshold problem: Given a function space $X$, determine the minimal $\beta(X)$ such that

 \begin{align}
    \label{stab:threshold}
    \norm{\omega_{\rm in}}_X \ll \nu^{\beta} &\longrightarrow \text{ asymptotic stability}, \\
    \norm{\omega_{\rm in}}_X \gtrsim \nu^{\beta} &\longrightarrow \text{ instability}. 
\end{align}

 The motivation for the formulation of \eqref{stab:threshold} comes from the problem of the ``subcritical transition" \cite{BGM19review} 
 where in experiments, instabilities are observed at a lower Reynolds number than the linear theory predicts. In \cite{Kelvin87}, Kelvin suggested that the fluid may be increasingly sensitive to disturbances at higher Reynolds numbers, which is a qualitative version of \eqref{stab:threshold}. In the special case of Couette flow $b(t,y) = y$, there has been a significant amount of work done on the stability threshold problem. In \cite{Romanov73}, it was shown that Couette flow is nonlinearly stable for all Reynolds numbers (with the allowable size of the perturbation depending on the Reynolds number). In \cite{BMV16}, stability independent of Reynolds number was proven for $X$ being Gevrey-$s$ with $s \in [1,2)$. At finite regularity, for $X = H^N$ with $N > 1$, the threshold $\beta = 1/2$ was established for Couette and shears close to Couette. The current state-of-the-art for Couette is summarized in the following table: 
\begin{itemize}
    \item \cite{MZ20}: $\norm{\omega_{\rm in}}_{H_x^{log}L_y^2} \leq \epsilon \nu^{1/2} \longrightarrow \text{ asymptotic stability;}$
    \item \cite{MasmoudiZhao2022}: $\norm{\omega_{\rm in}}_{H^N} \leq \epsilon \nu^{1/3},\, N\geq40 \longrightarrow \text{ asymptotic stability;}$
    \item \cite{LiMasmoudiZhao24}:  $\norm{\omega_{\rm in}}_{H_x^{1}L_y^2} \geq M \nu^{1/2 -\delta} \longrightarrow \text{ instability, $M$ independent of 
} \nu \text{ and any } \delta > 0   ;$ 
    \item \cite{WZ23}: $\norm{\omega_{\rm in}}_{H^N} \leq \epsilon \nu^{1/3},\, N \geq 2 \longrightarrow \text{ asymptotic stability;}$
    \item \cite{LiMasmoudiZhao22}: $\norm{\omega_{\rm in}}_{\mathcal{G}^s } \leq \epsilon \nu^{\beta} \longrightarrow \text{ asymptotic stability}$ with $s \in \left[1, \frac{2- 3\beta}{1 - 3 \beta} \right],\  \beta \in [0, 1/3]$; 
    \item \cite{DM23}:  $\norm{\omega_{in}}_{\mathcal{G}^s} \leq \epsilon \longrightarrow \text{ instability}$ with $s \in (2, \infty). $ 
\end{itemize}
For general monotonic shear flows, there are significantly fewer results. The key issue is the presence of $\p_y^2b(t,y) \p_x\psi$ in \eqref{shear} which can lead to unstable eigenvalues \cite{lin03} and at a technical level breaks the transport-diffusion structure that simplifies the analysis of  Couette. We highlight the recent work \cite{LZ23}, which proved asymptotic stability in $H_x^{log} L_y^2$ for perturbations of spectrally stable monotonic shear flows (see also \cite{IJ23} and \cite{MZ24} for proofs of nonlinear inviscid damping for spectrally stable monotonic shear flows). Our result relies on the linear analysis done in \cite{J23} (see also \cite{CWZ23} for the case of a channel with non-slip boundary conditions). Following \cite{J23}, we state the main assumptions on $b(y)$.    
\begin{assump}
\label{assum:b}
    We assume that the shear flow $b(y)$ satisfies the following conditions:
    \begin{itemize}
        \item We normalize $b$ so that $b(0,0) = 0$. 
        \item For some $\sigma_0 \in (0,1)$ we have
        \begin{align}
             \p_y b(y) \in [\sigma_0, 1/\sigma_0] \text{ for } y \in \R, \quad \mathrm{support }\ \p_y^2 b \subseteq [-1/\sigma_0, 1/ \sigma_0 ], \quad \sup_{\xi \in \R} e^{\sigma_0 \lb \xi \rb^{1/2}} |\widehat{b''}(\xi)| \leq 1/\sigma_0 ;
        \end{align}
       \item The linearized operator $L_k: L^2(\R) \to L_{loc}^2(\R)$ with $k \in \mathbb{Z} \setminus \{ 0\}$ defined for $g \in L^2(\R)$ as 
       \begin{equation}
       \label{rayleigh}
           L_k g(y) := b(y) g(y) - b''(y)\varphi, \quad (\p_y^2 - k^2) \varphi = g, \quad y \in \R
       \end{equation}
       has purely continuous spectrum $\R$. 
    \end{itemize}
\end{assump}
\begin{remark}
    For all $t \geq 0$ we have
    \begin{align}\label{assump:b_t}
        \p_yb(t,y)\in [ \sigma_0, 1/\sigma_0] \text{ for } y \in \R, \quad \sup_{\xi \in \R} e^{\sigma_0 \lb \xi \rb^{1/2}} |\widehat{b''}(t,\xi)| \leq 1/\sigma_0 
    \end{align}
    as immediate consequences of $\p_yb(t,y)$ satisfying the heat equation. 
\end{remark}
\begin{remark}
     The infinite regularity assumption is not necessary for Theorem \ref{main} to hold as we will only be interested in finite regularity. Similarly, the assumption that  $b''$ is compactly supported
     can be weakened to assuming that $b''$ decays sufficiently fast as $|y| \to \infty$. The assumption that $L_k$ has only continuous spectrum is essential. Finally, the assumption that $b'(t,y)$ has upper and lower bounds is extensively used throughout the linear and nonlinear analysis. 
\end{remark}
Before stating our main theorem, we introduce the projection operators 
\begin{align} \label{avg_rem}
\mathbb{P}_{\mathbbm{o}} f(t,y)=\frac{1}{2\pi}\int_{\Torus} f(t,x,y)dx,\quad\mathbb{P}_\nq f(t,x,y)=f(t,x,y)-\po f(t,y).
\end{align}
\begin{theorem}
\label{main}
{ Let $s \geq 2$, and suppose that the shear profile $b(t,y)$ solves the equation
\begin{equation}
    \p_t b(t,y) - \nu \p_y^2 b(t,y)=0, \quad b(0, y) = b(y),
\end{equation}
where the shear $b(y)$ satisfies Assumption \ref{assum:b}. Then there exist $ \epsilon_0,\nu_0, c_* ,\delta\in (0,1)$, which depend only on the initial shear profile $b(y)$, such that the following statements hold.}
Assume that $\epsilon \in (0, \epsilon_0)$,  $\nu \in (0, \nu_0)$, and that the initial perturbation satisfies
\begin{align}
    \norm{\omega_{\rm in}}_{H^s(\TT \times \R)} +   \norm{\po u_{\rm in}^1}_{L^2(\R)} \lesssim_{}  \epsilon \nu^{1/3}.
\end{align}
Then  the  solution $(u, \omega)$ of \eqref{velocity:shear} and \eqref{shear} exhibits the following three properties for $t \in [0 ,c_*\nu^{-1}]$: 
\begin{enumerate} 
    \item Stability 
    \begin{align}
\label{stability}
\norm{\omega (t , x + tb(t,y), y  )}_{H^s(\TT \times \R)} 
+  \norm{\po u^1 }_{L^2(\R)}  
\lesssim_{} \epsilon\nu^{1/3}; 
\end{align}

\item  Enhanced Dissipation
\begin{align}
    \label{Enhanced:dissip}
    \norm{{e^{\delta \nu^{1/3}|\pa_x|^{2/3}t }}\pn \omega(t , x + tb(t,y), y   )}_{H^s(\TT \times \R)} \lesssim_{} \epsilon \nu^{1/3};
\end{align}
\item Inviscid Damping
\begin{align}
\label{inviscid:damping}
    \norm{ \pn u(t, x + tb(t,y), y)}_{L_t^2[0, c_* \nu^{-1}]H^s(\TT \times \R)} \lesssim_{} \epsilon \nu^{1/3}.
\end{align}
\end{enumerate}
\end{theorem}
\begin{remark}
    A more detailed description of the perturbations is given in Propositions \ref{pro:short} and \ref{pro:long}. 
\end{remark}
\begin{remark}
    The principal reason why Theorem \ref{main} is not global in time is we only assume spectral stability for $b(t,y)$ at $t = 0$. For $t \gg \nu^{-1}$,  $b(t,y)$ may no longer be spectrally stable. This is also the reason why \cite{LZ23} needs to assume that $b(t,y)$ is spectrally stable for $t\geq 0$ in order to obtain a global result. 
\end{remark}
\begin{remark}
    The regularity in Theorem \ref{main} matches the regularity obtained in \cite{WZ23} for Couette.  
\end{remark}


\subsection{Outline and Main Challenges }
Recent literature identifies enhanced dissipation and inviscid damping as the two main mechanisms guaranteeing the stability of the shear flow. 
These two phenomena can be easily pinpointed for a simplified linear model to \eqref{shear}, which is the passive scalar equation subject to the Couette flow $(y,0)$ and small viscosity $\nu\leq 1$: 
\begin{align}\label{PS}
    \pa_t f+y\pa_x f=\nu \de f,\quad f(t=0)=f_{\rm in},\quad (x,y)\in \Torus\times \rr.
\end{align}
The above equation has been extensively studied since the work of Kelvin \cite{Kelvin87} and Orr \cite{Orr07}. If the ambient shear flow $b(t,y)=y$, then the linear part of the vorticity perturbation equation \eqref{shear} reduces to that of \eqref{PS}. By taking the change of coordinate $z=x-ty,\, v=y$ and set $F(t,z(x,y),v(y))=f(x,y)$, one ends up with the following equation
\begin{align}
\pa_t F(t,z,v)=\nu(\pa_z^2+(\pa_v-t\pa_z)^2)F(t,z,v).
\end{align}
Here, the Laplacian $\de$ in the $(x,y)$-coordinate is transformed into a degenerate elliptic operator in the $(z,v)$-frame. By taking the Fourier transform $(z,v)\rightarrow (k,\eta)$, one ends up with 
\begin{align}\label{PS_zv}
    \pa_t \wh F(t,k,\eta)=-\nu(|k|^2+|\eta-kt|^2) \wh F(t,k,\eta).
\end{align}
As long as the wave number $k\neq 0$, a direct integration in time yields that there exist universal constants $C\geq 1,\ \delta\in(0,1)$ such that the following estimate holds
\begin{align}\label{PS_keta_bd}
   | \wh F(t,k,\eta)|\leq C|\wh F_{\rm in}(k,\eta)|\exp\{-\delta\nu^{1/3}t\},\quad k\neq 0,\quad \forall t\geq 0.
\end{align}
Hence, the solution decays with the rate $\mathcal{O}(\nu^{1/3})$, which is much larger than the classical heat dissipation rate, $\mathcal{O}(\nu)$ as long as $\nu$ is small. This accelerated dissipation phenomenon is called the \emph{enhanced dissipation induced by shear flow advection}. Thanks to the heterogeneous nature of this phenomenon, we use the decomposition from \eqref{avg_rem}. 
It is worth highlighting that the enhanced dissipation phenomena and the inviscid damping phenomena only take effect on the ``remainder'' part of the solution $\mathbb{P}_\nq f.$ 
There is extensive literature devoted to understanding the enhanced dissipation for passive scalar flows, and we refer interested readers to the classical work \cite{BCZ15} and recent extensions \cite{AlbrittonBeekieNovack21,Wei18,ElgindiCotiZelatiDelgadino18, FengIyer19}.

Thanks to the Biot-Savart law $u=\na^\perp \de^{-1}\omega$, the $L^2$ norm of the velocity perturbation $\|u\|_{L_{x,y}^2}$ is equivalent to the $\dot H_{x,y}^{-1}$-norm of the solution $\omega$. For the passive scalar equation \eqref{PS_zv}, the $\dot H^{-1}_{x,y}$-norm evolution of the solution has the following characterization 
\begin{align}
\int_0^\infty \|\pn f\|_{\dot H_{x,y}^{-1}}^2 dt=\int_0^\infty\int  \pn f(-\de)^{-1}\pn f  dxdy dt=\sum_{k\neq 0}\int_0^\infty \int_\rr \frac{|\wh F(t,k,\eta)|^2}{|k|^2+|\eta-kt|^2}d\eta dt.
\end{align}
Now, thanks to the bound \eqref{PS_keta_bd}, we end up with a viscosity $\nu$ independent estimate:
\begin{align}
    \int_0^\infty \|\pn f\|_{\dot H_{x,y}^{-1}}^2 dt\lesssim \sum_{k\neq 0}\frac{1}{|k|^2}\int_{\rr}\int_0^\infty \frac{|\wh F_{\rm in}(k,\eta)|^2}{1+|t-\eta/k|^2} dtd\eta\leq  \||\pa_z|^{-1}\pn F_{\rm in}\|_{L^2}^2.
\end{align}
Since this $L^2_tH_{x,y}^{-1}$-estimate persists as $\nu\rightarrow 0$, it is called \emph{inviscid damping}. By invoking higher regularity, one can derive sharp pointwise in time decay
\begin{align}
    \|\pn f\|_{\dot H_{x,y}^{-1}}^2\lesssim \sum_{k\neq 0}\frac{1}{|k|^2}\int_{\rr} \frac{(1+|\eta|^2)|\wh F_{\rm in}(k,\eta)|^2}{(1+|\eta|^2)(1+|t-\eta/k|^2)} d\eta\lesssim \frac{ \||\pa_z|^{-1}\pa_v \pn F_{\rm in}\|_{L^2}^2}{1+t^2}.
\end{align}
Even though the inviscid damping decay is polynomial in time, it takes effect on a time scale that is much shorter than the enhanced dissipation time scale $\mathcal O (\nu^{-1/3})$. Hence, inviscid damping is crucial for stabilizing the motion of the fluid in the initial time layer. There is extensive literature on the inviscid damping of the Couette flow and other shear flows, and we refer the readers to the classical works \cite{BM13,Zillinger2014,Zillinger16} and extensions, \cite{WeiZhangZhao15,WeiZhangZhao20,J23,MZ24,IJ23,ChenLiWeiZhang18, DengZillinger21}.

The goal of our work is to extend these simple observations on the passive scalar equation \eqref{PS} to the general setting of \eqref{shear}. The main challenges that one faces are as follows:
\begin{itemize}
    \item The effect from the nonlocal term $b''\pa_x\psi$ in \eqref{shear}. A key difficulty in extending the aforementioned estimates to the general strictly monotone shear flows arises from the linear nonlocal term $b''\pa_x\psi$. If the second derivative of the shear profile is small, i.e., $\|b''\|\ll 1$, one can still carry out classical energy methods to recover the enhanced dissipation and inviscid damping estimates above. However, once $b''$ becomes large, the linear nonlocal effect is no longer negligible, and the direct energy approach fails.

    \item Elliptic estimates. To identify and employ the inviscid damping mechanism, detailed elliptic estimates of the stream function (or related quantities) are essential. We note that in the strictly monotone shear flow setting, the derivative of the shear profile $b'$ is only bounded and does not decay near spacial infinity. This will cause extra difficulties compared to the finite domain situations, e.g., channel \cite{IJ23}.  
    \item Commutator structures of the inviscid damping Fourier multipliers for the initial time layer. As we highlighted before, the inviscid damping effect plays a key role in stabilizing the dynamics of the fluid for $t\ll \nu^{-1/3}$. In the classical literature \cite{Zillinger2014,BVW18}, a norm induced by an inviscid damping multiplier ($\W_I$ in \eqref{W_I}) is utilized to derive the required estimate for the nonlinear system. However, our approach requires us to derive a nice commutator estimate of the multiplier for time $t\leq \nu^{-1/6}$. 
\end{itemize}

To address these challenges, we utilize the following ideas.
\begin{itemize}

\item The plan to address the nonlocal term is to invoke the delicate linear estimate derived in the paper \cite{J23}. To apply them in a nonlinear setting, we follow the idea in the paper \cite{IJ23} to decompose the solution $\omega$ into the linear part and the auxiliary part. The linear part of the solution encodes the linear nonlocal effects in \eqref{shear}, and we can use the linear estimates in \cite{J23} to resolve the problems coming from the linear nonlocal term. For the auxiliary part, we rely on the Couette flow analysis in \cite{WZ23}. 
\item To derive the elliptic estimates, we utilize a `local' and `far-field' decomposition of the Green's functions associated with our elliptic operators, leveraging the fact that $b''(t,y)$ converges to zero near spatial infinity. In the compact local region, we follow the approach in \cite{IJ23}, while in the far-field region, we treat the elliptic operator as a perturbation of a constant coefficient elliptic operator.

\item To address the last difficulty, we observe that the $\pa_\eta \W_I$ has a nice structure, and one can handle nonzero mode interactions relatively easily. However, the interaction between the nonzero and zero mode is challenging to deal with. The problem mainly stems from the fact that one typically defines $\W_I(t,k=0,\eta)\equiv 1$. It turns out that for a short time $t\leq \nu^{-1/6}$, one can slightly adjust the zero-mode component of the multiplier \eqref{WIcirc} and guarantee a nice commutator estimate (Lemma \ref{lem:Acm_sh}).  

\end{itemize}

\subsection{Acknowledgment}
The authors would like to thank Dallas Albritton, Tarek Elgindi, Hao Jia, Hui Li, and Weiren Zhao for fruitful discussions and insightful suggestions. RB's research is partially supported by NSF Grant DMS-2202974. SH's research is partially supported by NSF  Grant DMS-2006660, 2304392, and 2406293.  

\section{Notation and Coordinate System}

\subsection{Coordinate System}
We consider the following change of variable on the time interval $[0,c_*\nu^{-1}]$: 
\begin{align}\label{chg_of_v}
z(t,x,y):=&x-t\ b(t,y),\quad 
v(t, y):=b(t,y).
\end{align}
We denote 
\begin{align}
    \Omega(t,z(t,x,y),v(t, y))&:=\omega(t,x,y), \\
    U(t, z(t,x,y),v(t, y))&:=  u(t, x ,y),\\
    \Psi(t, z(t,x,y),v(t, y))&:=  \psi(t, x ,y), \\
    B(t, v(t,y)) &:=\pa_y b(t,y),\\
    B'(t, v(t,y)) &:= \pa_{yy}b(t,y).
\end{align} 
Note that 
\begin{align}
   B'(t,v) =   B(t,v) \p_v B(t,v). \label{B_B'}
\end{align} 
\myb{HS: $B=v',\ B'=v''$, Can use $\po, \pn$.}
We will also use the notation
\begin{equation}
\label{B0}
    B_0(v) := B(0,v) = \p_yb(b^{-1}(v)), \quad B_0'(v) := B'(0, v ) = \p_y^2b(b^{-1}(v)).
\end{equation}
In what follows, we will drop the $(t,v)$ dependence when writing $B, B_0$ when there is no risk of confusion.   
In the new coordinate system, we have 
\begin{align}
   \label{derivatives} \begin{cases}
   &\nabla_L := (\p_z, \p_v - t\p_z), \quad \Delta_L := \p_z^2 + (\p_v - t\p_z)^2, \\
   &\pa_z F(t,z(t,x,y),v(t,y))=\pa_x f(t,x,y),\\
    & \pa_v F(t,z(t,x,y),v(t,y))=\lf(\frac{1}{v_y}\pa_y+t\pa_x\rg) f(t,x,y),\\
     &\nabla f(t,x ,y) =  \lf( \p_z F (t,z(t,x,y) ,v(t,y)), B(\p_v - t\p_z )F(t,z (t,x,y),v(t,y)) \rg), \\
    & \de f= \Delta_t F =  (\pa_z^2+ ({B}(\pa_v-t\pa_z))^2)F(t,z(t,x,y),v(t,y)).
    \end{cases}
\end{align}
In the new coordinate system, \eqref{shear} can be written as 
\begin{align}
\label{profile}
    \p_t \Omega  - B'\p_z \Psi - \nu \tilde{\Delta}_t \Omega  + B \nabla_L^{\perp}\Psi \cdot \nabla \Omega  &=  0,\quad 
     \Delta_t \Psi  = \Omega, 
\end{align}
where 
\begin{equation}
\label{lap:til}
    \tilde{\Delta}_t := B^2 (\p_v -t\p_z)^2 + \p_z^2.
\end{equation}
We isolate the zero-mode of $\nabla_L^{\perp}\Psi$ in the nonlinearity and write \eqref{profile} as 
\begin{align}
\label{profileprime}
    \p_t \Omega  - B'\p_z \Delta_t^{-1} \pn \Omega - \nu \tilde{\Delta}_t \Omega  -\po U^1 \p_z \Omega  +   B \nabla_L^{\perp} \Delta_t^{-1}\pn\Omega \cdot \nabla_L \Omega  =  0, \quad
       -B\p_v\po U^1 = \po \Omega .\quad
\end{align}
Following \cite{IJ23}, we introduce the operator 
\begin{equation}
\label{lap0}
    \Delta_0  := B_0^2 (\p_v -t\p_z)^2 + B_0'(\p_v -t\p_z) + \p_z^2 
\end{equation}
 and decompose the solution $\Omega$ as $\Omega  = \Omega^* + F$ where $F$ (the ``linear profile) will satisfy a nonhomogeneous linear equation with 0 initial data and $\Omega^*$ (the ``auxiliary profile") will satisfy an equation that can be handled like the Couette case in \cite{WZ23}. More explicitly, $F$ satisfies the equation
 \myb{HS: $*$ aux, lin linear, $\Phi$ frozen time. }
\begin{align}
\label{Flin}
    &\p_t F - B_0'\p_z \Delta_0^{-1}\pn F - \nu \Delta_{0} F = B_0' \p_z \Delta_0^{-1} \pn \Omega^*, \quad F(0, z,v) = 0, 
\end{align}
 and   $\Omega^*$  satisfies the equation
 \begin{align}
 \label{profileaux}
    \begin{cases}&\p_t \Omega^* - \nu \tilde{\Delta}_t \Omega^*   - \po U^1\p_z (F + \Omega^*) + B \nabla_L^{\perp} \Delta_t^{-1}\pn(F + \Omega^*) \cdot \nabla_L ( F + \Omega^*)  \\
    &\quad= (B' - B_0')  \p_z \Delta_t^{-1}\pn (F + \Omega^*) +  
    B_0'(\Delta_t^{-1} - \Delta_0^{-1}) \p_z \pn (F + \Omega^*)
    + \nu (\tilde{\Delta}_t  - \Delta_{0}) F ,  \\
    & -B\p_v\po U^1 = \po \Omega^*,\\
     &\Omega^*(0, z, v) = \Omega(0, z, v)
     .  
     \end{cases}
\end{align}
 
The structure of \eqref{profileaux} is complicated, so we briefly comment on the plan of attack. The first crucial observation is that the term $-B' \p_z \Delta_0^{-1}\pn\Omega^*$ is no longer in the auxiliary profile equation \eqref{profileaux}. Since $B'$ is not assumed to be small, its presence would have ruined any chances of doing direct energy estimates. Instead, we use the evolution of the  ``linear profile" $F$ to balance this potentially harmful nonlocal term, which relies on the linear theory developed in \cite{J23}. The right-hand side of \eqref{profileaux} can be made small by taking $\nu t$ sufficiently small, which allows us to estimate it directly. Using the bootstrap assumptions on the ``auxiliary profile" $\Omega^*$ in combination with the linear theory for \eqref{Flin}, we can get bounds for the linear profile $F$,  which implies control on the profile of the full vorticity perturbation $\Omega = F + \Omega^*$. Finally, to control the velocity $\po U^1$, we directly use the velocity perturbation equation to prove $L^2$ control on $\po U^1$ (see Section \ref{u1:0:mode}).

\subsection{Notations }
We employ the following notations in the paper.

\begin{enumerate}
    \item  We define the notations
     \begin{align}
        \lb x \rb^s &:= (1 + x^2)^{s/2},\\
        \lb x, y \rb^s &:= (1 + x^2)^{s/2} + (1 + y^2)^{s/2} .\label{mul_bracket}
    \end{align}
    Throughout the text, the parameter $s$ is used to denote the regularity level. 
    \item For positive constants $A, B>0$, the notation $A\lesssim B$ ($A\gtrsim B$, respectively) means that there exists a positive constant $C>0$ such that $A\leq CB$ ($A\geq B/C$, respectively). We write $A \approx B$ if $A \lesssim B$ and $B \lesssim A$.  If the constant $C$ depends on certain parameters (e.g., $s$, $\sigma_0$), we will highlight them in subscript (e.g., $\lesssim_{s,\sigma_0}$). 
   
    \item The definition of the $\nu$-independent constant $C$ changes from line to line. The constants $C_0$ and $C_1$ are reserved for bootstrap constants. 
     We choose to recycle the symbols $T_1,\ T_2,\cdots$. In the proof of individual lemmas, they are referring to different terms in the decomposition. However, after the conclusion of the proof, these ``local'' symbols will be redefined elsewhere.  
    \item The $\sigma_0$ is the control parameter for $b(t,y)$ (Assumption \ref{assum:b}) and $\theta_0$ is the control parameter for $B$ (Lemma \ref{lem:B:regularity}). 
    The parameter $\delta$ denotes the enhanced dissipation coefficient of the nonlinear dynamics, and $\delta_{lin}$ denotes the linear enhanced dissipation coefficient.  The notation $T_*$ denotes the end of the bootstrap time interval, and $c_*\nu^{-1}$ is the time horizon of this paper. 
    \item We use standard Fourier transform conventions for $(z,v)\in\Torus\times \rr=[-\pi,\pi]\times \rr$, i.e.,
    \begin{align}
\wh{f}(k,\eta)=\frac{1}{2\pi}\iint e^{-ikz-i\eta y}f(z,v)dzdv.
    \end{align}
 We also use the $x$-Fourier transform:  \begin{align}
\wh{f}(k,v)=\wh{f}_k(v)=\frac{1}{2\pi}\int e^{-ikz}f(z,v)dz.
    \end{align}
The Fourier variables $k,\ell$ correspond to the $z$-variable, and the Fourier variables $\eta,\xi$ correspond to the $v$-variable. The symbols $(\cdots)^{\wedge}$ and $(\cdots)^\vee$ are reserved for the Fourier transform and inverse Fourier transform, respectively. The Sobolev space that we use can be represented as 
    \begin{align}
        \|f\|_{H^s(\Torus\times \rr)}^2:=\sum_{k\in \mathbb Z}\int \lan k,\eta\ran^{2s}|\wh{f}(k,\eta)|^2d\eta.
    \end{align} 
    \item  For $f(z,v): \TT \times \R \to \mathbb{C}$
\begin{align}
\int f(z,v) \, dV = \int_{\TT \times \R} f(z,v) \, dzdv.
\end{align} 
\item There are multiple different versions of Laplace-like operators. We collect the references to them here: $\Delta_L$ \eqref{derivatives}, $\Delta_0$ \eqref{lap0}, $\Delta_t$ \eqref{derivatives}, $ \Tilde{\Delta}_t$ \eqref{lap:til}, and $\de_{B_0}$ \eqref{db}. 
  We will slightly abuse notation by suppressing the $k$ dependence of the Laplacians when acting on $z$ Fourier transformed functions, e.g. $\Delta_L f_k(v) = (-|k|^2+(\pa_v-ikt)^2)f_k(v).$
\end{enumerate}


\section{Bootstrap and Key Linear Estimates}
In order to introduce the functional setup to analyze the system \eqref{Flin}, \eqref{profileaux}, we apply the spacial Fourier transform $f(t,z,v)\xrightarrow{\mathcal{F}} \wh f(t,k,\eta)$.  Next, we introduce the following multipliers:
\begin{subequations}
    \label{multipliers}
\begin{align}
&\text{a) Enhanced Dissipation Multiplier:}\n\\
 &\quad\W_{\nu}(t,k,\eta) :=\pi-\arctan\left(\frac{\nu^{1/3}|k|^{2/3}}{\myr K}\left(t-\frac{\eta}{k}\right)\right)\mathbbm{1}_{0<|k|\leq \nu^{-1/2}};\label{W_nu}\\
&\text{b) Inviscid Damping Multiplier:}\n\\
&\quad\mathcal W_I(t,k,\eta):=\pi-\arctan\left(\frac{1}{\myr{K}}\lf(t-\frac{\eta}{k}\rg)\right)\myr{\mathbbm{1}_{k\neq0}};\label{W_I}\\
&\text{c) Echo Multiplier:}\n\\
&\quad\mathcal W_E(t,k,\eta):= \pi  +\frac{1}{\pi^2}\sum_{\ell\in \mathbb{Z}\backslash\{0\}}\frac{1}{|\ell|^2}\lf(\frac{\ell}{|\ell|}\arctan\lf(\frac{1}{\myr{K}}\lf(\frac{\eta-\ell t}{1+|k-\ell|+|\ell|}\rg)\rg)\rg)\mathbbm{1}_{\myr{0\leq}|k|<\nu^{-1/2}}.\label{We}
\end{align}
Here, the parameter $K\geq 1$ is a constant that depends on the shear profile $B(t,v)$ ($\|B'\|_{L_t^\infty H_v^s}+\|B\|_{L_{t,v}^\infty}+\|B^{-1}\|_{L_{t,v}^{\infty}}$) and will be chosen later in the proof. The $\W_\nu$ is the multiplier that captures the enhanced dissipation, the $\W_I$ multiplier captures the inviscid damping, and the $\W_E$ multiplier controls the echo cascade. We note that the multipliers $\mathcal{W}_\nu \ \eqref{W_nu},\, \mathcal{W}_I$ \eqref{W_I}, and $\mathcal{W}_E$ \eqref{We} are bounded and take values in $[\frac{\pi}{2}, \frac{3\pi}{2}]$, and hence the norms they induce are equivalent to the $L^2$-based Sobolev norms. These multipliers have their origin in \cite{BVW18}, \cite{WZ23}. We remark that our multipliers are slightly different from those of \cite{WZ23} in the sense that the $\W_\nu$ has frequency cutoff $0<|k|\leq \nu^{-1/2}$, and $\W_I$ has support on all the nonzero modes. The echo multiplier $\W_E$ is identical to the one in \cite{WZ23}.  In order to obtain a nice commutator structure for the short time $t\in[0,\nu^{-1/6}]$, we further define, \begin{align}\label{WIcirc}
\W_I^\circ(t,0,\eta):=\pi+\arctan \lf(\frac{1}{K}\frac{\eta-t/2}{1/2}\rg),\quad t\leq \nu^{-1/6}.
\end{align}
\end{subequations}
We define
\begin{align}
\label{M_N_Omega}
\mathfrak M(t, k, \eta):= 
\begin{cases}
    \mathbbm{1}_{k = 0}\W_I^\circ(t,\eta) + \mathbbm{1}_{k \neq 0} \mathcal W_I(t, k , \eta), & \forall t \in [0, \nu^{-1/6}]; \\
    \W_\nu \mathcal W_I \mathcal W_E(t, k ,\eta), &\forall t \in [\nu^{-1/6}, \infty).
\end{cases}
\end{align}
We slightly abuse the notation in the sense that at $t=\nu^{-1/6}$, the two definitions of $\mf M$ do not match. However, it will be clear from the context which multiplier we are considering.   
To capture the $|k|$-dependent enhanced dissipation, we introduce the Fourier multiplier
\begin{align}\label{zeta}
    \zeta_k(t):=\mathbbm{1}_{k\neq 0}\exp\lf\{\delta \nu^{1/3}(|k|^{2/3}+1)t \rg\}+\mathbbm{1}_{k=0};\quad (\zeta(t,|\pa_z|)f(t,z,v))^\wedge=\zeta_k(t)\wh f(t,k,\eta).
\end{align}
The $\zeta_k$ weight has the following estimate:
\begin{align}
    \zeta_k(t)
    \leq&\exp\lf\{\delta \nu^{1/3}(|k-\ell|^{2/3}+|\ell|^{2/3}+1)t\rg\}
    \leq  \begin{cases}\zeta_\ell(t)\zeta_{k-\ell}(t)\exp\{-\delta\nu^{1/3}t\},\,&  (k-\ell)\ell\neq 0;\\
    \zeta_\ell(t)\zeta_{k-\ell}(t),\,&  (k-\ell)\ell= 0 .
    \end{cases}\label{zeta_product}
\end{align}
Further, define that for $ s\geq 2$,
\begin{align}\label{A_N_Omega}
A(t,k,\eta) := \zeta_k(t) \mf M (t, k , \eta) \lb k, \eta \rb^s,\quad \wt A(t,k,\eta):= \mf M (t, k , \eta) \lb k, \eta \rb^s.
\end{align}
Here, the constant $0<\delta<1$ will be chosen along the proof and will be independent of the $\nu$ and the solution $\oms,F$. We observe that these multiplier functions \eqref{multipliers} are monotonically decreasing in time.  Hence, when we take time derivatives of the norm induced by these Fourier multipliers, we obtain multiple negative terms, which are referred to as the Cauchy-Kovalevskaya terms ($CK$-terms). These damping terms are defined as follows:
\begin{align}
CK_\iota[H]:=\lf\|A\sqrt{\frac{-\pa_t\W_\iota}{\W_\iota}}H\rg\|_{L^2}^2,\quad \iota\in\{\nu, I,E\}.    \label{CK_terms}
\end{align}
Here $H$ is an arbitrary $H^s$-function. 
The properties of these multipliers are collected in the appendix. Next, we are ready to lay out the bootstrap argument.

\noindent
{\bf Short time: $t\leq \nu^{-1/6}$.} We assume that  $s\geq 2$ and $[0,T_*]\subset[0,\nu^{-1/6} ] $ is the largest time interval on which the following \textbf{Hypotheses} hold: \myb{HS: It seems to me that the $C_0$ in \eqref{Hyp_reg_sh} can be replaced by $4$. Is that true? If we can replace it with a simpler expression, then we don't need to change the $\eqref{Low_reg_1}_{r.h.s}$ to $C(C_0)\ep^2\nu^{2/3}...$} 
\begin{align} \begin{cases}
&\displaystyle{
\| A \oms(t)\|_{L^2}^2+\int_{0}^{t}\left\|\sqrt{\frac{-\pa_\tau \W_I}{ \W_I}}A  \oms\right\|_{L^2}^2+\nu \|A\sqrt{-\de_L} \Omega^*\|_{L^2}^2d\tau \leq  \myr{8}\ep^2\nu^{2/3};}\\ 
&
\|\po U^{1}(t)\|_{L_v^2}^2\leq 8\ep^2\nu^{2/3},\quad \forall t\in[0,T_*].\end{cases}\label{Hypotheses_sh}
\end{align}
The goal is to show that the estimates can be improved as long as the $\ep,\ \nu$ are chosen small enough.
\begin{proposition}[Short time]\label{pro:short}
    There exist thresholds $\epsilon_0, \nu_0 > 0$, depending only on the parameters specified in {\bf Assumption \ref{assum:b}} for the shear profile $B$, such that for $\epsilon \in (0, \epsilon_0), \nu \in (0, \nu_0)$
    the following conclusions hold:
\begin{align}
\begin{cases}&\displaystyle{\| A \oms(t)\|_{L^2}^2+\int_{0}^{t}\left\|\sqrt{\frac{-\pa_\tau \W_I}{ \W_I}}A  \oms\right\|_{L^2}^2+\nu \|A\sqrt{-\de_L} \Omega^*\|_{L^2}^2d\tau \leq   \myr{4}\ep^2\nu^{2/3};}\\ 
&\|\po U^{1}(t)\|_{L_v^2}^2\leq 4\ep^2\nu^{2/3},\quad \forall t\in[0,T_*].
\end{cases}
\end{align}

\end{proposition} A similar setup applies to the long time.

{\bf Long time: $t> \nu^{-1/6}$.}
We assume that  $s\geq 2$ and $[\nu^{-1/6},T_*]\subset[\nu^{-1/6},c_*\nu^{-1}] $ is the largest time interval on which the following \textbf{Hypotheses} hold:

\begin{subequations}\label{Hypotheses}
\begin{align}
\|  \zeta(t,|\pa_z|)\pn  \Omega^*(t)\|_{L^2}^2
+\nu\int_{\nu^{-1/6}}^t\|  \zeta(\tau,|\pa_z|)\sqrt{-\de_L}\pn\Omega^*\|_{L^2_{}}^2d\tau\leq& 16 \ep^2 \nu^{2/3},\label{Low_reg_1}\\
\label{Hyp_ed} 
\| A \oms(t)\|_{L^2}^2+\int_{\nu^{-1/6}}^{t}\left\|\sqrt{\frac{-\pa_\tau \mf M}{ \mf M}}A  \oms\right\|_{L^2}^2+\nu \|A\sqrt{-\de_L} \Omega^*\|_{L^2}^2d\tau \leq  &2C_1\ep^2\nu^{2/3};\\ 
 \label{Hyp_U1_0}
\|\po U^{1}(t)\|_{L_v^2}^2\leq 8\ep^2\nu^{2/3},\quad \forall t\in&[\nu^{-1/6},T_*].
\end{align} 
\end{subequations}
{\bf Determining the parameters: } First of all, we will choose the $K$ in the multipliers $\W_I, \W_\nu, \W_E$ along the proof. The $K$ will depend on the norm of $B$ ($\|B^{-1}\|_{L_{t,v}^\infty},\ \|B\|_{L_{t,v}^\infty}$ and $\|B'\|_{L_t^\infty H_v^s}$). Next, we determine the bootstrap constant $C_1$ along with the proof. The parameter will depend on the norm of the background shear flow $B$. 
Finally, we will choose the $\ep$ and then $\nu$ \myb{(HS: Do we need this:``and then $\nu$''? Can we choose the $\ep$ only?)} to be small enough depending on the constants chosen before.  


\begin{proposition}[Long time]\label{pro:long}
     There exist constants $\epsilon_0, \nu_0, c_* > 0$, { which depend only on the parameters specified in {\bf Assumption \ref{assum:b}}}, such that for $\epsilon \in (0, \epsilon_0),  \nu \in (0, \nu_0)$,
    the following improvements 
(\eqref{Con_Lw_rg_1}, \eqref{Con_ed}, \eqref{Con_U1_0}) to the hypotheses (\eqref{Low_reg_1}, \eqref{Hyp_ed}, \eqref{Hyp_U1_0})  hold for $T_* \leq c_* \nu^{-1}$:
\begin{subequations}\label{Con}
\begin{align}
\|  \zeta(t,|\pa_z|) \pn \Omega^*(t)\|_{L^2}^2
+\nu\int_{\nu^{-1/6}}^t\|  \zeta(\tau,|\pa_z|) \pn \sqrt{-\de_L}\Omega^*\|_{L^2_{}}^2d\tau\leq& 8 \ep^2 \nu^{2/3},\label{Con_Lw_rg_1} \\
\label{Con_ed} 
\| A \oms (t)\|_{L^2}^2+\int_{\nu^{-1/6}}^{t}\left\|\sqrt{\frac{-\pa_\tau \mf M}{ \mf M}}A  \oms \right\|_{L^2}^2+\nu \|A\sqrt{-\de_L} \Omega^* \|_{L^2}^2d\tau \leq& C_1\ep^2 \nu^{2/3};\\ 
\label{Con_U1_0}
\|\po U^{1}(t)\|_{L_v^2}^2\leq 4\ep^2\nu^{2/3}, \quad \forall t\in&[\nu^{-1/6},T_*].
\end{align}  
\end{subequations}
\end{proposition} 
\myb{HS: The extra $\exp\{\delta \nu^{1/3}(|k|^{2/3}+1)t\}$ is used in \eqref{F_1_est}.}


With the bootstrap assumptions \eqref{Hypotheses_sh} and \eqref{Hypotheses}, we can deduce the following bounds on the linear profile $F$.
\begin{proposition}
\label{linCK}
Let $F$ satisfy \eqref{Flin} and assume that  
\begin{align}
    \norm{A \sqrt{ - \frac{\p_t \mathcal{W}_I}{\mathcal{W}_I} } \Omega^*}_{L^2[0,T]L_{}^2} < \infty.
\end{align}
Then there exists $\nu_0 > 0$ such for $\nu \in (0, \nu_0)$, the following bound holds:
\begin{align}
\label{F:CK}
 & \norm{ \lb t \rb^{-1}A  |\p_z| (-\Delta_L)^{1/2} F }_{L_t^{\infty}[0, T]L_{}^2  }    +  \norm{A|\p_z|F}_{L^{\infty}[0,T]L_{}^2} \notag\\
 &+ \int_0^T\sum_{\iota \in \{ \nu , I, E \}} CK_{\iota}[|\p_z|F] 
 \, dt + \nu \norm{ BA|\p_z|(-\Delta_L)^{1/2} F }_{L^2[0,T]L_{}^2} 
 \lesssim
    \norm{A \sqrt{ - \frac{\p_t \mathcal{W}_I}{\mathcal{W}_I} } \Omega^*}_{L^2[0,T]L_{}^2}.
\end{align}

\end{proposition}
We defer the proof of Proposition \ref{linCK} to Section \ref{sec:linear}.

\section{Preliminaries}

\subsection{Basic Properties of $B(t,v)$ and Commutators}
Recall that $B(t,v) = \p_yb(t,y)$ and that $b(t,y)$ satisfies \eqref{heat}. We want to express the assumptions we made on $b(t,y)$ in Assumption \ref{assum:b} and express them in terms of the new coordinate system. These properties will follow from the following formula for $B(t,v)$:
\begin{align}
\label{B:rep}
B(t,v) = \int_\R  \frac{1}{4 \pi \nu t} \exp(\frac{|b^{-1}(t,w)|^2}{4\nu t} )  \p_y b(  b^{-1}(t, v) - b^{-1}(t,w) ) (1 / \p_y b(t,w)) \, dw.
\end{align}

As a consequence of \eqref{B:rep}, Assumption \ref{assum:b}, and properties of Gevrey spaces (see \cite{IJ20, J20})  we have the following:
\begin{lemma}
\label{lem:B:regularity}
There exists $\theta_0 \in (0,1)$ such for all $t \geq 0$, 
\begin{subequations}
    \begin{align}
   \label{B:smooth}
 \norm{B(t)}_{L_v^{\infty}}  +  \sup_{\xi \in \R} e^{\theta_0 \lb \xi \rb^{1/2}} |\widehat{\p_vB}(t,\xi)| &\leq  1/\theta_0, \\
  \label{B:bdd}
    B(t,v) &\geq \theta_0.
\end{align}
\end{subequations}

\end{lemma}

A key technical difficulty is dealing with commuting various Fourier multipliers past multiplication by very smooth, but not necessarily decaying, coefficients. In the following, we record various lemmas that we will frequently use throughout the paper.

\begin{lemma}
    \label{lem:C:bdd}
    Let $f \in H^s(\TT \times \R)$, $s \geq 2$,  and  $\mathfrak{C}$ satisfies \eqref{B:smooth}. Then we have that  
\begin{align}
        \norm{  \lb \p_z, \p_v \rb^s (\mathfrak{C} f)  }_{L_{}^2} \lesssim_{\mathfrak{C}}  \norm{\lb \p_z, \p_v \rb^sf}_{L_{}^2}. 
    \end{align}
Moreover, if $\mathfrak{C}$ additionally satisfies \eqref{B:bdd} 
then we have 
\begin{align}
\label{C:equiv}
   \norm{  \lb \p_z, \p_v \rb^s (\mathfrak{C} f)  }_{L_{}^2} \approx_{_{\mathfrak{C}}} \norm{\lb \p_z, \p_v \rb^sf}_{L_{}^2}. 
    \end{align}
\end{lemma}

\begin{lemma}
\label{lem:com:half}
Let $f \in H^s(\TT \times \R)$ and $\mathfrak{C}$ satisfy $ \p_v \mathfrak{C} \in  W^{\ceil{s} + 3,1}(\R)$. Then we have
\begin{subequations}
    \begin{align}
\label{halfhelm}
    \norm{ \lb k, \p_v \rb^s   [\mathfrak{C}, (-\Delta_L)^{-1/2}  ]  \pn f }_{L_{}^2} \lesssim   
\norm{\p_v \mathfrak{C}}_{W^{\ceil{s} + 3, 1  }} \norm{  \lb k, \p_v \rb^s (-\Delta_L)^{-1/2}   \pn f }_{L_{}^2}, \\
\label{helm}
\norm{ \lb k, \p_v \rb^s   [\mathfrak{C}, (-\Delta_L)^{-1}  ]  \pn f }_{L_{}^2} \lesssim   
\norm{\p_v \mathfrak{C}}_{W^{\ceil{s} + 3, 1  }}\norm{  \lb k, \p_v \rb^s (-\Delta_L)^{-1}   \pn f }_{L_{}^2}.
\end{align}
\end{subequations}

\end{lemma}
\begin{lemma}
    \label{lem:com:ghost}
    Let $f \in L^2(\TT \times \R)$, $\mathfrak{C}$ satisfy $\mathfrak{C} \in W^{1,\infty}(\R)$, and assume that 
    $m(t,k, \eta) \in \{ \mathcal{W}_{\nu}, \mathcal{W}_I, \mathcal{W}^\circ_I, \mathcal{W}_E \} $. Then we have
    \begin{align}
    \label{com:ghost}
        \norm{ [m, \mathfrak{C}] f }_{L_{}^2} \lesssim \frac{1}{K}\norm{m f}_{L_{}^2},
    \end{align}
    where $K$ is from the definition of the multipliers \eqref{multipliers}. 
\end{lemma}
We also have the following lemma that will be useful for controlling the difference between $B(t,v)$ and $B_0(v)$. 
\begin{lemma}
\label{lem:B:diff}
Assume that $f\in H^s(\TT \times \R)$. Then for any $\gamma \in (0,1)$, there exists a $c_{\gamma}$  such that  if $\nu t  \leq c_{\gamma}$ then
\begin{align}
\label{B:difference}
    \norm{\lb \p_z , \p_v \rb^s [ (\p_v^{\alpha}B(t, \cdot) - \p_v^{\alpha}B(0, \cdot))  f] }_{L_{}^{2}} \leq \gamma \norm{ \lb \p_z , \p_v \rb^s f}_{L_{}^2}, \quad \alpha \in \{0,1 \},
\end{align}
\end{lemma}
The proofs of Lemmas \ref{lem:C:bdd} - \ref{lem:B:diff} can be found in Appendix \ref{app:com}.

\subsection{Elliptic Lemmas}
Assume that $\varphi', X': \TT \times \R \to \R$ are $C^2$ functions satisfying
\begin{equation}
\label{flat:lap}
(\p_y^2 + \p_x^2) \varphi' = X', \quad \varphi' \to 0,\quad |y| \to \infty.  
\end{equation}
Taking the Fourier transform in $x$, we can write \eqref{flat:lap} as 
\begin{align}
\label{eq:helmholtz}
    (\p_y^2 - k^2) \varphi_k'(y) = X_k'(y) 
\end{align}
where 
\begin{align}
    \varphi_k'(y) &= \frac{1}{2\pi }\int_\TT \varphi'(x,y) e^{-ikx}\, dx,  \\
    X_k'(y) &= \frac{1}{2\pi }\int_\TT X'(x,y) e^{-ikx} \, dx.
\end{align}
For $k \in \mathbb{Z} \setminus \{ 0\}$ we can solve for $\varphi_k'$ via the relationship
\begin{align}
    \varphi_k'(y) = -\int_\R \frac{e^{-|k||y-z|}}{|k|} X_k'(z)\, dz.
\end{align}
A key technical difficulty not present in the Couette case is that in the new variables from \eqref{chg_of_v}, the operator $\Delta_0$ is no longer constant coefficient and does not have a convenient representation as a Fourier multiplier. The next lemma shows that the coefficients do not change the ellipticity properties from the linear case too drastically.
\begin{lemma}
\label{lem:ellip}
Assume that $\mathfrak{C}(t,v)$ satisfies \eqref{B:smooth} and that $X_k \in  C([0,T], H^2(\R))$. Let $\varphi_k \in  C([0,T], H^2( \R)) $ denote the solution to the equation 
\begin{align}
   \de_0\varphi_k = B_0^2 (\p_v - ikt)^2 \varphi_k  + B_0'(\p_v - ikt  )  \varphi_k -k^2 \varphi_k  =  X_k(t ,v ). 
\end{align}
Then for $k \in \mathbb{Z} \setminus \{ 0\}$, $M \in \{1 , (k^2 + (\xi - kt)^2 )^{-1/2} \}$ and $s \geq 0$, we have that
\begin{align}\label{ell_est}
   \lf\|\lb k, \xi \rb^s  M(t, k, \xi )\pn (\Delta_L (\mf{C}  \varphi_k) )^{\wedge}( t, \xi )\rg\|_{L_\xi^2(\R)} \lesssim \lf\|\lb k, \xi \rb^s  M(t, k, \xi)\pn  \widehat{X_k}( t, \xi) \rg\|_{L_\xi^2(\R)}  .
\end{align}

\end{lemma}



\begin{proof}
We will closely follow the proof of Lemma 4.5 in \cite{IJ23}. 
The proof when $M(t , k, \xi) = (k^2 + (\xi - kt)^2)^{-1/2}$ is harder so we will only provide the details in that case. 
Define
\begin{align}
    \varphi'(t,x ,y) := \varphi (t, x - tb( y), b(y)), \quad X'(t ,x ,y) := X (t, x - tb(y), b(y)),
\end{align}
which implies
\begin{align}
   e^{-itk b(y)} \varphi_k(t, b(y)) = 
    - \int_{\R} \frac{e^{-|k||y- y'|}}{|k|} e^{-itkb(y') } X_k(t, b(y')) dy'.
\end{align}
Letting $(v,w) = (b(y), b(y'))$, and defining 
\begin{align}
\label{green:helmholtz}
    \mathcal{G}_k(b(y), b(y')) := \frac{e^{-|k||y- y'|}}{|k|},
\end{align}
 we arrive at 
\begin{align}
\label{ell1}
    \varphi_k(t,v) = -\int_{\R} X_k(t,w) e^{ik t (v-w)} \mathcal{G}_k(v,w) (1 / B(w)) dw.
    \end{align}
Multiplying \eqref{ell1} by $\mf{C}$ and using the decomposition \eqref{Gdecomp} for $\G(v,w)$ from Lemma \ref{lem:G:bdds} gives
\begin{align}
\label{ell2}
     \mf{C}(v)\varphi_k(t,v)
     =& -\int_{\R} X_k(t,w) e^{ik t (v-w)}  \mathcal{G}_{1,k }(v -w)\mf{C}(w)\chi(w) (1 / B(w)) dw  \\
     &-\int_{\R} X_k(t,w) e^{ik t (v-w)} (\mf{C}(v) - \mf{C}(w) ) \mathcal{G}_{1,k}(v -w)\chi(w) (1 / B(w)) dw
     \notag \\
     &-\int_{\R} X_k(t,w) e^{ik t (v-w)}\mathcal{G}_{2,k}(v,w) \mathfrak{C}(v)    (1 / B(w)) dw .  \notag
\end{align}
Taking the Fourier transform with respect to $v$ yields
\begin{align}
\label{ell3}
   \widehat{\mathfrak{C}\varphi}(t, k ,\xi ) &=    -\widehat{\mathcal{G}_1}( \xi - kt  )  \widehat{\mathcal{C}X}(t, k ,\xi)  - \int_\R \widehat{X}(t, k, \eta) H(\xi - kt, kt - \eta ) \,  d \eta
\end{align}
where $\mathcal{C}(w) := \mathfrak{C}(w) \chi(w) (1 / B(w)) $ and 
\begin{align}
    H(\alpha, \beta) &:= \int_{\R^2} (\mf{C}(v) - \mf{C}(w) ) \mathcal{G}_1(v -w)\chi(w) (1 / B(w)) e^{-i \alpha v - i \beta w}  dv dw \\
    &+ \int_{\R^2} \mathcal{G}_2(v,w) \mathfrak{C}(v)    (1 / B(w)) e^{-i \alpha v - i \beta w}  dv dw. \notag
\end{align}
By Lemma \ref{lem:G:bdds} and Lemma \ref{lem:com:half}, we have 
\begin{align}
    |H(\alpha, \beta)| \lesssim \frac{1}{(k^2 + \alpha^2) \lb  \alpha + \beta \rb^{\ceil{s}  +2  } }.
\end{align}
Recalling that we are assuming $M(k, \xi, t) = (k^2 + (\xi - kt)^2)^{-1/2}$, \eqref{ell3} implies 
\begin{align}
\label{ell4}
    |\lb k, \xi \rb^s  (k^2 + |\xi - kt|^2)^{1/2}   \widehat{\mathfrak{C}\varphi}(t, k ,\xi )| \lesssim  \frac{\lb k, \xi \rb^s   |\widehat{\mathcal{C}X}(t, k ,\xi)|}{(k^2 + |\xi -kt|^2)^{1/2}} + \int_{\R} \frac{|\widehat{X}(t ,k, \eta)| \lb k, \xi \rb^s   }{   (k^2 + |\xi -kt|^2)^{1/2}  \lb \xi - \eta \rb^{ \ceil{s} +2 }   }  d \eta.
\end{align}
Young's convolution inequality yields
\begin{align}
  \norm{ \lb k, \xi \rb^s  (k^2 + |\xi - kt|^2)^{1/2} \widehat{\mathfrak{C}\varphi}(t, k ,\xi )}_{L_\xi^2} &\lesssim \norm{  \lb k, \xi \rb^s  (k^2 + |\xi - kt|^2)^{-1/2}  \widehat{\mathcal{C}X}(t, k ,\xi )   }_{L_\xi^2}  \\
  &+\norm{  \lb k, \xi \rb^s  (k^2 + |\xi - kt|^2)^{-1/2} \widehat{X} (t, k ,\xi )   }_{L_\xi^2}. \notag   
\end{align}
To finish the proof, we  commute $ \lb k, \xi \rb^s $ and $(k^2 + |\xi - kt|^2)^{-1/2}$ past $\mathcal{C}$ and use
 Lemma \ref{lem:com:half} and the Fractional Leibniz rule
 \begin{align}
     \norm{|\p_v|^s( f g  )}_{L_v^2} \lesssim \norm{f}_{L_v^{\infty}} \norm{|\p_v|^sg}_{L_v^2} + \norm{|\p_v|^sf}_{L_v^{2}} \norm{g}_{L_v^{\infty}},
 \end{align}
from \cite{Li19kato}. This completes the proof.
\end{proof}
Recall from \eqref{profileaux} that we invert $\Delta_0$ and $\Delta_t$. The next lemma shows that, for time sufficiently small relative to $\nu^{-1}$, we can replace $\Delta_0$ by $\Delta_t$ in Lemma \ref{lem:ellip}.
\begin{lemma}
    \label{ellip:perturb}
Assume that the coefficient $\mathfrak{C}(t,v)$ satisfies the regularity assumption \eqref{B:smooth} and that
 $X_k \in  C([0,T], H^2(\R))$. 
Given any $\gamma_* \in (0,1)$, there exists $c_*\in (0,1)$,  such that for $t \nu \leq c_*$, $k \in \mathbb{Z} \setminus \{ 0 \}$ 
and for  $M \in \{1 , (k^2 + (\xi - kt)^2 )^{-1/2} \}$ we have the following
\begin{align}
\label{Delta_diff}
\hspace{-0.7 cm}
    \lf\|\lb k, \xi \rb^s  M(t, k, \xi ) (\Delta_L (\mf{C} (\Delta_t^{-1} - \Delta_0^{-1}) X_k )^{\wedge}(t, \xi )\rg\|_{L_\xi^2(\R)}\leq \gamma_* \lf\|\lb k, \xi \rb^s  M(t, k, \xi) \widehat{X_k}(t, \xi) \rg\|_{L_\xi^2(\R)}.
\end{align}
Moreover, 
\begin{align}
\label{ell_est_t}
    \lf\|\lb k, \xi \rb^s  M(t, k, \xi ) (\Delta_L (\mf{C} \Delta_t^{-1}  X_k )^{\wedge}(t , \xi)\rg\|_{L_\xi^2(\R)}\lesssim   \lf\|\lb k, \xi \rb^s  M(t, k, \xi) \widehat{X_k}(t, \xi) \rg\|_{L_\xi^2(\R)}.
\end{align}
\end{lemma}
\begin{proof}
The estimate \eqref{ell_est_t} is an immediate consequence of \eqref{Delta_diff}, \eqref{ell_est}, and the triangle inequality, so we focus on the proof of \eqref{Delta_diff}. We begin by writing 
    \begin{align}
    \label{t:decomp}
        \Delta_{t} &=   \Delta_0 +  (B^2-B_0^2)(\p_v - ikt)^2  + (B' -B_0')(\p_v - ikt) \\
        &= (I -R)\Delta_0, \notag \\
        R &:=  (B_0^2-B^2)(\p_v - ikt)^2\Delta_0^{-1}  - (B'_0- B')(\p_v - ikt)\Delta_0^{-1}. \notag
    \end{align}
We focus on the term $(B_0^2 - B^2)(\p_v - itk)^2 \Delta_0^{-1}$ as the other term is simpler. We write 
\begin{align}
     \left( M (B_0^2-B^2)(\p_v - ikt)^2\Delta_0^{-1} X_k \right) 
    &=  (B_0^2-B^2) \left( M (\p_v - ikt)^2\Delta_0^{-1} X_k \right)\\
    &+   \left( [ M  , B_0^2-B^2 ] (\p_v - ikt)^2\Delta_0^{-1} X_k \right).
\end{align}
Then, given $\gamma_\diamond \in (0,1)$, Lemma \ref{lem:com:half}, Lemma \ref{lem:ellip}, \eqref{B:difference}, and \eqref{B:rep} imply there exists a $c_\diamond \in (0,1)$ such that for $t\nu \leq c_\diamond$,  $k \in \mathbb{Z} \setminus \{ 0 \}$, and $X_k \in H^s(\R )$  
\begin{align}
\label{Rcontract}
   \norm{ \lb  k, \p_v\rb^s  M R  X_k }_{L^2(\R)}\leq \gamma_\diamond\norm{\lb  k, \p_v \rb^s M X_k}_{L^2( \R)}.
\end{align}

As a consequence of \eqref{Rcontract},
the operator $ I - R$ is invertible on $H^s(\R)$. We can express the inverse with a Neumann series as 
\begin{align}
\label{inverse:neumann}
    (I - R)^{-1} =   I +  \sum_{n\geq 1} R^n.
\end{align}
Using \eqref{t:decomp} and \eqref{inverse:neumann} we can write
\begin{align}
\label{expansion}
    \Delta_t^{-1} = \Delta_0^{-1} + \Delta_0^{-1} \sum_{n\geq 1} R^n.
\end{align}
Iterating \eqref{Rcontract} yields
\begin{equation}
\label{induction}
    \norm{\lb k, \p_v \rb^s M R^n X_k }_{L^2(\R)} \leq \gamma_\diamond^{n } \norm{\lb  k, \p_v \rb^s M  X_k }_{L^2( \R )}. 
\end{equation}
Finally, we combine Lemma \ref{lem:ellip}, the estimate  \eqref{induction}, and the triangle inequality to obtain that 
\begin{align}
\label{ellip:sum}
    &\lf\|\lb k, \xi \rb^s  M(k, \xi,t ) (\Delta_L (\mf{C} (\Delta_t^{-1} - \Delta_0^{-1}) X_k )^{\wedge}(\xi, t)\rg\|_{L_\xi^2(\R)}\\
    &\leq \sum_{n\geq 1} \lf\|\lb k, \xi \rb^s  M(k, \xi,t ) (\Delta_L (\mf{C} \Delta_0^{-1} R^n X_k )^{\wedge}(\xi, t)\rg\|_{L_\xi^2(\R)} \leq C\gamma_\diamond \norm{\lb  k, \p_v \rb^s M  X_k }_{L^2( \R )}. \notag
\end{align}
We now choose a $c_* \in (0, c_\diamond)$ so that   $\gamma_\diamond \in (0, \gamma_*)$ is small enough to absorb the implicit constant in \eqref{ellip:sum}  which completes the proof. 
\end{proof}

\subsection{Trilinear Lemma}
 The following trilinear estimates are widely applied in the analysis.
 \begin{lemma}\label{lem:Tri_est}
    Considering three functions $\mathcal F, \mathcal G,\mathcal H\in H^s$, we have the following estimate
    \begin{align}
    &\int_{t\in \mathcal{I}}\sum_{k,\ell}\iint (\pn e^{\delta \nu^{1/3}(|k|^{2/3}+1)t}+\po )(\lan\xi,\ell\ran ^s+\lan k-\ell,\eta-\xi\ran^s) |\ell ||\wh{(B\de_t^{-1}\mathcal F_\nq)}(\ell,\xi) |\n\\
    &\hspace{3cm}\times|(\eta-\xi-(k-\ell)t)  \mathcal G(\eta-\xi,k-\ell)| \  |A\mathcal H(k,\eta)| d\eta d\xi dt\n\\
    & \lesssim \lf(\lf\|A\sqrt{\frac{-\pa_t\W_I}{\W_I}}\mathcal F_\nq\rg\|_{L^2(\mathcal{I};L^2)}^2 + \lf\|A\sqrt{\frac{-\pa_t\W_E}{\W_E}}\mathcal H\rg\|_{L^2(\mathcal{I};L^2)}^2\rg)\n\\
    &\hspace{2cm}\times\sup_{t\in \mathcal{I}}\lf(t e^{-\delta\nu^{1/3}t}\| e^{\delta\nu^{1/3}(|\pa_z|^{2/3}+1)t}\pn\mathcal G\|_{L^\infty(\mathcal{I};\myr{H^{2}})}+\|\mathbb{P}_{\mathbbm{o}}\mathcal G\|_{L^\infty(\mathcal{I};\myr{H^{2}})}\rg)\n\\
    &\quad+\int_{t\in \mathcal{I}}\lan t\ran^{-1}\lf\|A\sqrt{\frac{-\pa_t\W_I}{\W_I}}\mathcal F_\nq\rg\|_{L^2}\|A\sqrt{-\de_L}\mathcal G\|_{L^2}\|A\mathcal H\|_{L^2}dt. \label{Tri_est}
\end{align}
Here, $\mathcal{I}$ is an interval, e.g., $[0,T_*]$ or $[\nu^{-1/6},T_*]$. 
\end{lemma}
\begin{proof}
    The proof of the lemma is postponed to Appendix \ref{sec:multiplier}. 
\end{proof}

\section{Energy Structures of the Auxiliary Equation}
\label{sec:energy}

Now, we compute the time evolution of the energy 
\begin{align}
\frac{1}{2}\frac{d}{dt}\|A \oms\|_{L_{}^2}^2 
=&-\lf\|A\sqrt{\frac{-\pa_t \mf M}{\mf M}}\oms\rg\|_{L_{}^2}^2+\nu\int A\oms\ A\tilde\de_t \oms \, dV \notag\\
& -\int  A\oms A ( B \p_v \po U^1 \p_z \Omega^*) \, dV  -\int A\oms \ A\lf(B \na_L^\perp \Delta_t^{-1} \pn \Omega \cdot \na_L\oms \rg)\, dV  \notag \\
&-\int A\oms A ( B \p_v \po U^1 \p_z F)\, dV  -\int A\oms \ A\lf(B \na_L^\perp \Delta_t^{-1} \pn \Omega \cdot \na_L F \rg)\, dV \notag \\
&-\int A\oms \ A(B_0'\pa_z(\Delta_0^{-1}  - \Delta_t^{-1})\Omega)) dV-\int A\oms\ A((B_0'-B')\pa_z\Delta_t^{-1} \Omega)\, dV \notag \\
&-\nu\int A\oms\ A(\Delta_0 -\tilde{\Delta}_t)F\,  dV \notag \\
=:&-\lf\|A\sqrt{\frac{-\pa_t \mf M}{\mf M}}\oms\rg\|_{L_{}^2}^2 +D- \mathfrak{NL}_{\mathfrak{0a}} - \mathfrak{NL}_{\mathfrak{a}}- \mathfrak{NL}_{\mathfrak{0l}} - \mathfrak{NL}_{\mathfrak{l}}-\mathfrak{L}_{1}-\mathfrak{L}_{2}-\mathfrak{L}_{3}.
\label{Energy_Ev}
\end{align}
Here, `$\mathfrak l$' stands for `linear,'  `$\mathfrak{0l}$' stands for linear interaction with zero mode,  `$\mathfrak a$' stands for `auxiliary,'  and `$\mathfrak{0a} $' stands for auxiliary interaction with zero mode. 
In the remaining parts of the section, we estimate the terms that appear in \eqref{Energy_Ev}.

\subsection{The Estimate of $D $}
\label{sec:D}
We define $D$ as 
\begin{align}
D:= \nu \int A \Omega^* A \tilde{\Delta}_t \Omega^* \, dV.
\end{align}
To estimate the diffusion term $D$, we need to derive two sets of bounds. The first type of bound characterizes the behavior on the top regularity level, i.e., $H^s$ (Lemma \ref{lem:Diffusion}). Thanks to the varying coefficient $B$, the high regularity bound cannot close on its own. The main obstacle arises from the last term in \eqref{D_est}. So, we derive a separate enhanced dissipation estimate of the solutions on the low regularity level (Lemma \ref{lem:Con_Low_reg_1}) to overcome this obstacle. 
\begin{lemma}\label{lem:Diffusion}
{\bf a) Short time:} For $T_{*} \in (0, \nu^{-1/6}]$, under the bootstrap assumption of Proposition \ref{pro:short}, the following estimate holds as long as the parameters $\nu^{-1}, \, K$ are chosen large enough depending on the norms of the ambient shear profile derivative $B$ (as specified in Assumption \ref{assum:b} and \eqref{assump:b_t}), 
        \begin{align}\n
          D \leq&-\frac{1}{2}\nu \|A\sqrt{-\de_L}\oms\|_{L^2}^2+C\nu\|B\pa_vB\|_{H^s}^2\|A\oms_\nq\|_{L^2}^2+C\nu(\|B\|_{L^\infty}^2+\| B\pa_vB \|_{H^s}^2)\|A\oms_{\mathbbm{o}}\|_{L^2}^2\\
    &+C\nu(\|B\|_{L^\infty}+\|B\pa_v B\|_{H^{ s}})^{2s}\|e^{\delta\nu^{1/3}(|\pa_z|^{2/3}+1)t}\sqrt{-\de_L}\oms_\nq\|_{L^2}^2.\label{D_est}
        \end{align}
        Hence, 
\begin{align}
\label{D_est_int_sh}
    \int_{0}^{T_*} D \,dt\leq -\frac{\nu}{2} \int_{0}^{T_*}\|A\sqrt{-\de_L}\oms\|_{L_{}^2}^2dt +C(\|B\|_{L_{t,v}^\infty}+\| B'\|_{L^\infty_t H_v^s})^{2s}  \ep^2\nu^{2/3}.
\end{align}

\noindent
{\bf b) Long time:} For $T_*\in [\nu^{-1/6}, c_*\nu^{-1}]$, assume the bootstrap hypotheses of Proposition \ref{pro:long}. Assuming that the $\nu^{-1}, \, K$ are chosen large enough depending on the norms of the shear profile derivative $B$ (as specified in Assumption \ref{assum:b} and \eqref{assump:b_t}), the dissipation $D$-term in \eqref{Energy_Ev}  can be estimated as in \eqref{D_est} (with the time domain $[0,T_*]$ replaced by $[\nu^{-1/6}, T_*]$). Moreover,
\begin{align}\label{D_est_int_l}
    \int_{0}^{T_*} D\,  dt\leq& -\frac{\nu}{2} \int_{0}^{T_*}\|A\sqrt{-\de_L}\oms\|_{L_{}^2}^2dt +Cc_* (\|B\|_{L_{t,v}^\infty}^2+\|B'\|_{L^\infty_t H_v^s}^2 )C_1\ep^2\nu^{2/3}\\
    &+C(\|B\|_{L_{t,v}^\infty}+\| B'\|_{L_t^\infty H_v^s})^{2s}  \ep^2\nu^{2/3}.
\end{align}
 Hence, as a result of the bootstrap assumption \eqref{Hypotheses_sh} (\eqref{Hypotheses}, respectively), the following estimate holds for $\mathcal{I} = [0, T_*], [\nu^{-1/6}, T_*]$:
    \begin{align}
        \int_{ \mathcal{I}} D \, dt\leq -\frac{\nu}{2}\int_{ \mathcal{I}}\|A\sqrt{-\de_L}\oms\|_{L^2}^2 dt+C\ep^2\nu^{2/3}.
    \end{align}
\end{lemma}

\begin{proof}
We decompose the proof into two parts, i.e., the nonzero mode estimates and the zero mode estimate. For the nonzero mode, we apply the fact that $\pa_v-t\pa_z$ commutes with the multiplier $A$ and expand it as follows,
\begin{align}
D_\nq=&\nu\int A\oms_\nq \ A\tilde\de_t \oms_\nq dV=-\nu\|A\pa_z\oms_\nq\|_{L^2}^2+\nu\int A\oms_\nq \ A\lf((B^2(\pa_v-t\pa_z)^2 \oms_\nq\rg) dzdv\\
=&-\nu\|A\pa_z\oms_\nq \|_{L^2}^2+\nu\iint A \oms_\nq\   (\pa_v-t \pa_z)  A  ({B^2}\ (\pa_v-t \pa_z) \oms_\nq )dzdv\\
&-\nu\iint A \oms_\nq\     A(  \pa_v(B^2)\ (\pa_v-t \pa_z) \oms_\nq) dzdv\\
=&-\nu\|A\pa_z\oms_\nq \|_{L^2}^2-\nu\| {B}  A  (\pa_v-t \pa_z) \oms_\nq \|_{L^2}^2 -\nu\iint A \oms_\nq\     A(  \pa_v(B^2)\ (\pa_v-t \pa_z) \oms_\nq) dzdv\\
&-\nu\iint A   (\pa_v-t \pa_z)\oms_\nq\ [A, B^2]\lf(  (\pa_v-t \pa_z)   \oms_\nq\rg) dz dv\\
=:&-\nu \|A\pa_z\oms_\nq \|_{L^2}^2-\nu\| {B}  A  (\pa_v-t \pa_z) \oms_\nq \|_{L^2}^2+T_1+T_2.\label{D_T}
\end{align}
Now we estimate the $T_1$-term in \eqref{D_T} with the product estimate \eqref{A_product_rule_Hs},
\begin{align}
   T_{1}     \leq &     C\nu\|A\oms_\nq\|_{L^2} \|B\pa_vB\|_{H^s}\|A(\pa_v-t\pa_z) \oms_\nq\|_{L^2}\\
   \leq &-\frac{1}{4} \nu\|A(\pa_v-t\pa_z)\oms_\nq\|_{L^2}^2+C\nu\|{B\pa_vB}\|_{H^s}^2\|A\oms_\nq\|_{L^2}^2.
\end{align}
If the $\nu$ is chosen small enough, the second term can be absorbed by the enhanced dissipation.

To estimate $T_2$, we further decompose it as follows
\begin{align}
    T_2=&\nu\iint A (\pa_v-t \pa_z)\oms_\nq \mf M[\lan \pa_z,\pa_v\ran^s,B^2]( (\pa_v-t \pa_z)  e^{\delta\nu^{1/3}(|\pa_z|^{2/3}+1)t} \oms_\nq ) dz dv\\
    &+\nu\iint A (\pa_v-t \pa_z)\oms_\nq [\mf M,B^2]( \lan \pa_z,\pa_v\ran^s(\pa_v-t \pa_z)  e^{\delta\nu^{1/3}(|\pa_z|^{2/3}+1)t} \oms_\nq ) dz dv\\
    =:&T_{21}+T_{22}.
\end{align}
To treat the first term, we invoked the following Kato-Ponce inequality (see \cite{Li19kato}): 
\begin{align}\label{KP}
    \|\lan\pa_v\ran^s (fg) - f\lan\pa_v\ran^sg \|_{L^2(\rr)}\lesssim \|\lan\pa_v\ran^{s-1}\pa_v f \|_{L^4(\rr)}\|g\|_{L^4(\rr)}+\|\pa_v f \|_{L^\infty(\rr)}\|\lan \pa_v\ran^{s-1}g\|_{L^2(\rr)}.
\end{align}
Hence, for the $T_{21}$-term, we estimate it as follows
\begin{align*}
    |T_{21}|\leq& C\nu\|A (\pa_v-t \pa_z)\oms_\nq\|_{L^2} \|[\lan \pa_z,\pa_v\ran^s,B^2]( (\pa_v-t \pa_z) e^{\delta\nu^{1/3}(|\pa_z|^{2/3}+1)t}  \oms_\nq ) \|_{L^2}\\
    \leq& C(\|B\|_{L^\infty}+\|B'\|_{H^{ s}})\nu\|A (\pa_v-t \pa_z)\oms_\nq\|_{L^2}\|e^{\delta(\nu^{1/3}|\pa_z|^{2/3}+1)t}(\pa_v-t \pa_z) \oms_\nq \|_{H^{s-1}}\\
    \leq& \frac{1}{4}\nu\|A(\pa_v-t \pa_z)\oms_\nq\|_{L^2}^2+C\nu(\|B\|_{L^\infty}+\|B\pa_v B\|_{H^{ s}})^{2s}\|e^{\delta\nu^{1/3}(|\pa_z|^{2/3}+1)t}\sqrt{-\de_L}\oms_\nq\|_{L^2}^2.
\end{align*}
To estimate the $T_{22}$-term, we observe that:
\begin{align*}
a)& \quad t\leq \nu^{-1/6}:    [\mf M, B^2]\oms_\nq=[\W_I, B^2]\oms_\nq;\\
b)& \quad t\geq \nu^{-1/6}:    [\mf M, B^2]\oms_\nq=[\W_I\W_\nu\W_E, B^2]\oms_\nq\\
&\hspace{4.05cm}=\lf(\W_I\W_\nu[\W_E, B^2] +\W_I[\W_\nu, B^2]\W_E+[\W_I, B^2]\W_\nu\W_E\rg)\oms_\nq.
\end{align*}
Thanks to the commutator estimate \eqref{com:ghost}, we have that 
\begin{align}
|T_{22}|\leq \frac{C(\|B\|_{L^\infty}^2+\|B\pa_v B\|_{H^s}^2)}{K}\|A(\pa_v-t\pa_z) \oms_\nq\|_{L^2}^2.
\end{align}
As long as $K$ is large compared to the norms of $B$, the term can be absorbed by the diffusion. This concludes the treatment for $D_\nq.$

 For the diffusion in the zeroth mode, we estimate it as in the previous argument:\begin{align}
 {D}_{\mathbbm{o}} = & \nu\int A\oms_{\mathbbm{o}}  A(B^2\pa_v^2\oms_{\mathbbm{o}})dv\\
=&\nu\iint A \oms_{\mathbbm{o}}\   \pa_v  A  ({B^2}\ \pa_v \oms_{\mathbbm{o}}) dzdv -\nu\iint A \oms_{\mathbbm{o}}\     A(  \pa_v(B^2)\ \pa_v \oms_{\mathbbm{o}}) dzdv\\
=&-\nu\| {B}  A  \pa_v \oms_{\mathbbm{o}} \|_{L^2}^2 -\nu\iint A \oms_{\mathbbm{o}}\     A(  \pa_v(B^2)\ \pa_v \oms_{\mathbbm{o}}) dzdv -\nu\iint A   \pa_v\oms_{\mathbbm{o}}\ [A, B^2]\lf(   \pa_v   \oms_{\mathbbm{o}}\rg) dz dv\\
=&-\nu\| {B}  A  \pa_v \oms_{\mathbbm{o}} \|_{L^2}^2+T_3+T_4.
\end{align}
For the $T_3$-term, we can estimate it as $T_1$. For the $T_4$-term, we distinguish the $t\leq \nu^{-1/6}$ and $t\geq \nu^{-1/6}$ cases. For $t\geq \nu^{-1/6}$, we invoke the Kato-Ponce inequality to obtain that 
\begin{align*}
    T_{4}\mathbbm{1}_{t\geq \nu^{-1/6}}\leq& 
    \nu \|A\pa_v \oms_{\mathbbm{o}}\|_{L^2}(\|B\|_{L^\infty}+\|B\pa_vB\|_{H^{s-1}})\|\lan  \pa_v\ran^{s-1}\pa_v \oms_{\mathbbm{o}}\|_{L^2}\\
    \leq& -\frac{\nu}{4}\|BA \pa_v \oms_{\mathbbm{o}}\|_{L^2}^2+C\nu \|\lan  \pa_v\ran^s \oms_{\mathbbm{o}}\|_{L^2}^2(\|B\|_{L^\infty}^2+\|B\pa_vB\|_{H^{s}}^2).
\end{align*}
For $t\leq \nu^{-1/6}$, we invoke the Kato-Ponce inequality together with the commutator estimate \eqref{com:ghost} to obtain that 
\begin{align*}
    T_{4}\mathbbm{1}_{t\leq \nu^{-1/6}}\leq& 
    \lf(C\nu\|B^{-1}\|_{L^\infty}^2\frac{\|B\|_{L^\infty}^2+\|B\pa_vB\|_{H^{s}}^2}{K}+\frac{\nu}{8}\rg)\|BA \pa_v \oms_{\mathbbm{o}}\|_{L^2}^2\\
    &+C\nu \|\lan  \pa_v\ran^s \oms_{\mathbbm{o}}\|_{L^2}^2(\|B\|_{L^\infty}^2+\|B\pa_vB\|_{H^{s}}^2).
\end{align*}
If $K$ is chosen large enough, the first term can be absorbed by the diffusion and this is consistent with \eqref{D_est}. Hence the proof is concluded. \myb{(Check!)}
\end{proof}
As a variant of the proof above, we are able to prove the following low energy estimate, which is only used for the proof of the conclusion \eqref{Con_Lw_rg_1} in Proposition \ref{pro:long}. The motivation to introduce the low energy estimate is to control the third term on the right-hand side of the estimate \eqref{D_est}.

\begin{lemma} \label{lem:Con_Low_reg_1} Consider $T_*\in (\nu^{-1/6}, c_*\nu^{-1}]$. Under the bootstrap assumption of Proposition \ref{pro:long}, the following estimate holds if $\ep$ is small enough compared to the bootstrap constant $C_1$ and $\|B^{-1}\|_{L^\infty}, \|B'\|_{H^s}$,
\begin{align}\n
\|e^{\delta\nu^{1/3}(|\p_z|^{2/3}+1)t}&\oms_\nq\|_{L_t^\infty[\nu^{-1/6}, T_*] L_{}^2}^2+\nu\int_{\nu^{-1/6}}^{T_*}\|e^{\delta\nu^{1/3}(|\pa_z|^{2/3}+1)\tau}B(\pa_v-\tau\pa_z)\Omega_\nq^*\|_{L_{}^2}^2d\tau\\
\leq  &  {\|\zeta(\nu^{-1/6},|\pa_z|)\Omega(\nu^{-1/6})\|_{L^2}^2}+ \frac{1}{2}\ep^2\nu^{2/3}.
\label{L_2:non_0_md}
\end{align}
Hence, thanks to Proposition \ref{pro:short}, the conclusion \eqref{Con_Lw_rg_1} holds. 
\end{lemma}
 \begin{proof}

To simplify the notation, we define a new multiplier only in this proof.
\begin{align}
    \wt {\mf M}f=(e^{\delta\nu^{1/3}(|\pa_z|^{2/3}+1)t}\pn+\po)\W_\nu f.
\end{align}
The time derivative of the $L^2_{}$-norm of $\oms$ reads as follows (following the argument of \eqref{Energy_Ev}),
\begin{align}\label{ddtL2:oms_nq}
\frac{1}{2}     \frac{d}{dt}\| \wt {\mf M}\oms(t)\|_{L_{}^2}^2 
=&\nu\int  \wt {\mf M}B^2(\pa_v-t\pa_z)^2\Omega^* \wt {\mf M} \Omega^*dzdv-\nu \|\pa_z \wt{\mf M}\oms\|_{L^2}-2\lf\|\sqrt{\frac{-\pa_t \W_\nu}{\W_\nu}} \wt {\mf M}\oms\rg\|_{L^2}^2\\
&+2\delta\nu^{1/3}\lf\|(|\pa_z|^{2/3}+1) \wt {\mf M} \oms_{\neq}\rg\|_{L^2}^2+ \mathrm{Remainder}(t),
\end{align}
where the $\mathrm{Remainder}(t)$ satisfies the following bound: 
\begin{align}
    \int_{\nu^{-1/6}}^{T_*} \mathrm{Remainder}(t) dt   \lesssim   \epsilon\nu \int_{\nu^{-1/6}}^{T_*} \|A\sqrt{-\de_L}\oms\|_{L^2}^2  dt  +   \epsilon \int_{\nu^{-1/6}}^{T_*} CK[\Omega^*] dt +  \epsilon^3 \nu^{2/3}.
\end{align}
This follows from the estimates in the following subsections in Section \ref{sec:energy}. We will only focus on the estimate of the first term. The treatment of the remainder terms is similar to the latter estimates in the section, and hence, we omit it for the sake of brevity. We estimate the first term (with $k\neq 0$) in \eqref{ddtL2:oms_nq} as follows,
\begin{align*}
   T:= &\nu \int \wt {\mf M}  B^2(\pa_v-t\pa_z)^2\Omega_\nq^*  \wt {\mf M}\Omega_\nq^*dzdv\\
    =&\nu   \int [\wt {\mf M},  B^2(\pa_v-t\pa_z)^2]\Omega_\nq^*  \wt {\mf M}\Omega_\nq^*dzdv-\nu  \int  B^2\lf((\pa_v-t\pa_z)  \wt {\mf M}\Omega_\nq^*\rg) \lf((\pa_v-t\pa_z) \wt {\mf M}\Omega_\nq^*\rg) dzdv\\
    &-\nu \int 2B\pa_v B(\pa_v-t\pa_z)\wt {\mf M} \Omega_\nq^* \wt {\mf M}\Omega_\nq^*dzdv.
    \end{align*}
Now we repeat the argument as in the proof of Lemma \ref{lem:Diffusion} with the observation that the key difficult term $T_{22}\equiv 0$. As a result, if we choose the $K$ in \eqref{W_nu} large enough depending on the norms of the shear profile derivative $B$  (as specified in Assumption \ref{assum:b} and \eqref{assump:b_t}), then there exists a constant $C=C(B,B^{-1}, B')>1$ such that the following estimate holds,
\begin{align*}
T\leq &-\frac{3\nu}{4} \|B\wt {\mf M} (\pa_v-t\pa_z)\Omega_\nq^* \|_{L^2}^2 +C\nu(1+\|\pa_v B\|_{L^\infty}^2)\|\wt {\mf M}\Omega_\nq^*\|_{L^2}^2. 
\end{align*}
As a result, we end up with the following relation
\begin{align}
    \frac{d}{dt}&\|\wt {\mf M} \oms(t)\|_{L_{}^2}^2\\
    \leq &- {\nu} \|\wt {\mf M}  \pa_z \oms\|_{L^2}^2-\frac{\nu}{2}\|B\wt {\mf M} (\pa_v-t\pa_z)\oms\|_{L^2}^2-\lf(\mathfrak{c}-8\delta\rg)\nu^{1/3}\|(|\pa_z|^{1/3}+1)\wt {\mf M}\oms_\nq\|_{L^2}^2\\
    &+(C(1+\|\pa_v B\|_{L^\infty})\nu^{1/3})^2\nu^{1/3}\|\wt {\mf M} \oms_\nq\|_{L^2}^2 +C\ep\lf(\lf\|A\sqrt{\frac{-\pa_t M}{M}}\oms_\nq\rg\|_{L^2}^2+\nu\|A\sqrt{-\de_L}\oms_\nq\|_{L^2}^2\rg).
\end{align}
By invoking the relation \eqref{M_property_ED} and the bootstrap assumption, we have that 
\begin{align}
&\|e^{\delta\nu^{1/3}(|\pa_z|^{2/3}+1)t}\oms_\nq(t) \|_{L_{}^2}^2+\nu\|e^{\delta\nu^{1/3}(|\pa_z|^{2/3}+1)t}B(\pa_v-t\pa_z)\oms_\nq\|_{L_t^2L_{}^2}^2\\
&\leq \|\wt {\mf M}\oms(t) \|_{L_{}^2}^2+\nu \|B\wt {\mf M}(\pa_v-t\pa_z)\oms\|_{L_t^2L_{}^2}^2
\leq  \|\zeta(\nu^{-1/6},|\pa_z|)\Omega(\nu^{-1/6})\|_{L_{x,y}^2}^2+  C(C_1)\ep^3\nu^{2/3}.
\end{align}
As long as $\ep$ is chosen small enough compared to the bootstrap constant, we have the result. 

\end{proof}

\subsection{The Estimate of $\mathfrak{NL}_{\mathfrak{0l}}$}

We define $\mathfrak{NL}_{\mathfrak{0l}}$ as 
\begin{align}
\label{NLlo}
   \mathfrak{NL}_{\mathfrak{0l}} :=  \int  A ( B \p_v \po U^1 \p_z F) \,  A\oms \, d V.
\end{align}
We have for following lemma.
\begin{lemma}
    \label{lem:NLlo}
    {\bf a) Short time:} For $T_{*} \in (0, \nu^{-1/6}]$, under the bootstrap assumption of Proposition \ref{pro:short}, the following estimate holds: 
        \begin{align}
        \label{NLlo:est:short}
          \int_0^{T_*}|\mathfrak{NL}_{\mathfrak{0l}}| \, dt \lesssim   \nu^{-1/6} \norm{A\po U^1   }_{L_t^{\infty}L_v^2} \norm{A \p_z F}_{L_t^{\infty}L_{}^2}\norm{A\Omega^*}_{L_t^{\infty}L_{}^2}.
        \end{align}

        \noindent
         {\bf b) Long time:}  For $T_*\in [\nu^{-1/6}, c_*\nu^{-1}]$, 
        under the bootstrap assumption of Proposition \ref{pro:long}, the following estimate holds:
\begin{align}
\label{NLlo:est:long}
\int_{\nu^{-1/6}}^{T_*}|\mathfrak{NL}_{\mathfrak{0l}}| \, dt \lesssim   \nu^{-1/6} \norm{A\po U^1   }_{L_t^{\infty}L_v^2} \norm{A \p_z F }_{L_t^2L_{}^2}\norm{A\pn\Omega^*}_{L_t^{2}L_{}^2}.
        \end{align}
Hence, as a result of the bootstrap assumption \eqref{Hypotheses_sh} (\eqref{Hypotheses}, respectively), the following estimate holds for $\mathcal{I} = [0, T_*], [\nu^{-1/6}, T_*]$:
\begin{align}
    \int_{\mathcal{I}} |\mathfrak{NL}_{\mathfrak{0l}}| \, dt \lesssim \epsilon^3 \nu^{2/3}.
\end{align}
        
\end{lemma}
\begin{proof}
To bound \eqref{NLlo} for the short time estimates we use the Lemma \ref{A_product_rule_Hs}, Proposition \ref{linCK}, and the short time boostrap assumptions in Proposition \ref{pro:short} yields
\begin{align}
\label{NLlo:est:short1}
    &\int_{0}^{T_*} \int |  A ( B \p_v \po U^1 \p_z F) , A\oms  |\,dVdt \lesssim \int_0^{T_*} \norm{A(B \po U^1)   }_{L_v^2} \norm{A \p_z F}_{L_{}^2} \norm{A\Omega^*}_{L_{}^2}\, dt\\
    &\lesssim \nu^{-1/6} \norm{A\po U^1   }_{L_t^{\infty}L_v^2} \norm{A \p_z F}_{L_t^{\infty}L_{}^2}\norm{A\Omega^*}_{L_t^{\infty}L_{}^2} \notag
\end{align}
For long time estimate, the crucial observation is that since $\po U^1$ does not depend on $z$, we can replace $A\Omega^*$ by $A \pn \Omega^*$.  We then use the Lemma \ref{A_product_rule_Hs}, Proposition \ref{linCK}, 
 and the long time bootstrap assumptions in Proposition \ref{pro:long} to obtain 
\begin{align}
\label{NLlo:est:long1}
&\int_{\nu^{-1/6}}^{T_*}\int |  A ( B \p_v \po U^1 \p_z F)  A\pn \oms \ |\, dV dt  \lesssim \int_{\nu^{-1/6}}^{T_*} \norm{A(B \po U^1)   }_{L_v^2} \norm{A \p_z F}_{L_{}^2} \norm{A\pn\Omega^*}_{L_{}^2}\, dt \\
&\lesssim 
\nu^{-1/6}\norm{A \po U^1   }_{L_t^{\infty}L_v^2} \norm{A \p_z F }_{L_t^2L_{}^2} \norm{A\pn\Omega^*}_{L_t^2L_{}^2}. \notag
\end{align}
\end{proof}

\subsection{The Estimate of $\mathfrak{NL}_{\mathfrak{a}}$}\label{sec:NL_a}
We define $\mathfrak{NL}_{\mathfrak{a}}$ as 
\begin{align}
    \mathfrak{NL}_{\mathfrak{a}} := \int A(B \na_L^\perp\de_t^{-1}\Omega_\nq\cdot \na_L \Omega^* ) \  A\Omega^* \, dV
\end{align}
We have the following lemma.
\begin{lemma}
    \label{lem:NLa}
    \noindent {\bf a) Short time:} For $T_{*} \in (0, \nu^{-1/6}]$, under the bootstrap assumption of Proposition \ref{pro:short}, the following estimate holds as long as $\nu$ is chosen small enough: 
        \begin{align}
        \label{NLa:est:short}
           |\mathfrak{NL}_{\mathfrak{a}}|  \lesssim \ep \nu^{1/3}(\|A|\pa_z|^{1/3}\oms_\nq\|_{L^2}^2+\|A|\pa_z|^{1/3}F_\nq\|_{L^2}^2)+ (1+t)\|A\Omega\|_{L^2}\|A\Omega^*\|_{L^2}^2.
        \end{align}
        
        \noindent 
        {\bf b) Long time: }For $T_*\in [\nu^{-1/6}, c_*\nu^{-1}]$
        under the bootstrap assumption of Proposition \ref{pro:long}, the following estimate holds as long as $\nu$ is chosen small enough:
\begin{align}\n
 |\mathfrak{NL}_{\mathfrak{a}}|  \lesssim&  \ep \nu^{1/3}(\|A|\pa_z|^{1/3}\oms_\nq\|_{L^2}^2+\|A|\pa_z|^{1/3}F_\nq\|_{L^2}^2)+\ep\nu \|A\sqrt{-\de_L}\oms\|_{L^2}^2\\
 &+\ep \sum_{\iota\in \{I,\nu, E\}}CK_\iota [\oms]+{\ep} \sum_{\iota\in \{I,\nu\}}CK_\iota [F] .
\label{NLa:est:long}
        \end{align}
 Hence, as a result of the bootstrap assumption \eqref{Hypotheses_sh} (\eqref{Hypotheses}, respectively), the following estimate holds for $\mathcal{I} = [0, T_*], [\nu^{-1/6}, T_*]$:
    \begin{align}
        \int_{\mathcal{I}} |\mathfrak{NL}_{\mathfrak{a}}| \, dt\lesssim \ep^3\nu^{2/3}.
    \end{align}
\end{lemma}
To prove the lemma, we decompose the term as follows
\begin{align}   \mathfrak{NL}_{\mathfrak{a}} 
=&   \int \bigg(A\lf(B(\pa_v-t\pa_z)\de_t^{-1}\Omega_\nq \pa_z \Omega^* \rg)- B(\pa_v-t\pa_z)\de_t^{-1}\Omega_\nq \pa_z (A \Omega^*) \bigg) \  A\Omega^* dV\\
&\quad-\int \pa_z (B(\pa_v-t\pa_z)\de_t^{-1}\Omega_\nq) (A \Omega^*)^2 dV\\
&\quad-  \int \bigg(A(\pa_z\de_t^{-1}\Omega_\nq B(\pa_v-t\pa_z) \Omega^*) -(\pa_z\de_t^{-1}\Omega_\nq B(\pa_v-t\pa_z) A\Omega^*)\bigg)\  A\Omega^* dV \\
&\quad+  \int (\pa_v-t\pa_z) (B(\pa_z\de_t^{-1}\Omega_\nq ))(A\Omega^*)^2 dV\\
&\hspace{-1.2cm}=\sum_{k,\ell}\iint (A(k,\eta)-A(k-\ell,\eta-\xi))\lf(i(\xi-t\ell) \wh{(\de_t^{-1}\Omega_\nq)}(\ell,\xi) i(k-\ell) \wh{B\Omega^*}(\eta-\xi,k-\ell)\rg) \\
&\hspace{1.2cm} \times\overline{A\wh\Omega^*(k,\eta)} d\eta d\xi\\
&\hspace{-0.8cm} - \sum_{k,\ell}\iint (A(k,\eta)-A(k-\ell,\eta-\xi))\lf(i\ell \wh{(B\de_t^{-1}\Omega_\nq)}(\ell,\xi) i(\eta-\xi-(k-\ell)t) \wh\Omega^*(\eta-\xi,k-\ell)\rg) \\
&\hspace{1.2cm} \times\overline{A\wh\Omega^*(k,\eta)} d\eta d\xi\\
&\hspace{-0.8cm}+  \int  (\pa_v B)(\pa_z\de_t^{-1}\Omega_\nq)(A\Omega^*)^2 dV\\
&\hspace{-1.2cm}=:\mathfrak  E_{1}+\mathfrak  E_{2}+\mathfrak{E}_{3}.\label{E}
\end{align}
We organize the proof of Lemma \ref{lem:NLa} as a series of lemmas, each targeted at different components of \eqref{E}.
\begin{lemma}[$\mathfrak E_{1}$-estimate]\label{lem:E_1_proof} Assume the bootstrap hypothesis \eqref{Hypotheses}. For any $t\in[\nu^{-1/6},T_*]$, the following estimate holds
\begin{align}\label{E_1_est}
|\mathfrak{E}_{1}|\lesssim \myr{ \ep\nu^{1/3}\|A|\pa_z|^{1/3}\Omega^*_\nq\|_{L^2}^2}+\ep CK_I[\Omega^*_\nq]+\ep CK_I[F_\nq].
\end{align}
\end{lemma}
\begin{proof}

{\bf Step \# 1: Setup.} 
The idea of the proof is similar to the one that appeared in \cite{WZ23}. However, significant adjustments are required because we have a different multiplier $\mf M$ in the high-frequency regime ($|k|> \nu^{-1/2}$), and a varying coefficient $B$ appears in the Biot-Savart law. Thanks to the new structure of our multiplier, we decompose it as follows:
\begin{align}
    A(t,k,\eta)= \zeta_k(t) \underbrace{\mf M(t,k,\eta)\lan k,\eta\ran^s}_{=:\wt A(t,k,\eta)}.
\end{align}
We decompose the $k,\ell$-sum into the following two regimes:
\begin{align}\label{D_1_region}
    D_1:=&\{(k,\ell)\in \mathbb Z^2||\ell|\geq |k-\ell|/2\quad \text{or}\quad\min\{|k|,|k-\ell|\}\leq \nu^{-1/2}\},\\ \label{D_2_region}
    D_2:=&\{(k,\ell)\in \mathbb Z^2||\ell|< |k-\ell|/2  \quad \text{and}\quad\min\{|k|,|k-\ell|\}> \nu^{-1/2}\}.
\end{align}

We can further decompose the $D_1$ regime into two cases, i.e.,
\begin{align}
D_{1a}:=&\lf\{k,\ell\big||\ell|<|k-\ell|/2 \quad \text{and} \quad \min\{|k|,|k-\ell|\}\leq \nu^{-1/2}\rg\};\\
D_{1 b} := &\lf\{k,\ell\big||\ell|\geq |k-\ell|/2\rg\}.
\end{align}

The commutator term $\mf E_1$ can be decomposed as follows:
\begin{align}
    \mf E_1=& \sum_{(k,\ell)\in D_1}\iint \zeta_k(\wt A(k,\eta)-\wt A(k-\ell,\eta-\xi))\lf(B(\pa_v-t\pa_z) (\de_t^{-1}\Omega_\nq)\rg)^\wedge (\ell,\xi)\\
    &\hspace{4cm}\times i(k-\ell) \wh{\Omega^*}(\eta-\xi,k-\ell) \  \overline{A\wh\Omega^*(k,\eta)} d\eta d\xi \\
    &+\sum_{(k,\ell)\in D_2}   \iint \zeta_k(\wt A(k,\eta)-\wt A(k-\ell,\eta-\xi))\lf(B(\pa_v-t\pa_z) (\de_t^{-1}\Omega_\nq)\rg)^\wedge (\ell,\xi)\\
    &\hspace{4cm}\times i(k-\ell) \wh{\Omega^*}(\eta-\xi,k-\ell) \  \overline{A\wh\Omega^*(k,\eta)} d\eta d\xi \\
    &+ \sum_{(k,\ell)\in \mathbb{Z}^2}  \iint (\zeta_k(t)-\zeta_{k-\ell}(t))    \lf(B(\pa_v-t\pa_z)(\de_t^{-1}\Omega_\nq)\rg)^\wedge(\ell,\xi) \\
&\hspace{4cm}\times i(k-\ell) \wt A(k-\ell,\eta-\xi) \wh{\Omega^*}(\eta-\xi,k-\ell) \   \overline{A\wh\Omega^*(k,\eta)} d\eta d\xi\\
=:&T_1+T_2+T_3.\label{E_1}
\end{align}

\noindent 
{\bf Step \# 2: The \myr{$T_{1}$} estimate.}
In Case $D_{1a}$, we have that $|k-\ell|\leq |k|+|\ell|\leq |k|+\frac{1}{2}|k-\ell|$, and hence $|k-\ell|\leq 2|k|$. As a result, 
\begin{align}
|k-\ell|\leq (2|k|)^{1/3}|k-\ell|^{1/3}\min\{2|k|, |k-\ell|\}^{1/3}\leq2\nu^{-1/6}|k|^{1/3}|k-\ell|^{1/3}. 
\end{align}
Hence, if $|k-\ell|$ is in high frequency, one can move one-third of a derivative to $k$ and pay $\nu^{-1/6}$.  In Case $D_{1b}$, we obtain $|k-\ell|\leq 2|\ell|$. To summarize, we have the following
\begin{align}
|k-\ell|\leq 2(|\ell|+\nu^{-1/6}|k|^{1/3}|k-\ell|^{1/3}),\quad \forall (k,\ell)\in D_1. 
\end{align}

Now we apply this relation to estimate the $D_1$-contribution as follows
\begin{align}
    T_1&\lesssim \sum_{(k,\ell)\in D_1} \iint \zeta_k(\wt A(k-\ell,\eta-\xi)+\wt A(\ell,\xi))\lf(\lf|i(\xi-t\ell) \wh{(\de_t^{-1}\Omega_\nq)}(\ell,\xi)\rg|\myr{|k-\ell|} \lf|\wh{B\Omega^*}(\eta-\xi,k-\ell)\rg|\rg) \\
&\hspace{2.2cm} \times  |\overline{  A\wh\Omega^*(k,\eta)}| d\eta d\xi\\
  &\lesssim \sum_{(k,\ell)\in D_1} \iint \zeta_k\wt A(\ell,\xi)\lf(\lf|i(\xi-t\ell) \wh{(\de_t^{-1}\Omega_\nq)}(\ell,\xi)\rg| |k-\ell|  \lf|\wh{B\Omega^*}(\eta-\xi,k-\ell)\rg|\rg) \  |\overline{A\wh\Omega^*(k,\eta)}| d\eta d\xi\\
  &\quad + \sum_{(k,\ell)\in D_1} \iint \zeta_k\wt A(k-\ell,\eta-\xi)(\myr{|\ell|+\nu^{-1/6}|k|^{1/3}|k-\ell|^{1/3}})\lf|i(\xi-t\ell) \wh{(\de_t^{-1}\Omega_\nq)}(\ell,\xi)\rg|\\ 
 &\hspace{2.6cm} \times \lf|\wh{B\Omega^*}(\eta-\xi,k-\ell)\rg|  |\overline{ A\wh\Omega^*(k,\eta)}| d\eta d\xi \\
&=:T_{11}+T_{12}.
\end{align}
For the $T_{11}$-term, we recall the definition $(-\de_L f(z,v))^\wedge:=(|k|^2+|\eta-kt|^2)\wh f(k,\eta)$ to rewrite it as follows:
\begin{align}
 T_{11} 
 \lesssim&\sum_{(k,\ell)\in D_1} \iint \zeta_k\wt A(\ell,\xi)\lf|\lf( \frac{i\xi- it\ell}{|\ell|^2+|\xi-\ell t|^2}\lf({\de_L\de_t^{-1}\Omega_\nq}\rg)^\wedge(\ell,\xi)\rg|  |k-\ell|\lf|\wh{B\Omega^*}(\eta-\xi,k-\ell)\rg|\rg) \\
 &\hspace{1.6cm}\times|\overline{ A\wh\Omega^*(k,\eta)}| d\eta d\xi. 
\end{align}
Now we estimate this term using the fact that $\mathbbm{1}_{\ell\neq 0}\max\{|k|,|k-\ell|\}\geq 1$, the elliptic estimate \eqref{ell_est},  the Kato-Ponce inequality \eqref{KP}, and the bootstrap assumption \eqref{Hyp_ed} to obtain that
\begin{align}
T_{11}\lesssim&  \lf\| A \frac{|\pa_v-t\pa_z|}{\de_L}\Omega_\nq\rg\|_{L^2}\lf(\|\lan \pa_z,\pa_v\ran^s(B\Omega^*_\nq)\|_{L^2}
\| A\Omega^* \|_{L^2}+\|\lan \pa_z,\pa_v\ran^s(B\Omega^*)\|_{L^2} \| A\Omega^*_\nq \|_{L^2}\rg)\\
\lesssim&\ep \nu^{1/3} \lf\| A\sqrt{\frac{-\pa_t\W_I}{\W_I}} \Omega_\nq \rg\|_{L^2} \| A\Omega^*_\nq \|_{L^2} 
\lesssim  \ep CK_I[\oms_\nq] +\ep CK_I[F_\nq]+\ep\nu^{2/3}\|A\Omega^*_\nq \|_{L^2}^2.
\end{align}
Here in the last line, we have recalled the definition $\Omega_\nq=\Omega_\nq^*+F_\nq$. This is consistent with the conclusion \eqref{E_1_est}. 

Next we estimate the $T_{12}$-term with a similar argument,
\begin{align} 
    &|T_{12}|\lesssim\sum_{(k,\ell)\in D_1} \iint \zeta_k\wt A(k-\ell,\eta-\xi) \myr{|\ell|}\lf|(i\xi-i t\ell )\wh{(\de_t^{-1}\Omega_\nq)}(\ell,\xi)\rg|  \lf|\wh{B\Omega^*}(\eta-\xi,k-\ell)\rg| \\
 &\hspace{3.2cm}\times  |\overline{ A\wh\Omega^*(k,\eta)}| d\eta d\xi\\
   &
 \hspace{1.2cm}+ \nu^{-1/6}\sum_{(k,\ell)\in D_1} \iint \zeta_k\wt A(k-\ell,\eta-\xi) \lf|(i\xi-i t\ell )\wh{(\de_t^{-1}\Omega_\nq)}(\ell,\xi)\rg| |k-\ell|^{1/3}\lf |\wh{B\Omega^*}(\eta-\xi,k-\ell)\rg| \\
 &\hspace{3.2cm}\times   |\overline{ A|k|^{1/3}\wh\Omega^*(k,\eta)}| d\eta d\xi\\
   &\lesssim \sum_{(k,\ell)\in D_1} \iint  \zeta_k\wt A(k-\ell,\eta-\xi)\lf( \myr{|\ell|}\lf|\lf(\frac{\pa_v- t\pa_z}{\de_L}(\de_L\de_t^{-1}  \Omega_\nq)\rg)^\wedge(\ell,\xi)\rg|  \lf|\wh{B\Omega^*}(k-\ell,\eta-\xi)\rg|\rg) \\
 &\hspace{3.2cm}\times   |\overline{ A\wh\Omega^*(k,\eta)}| d\eta d\xi\\
&\hspace{1cm}+\nu^{-1/6}\sum_{(k,\ell)\in D_1}\lf\|\zeta_\ell\frac{\eta- t\ell }{\ell^2+|\eta-t\ell|^2}\wh{(\de_L\de_t^{-1}\Omega_\nq)}(\ell,\cdot)\rg\|_{L_\eta^1}\lf  \| A(k-\ell,\cdot) |k-\ell|^{1/3}\wh{B\Omega^*}(k-\ell,\cdot)\rg\|_{L_\eta^2}\\
 &\hspace{3.2cm}\times   \| A|k|^{1/3}\wh\Omega^*(k,\eta)\|_{L_\eta^2} \\
&=: T_{121}+T_{122}.
\end{align}
We note that for these two terms, $k$ and $k-\ell$ cannot be zero at the same instance because $\ell\neq 0$. 
For the $T_{121}$ term, we estimate it with Lemma \ref{lem:ellip} as follows
\begin{align}
\lf|T_{121}\rg|\lesssim& \lf\| A\frac{|\pa_z||\pa_z,\pa_v-t\pa_z|}{- \de_L} \Omega_\nq\rg\|_{L_{}^2}(\| A(B\Omega^*)\|_{L^2}\| A\Omega^*_\nq\|_{L^2}+\| A(B\Omega^*_\nq)\|_{L^2}\| A\Omega^*\|_{L^2})\\
\lesssim& \lf\| A\sqrt{\frac{-\pa_t\mathcal W_I}{\mathcal W_I}}  \Omega_\nq\rg\|_{L_{}^2}\| A\Omega^*\|_{L^2}\| A\Omega^*_\nq\|_{L^2} 
\lesssim\ep (CK_I[\oms_\nq]+CK_I[F_\nq])+\ep \nu^{2/3}\| A\Omega^*_\nq\|_{L^2}^2.
\end{align}
In the last line, we used the equation for $\pa_t\mathcal W_I$. This is consistent with the conclusion. 

For the $T_{122}$-term, we invoke the elliptic estimate \eqref{ell_est} and the time constraint $t\geq \nu^{-1/6}$ to obtain
\begin{align}
\lf|T_{122}\rg|\lesssim&\nu^{-1/6}\lf\|\frac{ A(k,\eta)\Omega_\nq(k,\eta)}{|k|(1+|t-\eta /k |)\lan \eta\ran}\rg\|_{\myr{\ell_k^1 L_\eta^2}}\| A|\pa_z|^{1/3}\Omega^*_\nq\|_{L^2}^2\lesssim  \frac{\nu^{-1/6}}{\lan t\ran}\lf\|  A\Omega_\nq\rg\|_{L^2}\| A|\pa_z|^{1/3}\Omega^*_\nq\|_{L^2}^2 \\
\lesssim&\nu^{-1/6+1/6}\lf\| A\Omega_\nq\rg\|_{L_{}^2}\| A|\pa_z|^{1/3}\Omega^*_\nq\|_{L^2} ^2
\lesssim \ep \nu^{1/3}\| A|\pa_z|^{1/3}\Omega^*_\nq\|_{L^2}^2.
\end{align}
Here we have used the inviscid damping estimate $(1+|t-\eta/k|)\lan\eta\ran\gtrsim \lan t\ran$ for $k\neq0$ and the bootstrap assumption \eqref{Hypotheses}. At this point, we have confirmed that all the components of $T_1$ satisfy the estimate \eqref{E_1_est}. This concludes {\bf Step \# 2}. 

\noindent 
{\bf Step \# 3: The $T_2$ estimate.}
{In the region where $(k,\ell)\in D_2$, our analysis deviates from the work \cite{WZ23}. We observe the following 
\begin{align}
    \ell\neq 0,\quad k(k-\ell)>0,\quad\wt A(k,\eta)=\W_I(k,\eta)\lan k,\eta\ran^s,\quad \wt A(k-\ell,\eta-\xi)=\W_I(k-\ell,\eta-\xi)\lan k-\ell,\eta-\xi\ran^s.
\end{align}
Hence,
\begin{align}
&| \wt A(k,\eta)-\wt A(k-\ell,\eta-\xi)| \\
&\leq|\lan k,\eta\ran^s-\lan k-\ell,\eta-\xi\ran^s|\W_I(k,\eta) +|\W_I(k,\eta)-\W_I(k-\ell,\eta-\xi)|\lan k-\ell,\eta-\xi\ran^s.\label{wtA_com}
\end{align}
For the first term, we can estimate it as follows 
\begin{align}
    |\lan k,\eta\ran^s-\lan k-\ell,\eta-\xi\ran^s|\W_I(k,\eta)&\lesssim (\lan k-\ell,\eta-\xi\ran^{s-1}+\lan \ell,\xi\ran^{s-1})(|\ell|+|\xi|).
\end{align}
For the second term in the expression \eqref{wtA_com}, we recall the explicit form of the multiplier \eqref{W_I} and introduce an interpolation 
\begin{align}\label{L_z_interpolant}
\mf L=\pi+\arctan\lf(\frac{1}{K}\frac{\theta(\eta-kt)+(1-\theta)(\eta-\xi-(k-\ell)t}{\theta k+(1-\theta) (k-\ell)}\rg),\quad \theta\in[0,1].
\end{align}

We note that $\mf L(1)=\wt A(k,\eta),\, \mf L(0)=\wt A(k-\ell,\eta-\xi)$. Moreover, since $k(k-\ell)>0$, the expression $|\theta k +(1-\theta)(k-\ell)|\geq 1$. Hence the mean value theorem yields that
\begin{align}
&|\W_I(k,\eta)-\W_I(k-\ell,\eta-\xi)|\leq \sup_{\theta\in[0,1]} \lf|\frac{d}{d\theta}\mf L(\theta)\rg|\\
&\leq\sup_{\theta\in[0,1]} \frac{|\theta k+(1-\theta)(k-\ell)||\xi-\ell t|+|\theta(\eta-kt)+(1-\theta)(\eta-\xi-(k-\ell)t)| |\ell|}{K(\theta k+(1-\theta)(k-\ell))^2+K^{-1}(\theta(\eta-kt)+(1-\theta)(\eta-\xi-(k-\ell)t)^2}\\
&\lesssim   \sup_{\theta\in[0,1]} \frac{|\xi-\ell t|}{K^{1/2}|\theta k+(1-\theta)(k-\ell)|} + \sup_{\theta\in[0,1]} \frac{ K^{1/2}|\ell|}{K^{1/2}|\theta k+(1-\theta)(k-\ell)|} \\
&\lesssim \frac{|\xi-\ell t|+|\ell|}{\min\{|k|,|k-\ell|\}}.\label{com_WIz_bd}
\end{align}
As a consequence, we have that 
\begin{align}\label{com_D_2}
    | \wt A(k,\eta)-\wt A(k-\ell,\eta-\xi)| \lesssim (\lan k-\ell,\eta-\xi\ran^{s-1}+\lan \ell,\xi\ran^{s-1})(|\ell|+|\xi|)+\lan k-\ell,\eta-\xi\ran^{s}\frac{|\xi-\ell t|+|\ell|}{\min\{|k|,|k-\ell|\}}.
\end{align}Now we decompose the $T_{D_2}$-term as follows
\begin{align}
&T_{2}\leq\sum_{\substack{(k,\ell)\in D_2}} \bigg|\iint \zeta_k(\wt A(k,\eta)-\wt A(k-\ell,\eta-\xi)) \\
&\hspace{3cm}\times\lf( ((\pa_v-t\pa_z)(B\de_t^{-1}\Omega_\nq))^\wedge(\ell,\xi) i(k-\ell) \wh{\Omega^*}(\eta-\xi,k-\ell)\rg)\overline{ A\wh\Omega^*(k,\eta)} d\eta d\xi\bigg|\\
&+\sum_{\substack{(k,\ell)\in D_2}} \bigg|\iint \zeta_k(\wt A(k,\eta)-\wt A(k-\ell,\eta-\xi)) \\
&\hspace{1.2cm}\times\lf( {((\pa_v B)\de_t^{-1}\Omega_\nq)^\wedge}(\ell,\xi) i(k-\ell) \wh{\Omega^*}(\eta-\xi,k-\ell)\rg)\overline{ A\wh\Omega^*(k,\eta)} d\eta d\xi\bigg|=:T_{21}+T_{22}.
\end{align}
By invoking the relation \eqref{com_D_2}, we have that 
  \begin{align}  &T_{21}\lesssim \sum_{(k,\ell)\in D_2} \iint \zeta_k(\lan k-\ell,\eta-\xi\ran^{s-1}+\lan \ell,\xi\ran^{s-1})(|\ell|+|\xi|)\\
&\hspace{2cm}\times\lf(\frac{|\xi-t\ell|}{\ell^2+|\xi-\ell t|^2} \lf|(\de_L (B\de_t^{-1}\Omega_\nq))^\wedge(\ell,\xi)\rg||k-\ell| \lf|\wh{\Omega^*}(k-\ell,\eta-\xi)\rg|\rg) \  |\overline{ A\wh\Omega^*(k,\eta)}| d\eta d\xi\\
&+\sum_{(k,\ell)\in D_2} \iint \zeta_k \frac{|\xi-\ell t|+|\ell|}{\min\{|k|,|k-\ell|\}}|k-\ell|\\
&\hspace{1.2cm}\times\lf(\frac{|\xi-t\ell|}{\ell^2+|\xi-\ell t|^2} \lf|(\de_L( B\de_t^{-1}\Omega_\nq))^\wedge(\ell,\xi)\rg| \lan k-\ell,\eta-\xi\ran^s \lf|\wh{\Omega^*}(k-\ell,\eta-\xi)\rg|\rg) \  |\overline{ A\wh\Omega^*(k,\eta)}| d\eta d\xi.\end{align} 
Next, we invoke the $\zeta$-property \eqref{zeta_product} to estimate all the $\zeta_k$ factors and apply the fact that in $D_2$, $|k-\ell|\approx |k|$ to estimate the last term. As a result,
\begin{align} &T_{21}\lesssim \\
&\sum_{(k,\ell)\in D_2} \iint
\frac{ \zeta_\ell |\xi-\ell t|\lan\ell,\xi\ran}{\ell^2+|\xi-\ell t|^2} \lf|(\de_L( B\de_t^{-1}\Omega_\nq))^\wedge(\ell,\xi)\rg|\zeta_{k-\ell}\lan k-\ell,\eta-\xi\ran^{s}\lf|\wh{\Omega^*}(k-\ell,\eta-\xi)\rg| \  |\overline{ A\wh\Omega^*(k,\eta)}| d\eta d\xi\\
& +\sum_{(k,\ell)\in D_2} \iint \frac{\zeta_\ell\lan \ell,\xi\ran^{s}|\xi-t\ell|}{\ell^2+|\xi-\ell t|^2} \lf|(\de_L (B\de_t^{-1}\Omega_\nq))^\wedge(\ell,\xi)\rg|\zeta_{k-\ell}|k-\ell| \lf|\wh{\Omega^*}(k-\ell,\eta-\xi)\rg| \  |\overline{ A\wh\Omega^*(k,\eta)}| d\eta d\xi\\
& +\sum_{(k,\ell)\in D_2} \iint \zeta_\ell  \lf|(\de_L (B\de_t^{-1}\Omega_\nq))^\wedge(\ell,\xi)\rg| \zeta_{k-\ell}\lan k-\ell,\eta-\xi\ran^s \lf|\wh{\Omega^*}(k-\ell,\eta-\xi)\rg| \  |\overline{ A\wh\Omega^*(k,\eta)}| d\eta d\xi.
\end{align}
Finally, applying the Young's convolution inequality $\|f*g\|_{\ell_k^2L_\eta^2}\leq C\|f\|_{\ell_k^{p_1}L_\eta^{p_2}}\|g\|_{\ell_k^{q_1}L_\eta^{q_2}},\, p_i^{-1}+q_i^{-1}={3/2}, \, i=1,2$ (which is a consequence of the classical Young's inequality and Minkowski inequality), the H\"older inequality and the elliptic estimate \eqref{ell_est_t} yields that,
\begin{align}
  &T_{21} \lesssim \lf\|\zeta_k |k|^{-1}\lan k,\eta\ran (\de_L (B\de_t^{-1}\Omega_\nq))^\wedge(k,\eta)\rg\|_{\ell_k^1 L_\eta^1}\|A\oms_\nq\|_{L^2}^2\\
   &\quad\quad+\lf\|\zeta_k |k|^{-1}\lan k,\eta\ran^s (\de_L (B\de_t^{-1}\Omega_\nq))^\wedge(k,\eta)\rg\|_{\ell_k^1 L_\eta^2}\|\zeta_k|k|\wh{\oms_\nq}\|_{\ell_k^2L_\eta^1}\|A\oms_\nq\|_{L^2}\\
   &\quad\quad +\lf\|\zeta_k (\de_L (B\de_t^{-1}\Omega_\nq))^\wedge(k,\eta)\rg\|_{\ell_k^1 L_\eta^1}\|A\oms_\nq\|_{L^2}^2\\ 
   &\lesssim \|\zeta(t,|\pa_z|)\de_L (B\de_t^{-1}\Omega_\nq)\|_{H^2}\|A\oms_\nq\|_{L^2}^2+\|\zeta(t,|\pa_z|)\de_L (B\de_t^{-1}\Omega_\nq)\|_{H^2}\|\zeta(t,|\pa_z|) \oms_\nq\|_{H^2}\|A\oms_\nq\|_{L^2}
   \\
   &\quad+\|\zeta(t,|\pa_z|)\de_L (B\de_t^{-1}\Omega_\nq)\|_{H^2}\|A\oms_\nq\|_{L^2}^2\\
   &\lesssim \|A\Omega_\nq\|_{L^2}\|  A\Omega^*_\nq\|_{L^2}^2. 
\end{align}
Thanks to the bootstrap assumption and the estimate \eqref{F:CK}, we have that this is consistent with \eqref{E_1_est}. The estimate of the $T_{22}$ term is similar. One can replace the $B$ by $\pa_v B$ and drop the $|\xi-\ell t|$ factor in the argument above to derive a similar bound. Hence, we omit the details for the sake of brevity. We also highlight that this argument is applicable if one replaces the $|k-\ell||\wh{\Omega^*_\nq}(k-\ell,\eta-\xi)|$ by $|k-\ell||\wh {F_\nq}(k-\ell,\eta-\xi)|$. This concludes the step.
}

\noindent 
{\bf Step \# 4: The $T_3$ estimate.} We recall that $\ell\neq 0$. To estimate the $T_3$ term, we distinguish between two cases, i.e.,
\begin{align}
    \begin{cases}
        \text{Case a) }|k-\ell|\leq 4|\ell|;\\
        \text{Case b) }|k-\ell|> 4|\ell|.
    \end{cases}
\end{align}
In Case a), we provide the estimate
\begin{align}
    |\zeta_k(t)-\zeta_{k-\ell}(t)|\leq \zeta_k(t)+\zeta_{k-\ell}(t)\leq \zeta_{\ell}(t)\zeta_{k-\ell}(t)+\zeta_{k-\ell}(t).
\end{align}
We also observe the following commutator relation in Case b):
\begin{align}
    &|k-\ell|>4|\ell|\geq 4\quad\Rightarrow\quad k(k-\ell)>0\quad \Rightarrow \\
    & |\zeta_k(t)-\zeta_{k-\ell}(t)|\leq  \exp\lf\{\delta \nu^{1/3}(1+\max\{|k|,|k-\ell|\}^{2/3})t\rg\}\frac{2\delta\nu^{1/3}t  }{3|k|^{1/3}}|\ell|\\
    &\hspace{1.4cm}\lesssim (\delta\nu^{1/3}t)\exp\lf\{-\delta\nu^{1/3}t\rg\}\zeta_{k-\ell}(t)\zeta_\ell(t)\frac{|\ell|}{|k-\ell|^{1/3}}\label{zeta_com_sh}\\
    &\hspace{1.4cm}\lesssim \zeta_{k-\ell}(t)\zeta_\ell(t)\frac{|\ell|}{|k-\ell|^{1/3}}\exp\lf\{-\frac{\delta}{2}\nu^{1/3}t\rg\}.
\end{align}
Now we estimate the term as follows
\begin{align}
 T_3\leq &\sum_{k,\ell}\mathbbm{1}_{|k-\ell|\leq 4|\ell|}\iint  \zeta_\ell\Bigg|i\ell\lf(\frac{(\pa_v-t\pa_z)}{\de_L}(\de_L(B\de_t^{-1}\Omega_\nq))\rg)^\wedge-i\ell\lf(\pa_vB\de_t^{-1}\Omega_\nq)\rg)^\wedge\Bigg|(\ell,\xi) \\
&\hspace{4cm}\times| A(k-\ell,\eta-\xi) \wh{\Omega^*}(\eta-\xi,k-\ell)| \   |{A\wh\Omega^*(k,\eta)}| d\eta d\xi\\
&+\sum_{k,\ell}\mathbbm{1}_{|k-\ell|> 4|\ell|}\iint  \zeta_\ell\Bigg|i\ell\lf(\frac{(\pa_v-t\pa_z)}{\de_L}(\de_L(B\de_t^{-1}\Omega_\nq))\rg)^\wedge-i\ell\lf(\pa_vB\de_t^{-1}\Omega_\nq)\rg)^\wedge\Bigg|(\ell,\xi) \\
&\hspace{4cm}\times|k-\ell|^{1/3}| A(k-\ell,\eta-\xi) \wh{\Omega_\nq^*}(\eta-\xi,k-\ell)| \   |k|^{1/3}|{A\wh\Omega_\nq^*(k,\eta)}| d\eta d\xi\\
\lesssim& \ep\nu^{1/3}\|A|\pa_z|^{1/3}\oms_\nq\|_{L^2}^2. 
\end{align}
This is consistent with the estimate \eqref{E_1_est} and hence concludes the proof.
\end{proof}

\begin{lemma}[$\mathfrak E_{2}$-estimate]\label{lem:E_2_proof} 
For any $t\in [\nu^{-1/6},T_*] $, the following estimate holds
\begin{align}\label{E_2_est}
|\mathfrak{E}_{2}|\lesssim& \ep\nu^{1/3}\|A\Omega^*_\nq\|_{L^2}^2+\ep\lf\|A\sqrt{\frac{-\pa_t \mf M}{\mf M}}\Omega^*\rg\|_{L^2}^2\\
&+\ep\lf\|A\sqrt{\frac{-\pa_t \W_I}{\W_I}}F\rg\|_{L^2}^2+\ep \nu\|A\sqrt{-\de_L}\Omega^*\|_{L^2}^2 .
\end{align}
\end{lemma}
\begin{proof}
 For the $\mathfrak{E}_{2}$ contribution, we decompose it into the contribution as follows
 \begin{align}
     \mathfrak{E}_{2}\lesssim &\sum_{k,\ell}\iint \zeta_k\lan\xi,\ell\ran ^s |\ell ||\wh{(B\de_t^{-1}\Omega_\nq)}(\ell,\xi) ||(\eta-\xi-(k-\ell)t) \wh\Omega^*(\eta-\xi,k-\ell)| \  |A\wh\Omega^*(k,\eta)| d\eta d\xi\\
     &+ \sum_{k,\ell}\iint \zeta_k|\ell|| \wh{(B\de_t^{-1}\Omega_\nq)}(\ell,\xi)| |\eta-\xi-(k-\ell)t| |\wt A(k-\ell,\eta-\xi)\wh\Omega^*(\eta-\xi,k-\ell)| \  |{A\wh\Omega^*(k,\eta)}| d\eta d\xi \\
     &+ \sum_{k,\ell}  \iint (\zeta_k(t)-\zeta_{k-\ell}(t))    \lf(\pa_z(B\de_t^{-1}\Omega_\nq)\rg)^\wedge(\ell,\xi) \\
&\hspace{4cm}\times  \wt A(k-\ell,\eta-\xi) \ ((\pa_v-t\pa_z)\Omega^*)^\wedge(\eta-\xi,k-\ell) \   \overline{A\wh\Omega^*(k,\eta)} d\eta d\xi\\
 =:&T_{1}+T_{2}+T_3. 
 \end{align}
The $T_1,\, T_2$ terms can be estimated with the relation \eqref{zeta_product} and Lemma \ref{lem:Tri_est}. \myb{(Check!)}

The last term can be estimated as follows. By the estimate \eqref{zeta_product}, we have that $|\zeta_k-\zeta_{k-\ell}|\leq 2\zeta_k\zeta_{k-\ell}$, and hence, 
\begin{align*}
    |T_3|\lesssim& \sum_{k,\ell}  \iint  \zeta_{\ell}    |\lf(\pa_z(B\de_t^{-1}\Omega_\nq)\rg)^\wedge(\ell,\xi)| \\
&\hspace{4cm}\times  A(k-\ell,\eta-\xi) \ |((\pa_v-t\pa_z)\Omega^*)^\wedge(\eta-\xi,k-\ell) |\  | \overline{A\wh\Omega^*(k,\eta)} |d\eta d\xi\\
\lesssim &\sum_{k\neq 0}\int\lf|\zeta_k \frac{|k|}{|k|^2+|\eta-kt|^2}(\de_L(B\de_t^{-1}\Omega_\nq)^\wedge(k,\eta)\rg|d\eta\, \|A\sqrt{-\de_L}\Omega^*\|_{L^2}\|A\Omega^*\|_{L^2}\\
\lesssim &\sum_{k\neq 0}\int\lf|\zeta_k \frac{\lan \eta\ran}{|k|\lan \eta\ran(1+|\eta/k-t|)}\sqrt{\frac{-\pa_t \W_I}{\W_I}}(\de_L(B\de_t^{-1}\Omega_\nq)^\wedge(k,\eta)\rg|d\eta\, \|A\sqrt{-\de_L}\Omega^*\|_{L^2}\|A\Omega^*\|_{L^2}\\
\lesssim &\lf(\sum_{k\neq 0}\frac{1}{|k|^2}\int\frac{1}{\lan\eta \ran^2}d\eta\rg)^{1/2}\lf(\sum_{k\neq 0}\int\lf| \frac{\zeta_k\lan \eta\ran^2}{\lan t\ran}\sqrt{\frac{-\pa_t \W_I}{\W_I}}(\de_L(B\de_t^{-1}\Omega_\nq)^\wedge(k,\eta)\rg|^2d\eta\rg)^{1/2}\\
&\hspace{1.2cm} \times \|A\sqrt{-\de_L}\Omega^*\|_{L^2}\|A\Omega^*\|_{L^2}\\
\lesssim & \underbrace{\frac{1}{\lan t\ran}}_{\leq\nu^{1/6}}\lf\|A\sqrt{\frac{-\pa_t \W_I}{\W_I}}\Omega_\nq\rg\|_{L^2} \|A\sqrt{-\de_L}\Omega^*\|_{L^2}\|A\Omega^*\|_{L^2}\\
\lesssim &\ep \nu\|A\sqrt{-\de_L}\Omega^*\|_{L^2}^2+\ep CK_I[\oms_\nq] + \ep CK_I[F].
\end{align*}


\end{proof}

\begin{lemma}[$\mathfrak E_{3}$-estimate]\label{lem:E_3_proof}For $t\in[0,T_*]$ or $[\nu^{-1/6}, T_*]$, the following estimate holds
\begin{align}\label{E_3_est}
|\mathfrak{E}_{3}|\lesssim \ep\nu^{1/3}\|A\Omega^*_\nq\|_{L^2}^2+\ep\nu^{1/3}\|AF_\nq\|_{L^2}^2.
\end{align}
\end{lemma}
\begin{proof}
We can expand the $\mathfrak E_{3}$-term as follows
\begin{align}
  |\mathfrak{E}_{3}|\leq& \lf|\int (\pa_v B)(\pa_z\de_t^{-1}\Omega_\nq) (A\Omega^*_\nq) A\Omega^*_{\mathbbm{o}} dV\rg|+\lf|\int (\pa_v B)(\pa_z\de_t^{-1}\Omega_\nq) (A\Omega^*_{\mathbbm{o}}) A\Omega^*_\nq dV\rg|\\
  &+\lf|\int (\pa_v B)(\pa_z\de_t^{-1}\Omega_\nq) (A\Omega^*_\nq) A\Omega^*_\nq dV\rg|\\
  \lesssim & \|\pa_v B\|_{L^\infty}\lf\|A\Omega_\nq\rg\|_{L^2}\|A\Omega^*_\nq\|_{L^2}\|A\Omega^*\|_{L^2}\\
  \lesssim & \ep \nu^{1/3}\|A\Omega^*_\nq\|_{L^2}^2+\ep\nu^{1/3}\|AF_\nq\|_{L^2}^2.
\end{align}
This concludes the proof of the lemma. 
\end{proof}

\begin{lemma}\label{lem:E_4_proof} For any $t\in [0, \nu^{-1/6}]$, the following estimate holds
\begin{align}\label{E_12est_sh}
    \mf E_{1}+\mf E_{2}\lesssim (1+t)\|A\Omega\|_{L^2}\|A\oms\|_{L^2}^2+\ep CK_I[\Omega_\nq^*]+\ep CK_I[F_\nq]+\ep\nu\|A\sqrt{-\de_L}\oms\|_{L^2}^2.
\end{align}
\end{lemma}
\begin{proof}We organize the proof in two steps. 
{
\myb{HS: Right now, it seems that we need to propagate the following weight in the initial time layer in order to get the $k$-dependent enhanced dissipation for $t\leq \nu^{-1/6}$ (For $k\in[\nu^{-1/4},\nu^{-1/2}],$ the estimate cannot be easily derived):
\begin{align}
\zeta_k = e^{\delta\nu^{1/3}(|k|^{2/3}+1)t}.
\end{align}
The way to treat the commutator term is the following.}

\noindent
{\bf Step \# 1:} We focus on the term $\mf E_1$ for $t\leq \nu^{-1/6}$: 
\begin{align}
\mf E_1&\lesssim  \bigg| \sum_{k,\ell}\iint \zeta_k(\wt A(k,\eta)-\wt A(k-\ell,\eta-\xi))\bigg((B(\pa_v-t\pa_z)\de_t^{-1}\Omega_\nq)^\wedge(\ell,\xi) i(k-\ell) \wh\Omega^*(\eta-\xi,k-\ell)\bigg) \\
&\hspace{4cm}\times  \overline{A\wh\Omega^*(k,\eta)} d\eta d\xi\bigg|\\
&\quad +\sum_{k,\ell}\iint  |\zeta_k-\zeta_{k-\ell}| \Bigg|\lf(\frac{(\pa_v-t\pa_z)}{\de_L}(\de_L(B\de_t^{-1}\Omega_\nq))\rg)^\wedge-\lf((\pa_vB)\de_t^{-1}\Omega_\nq)\rg)^\wedge\Bigg|(\ell,\xi) \\
&\hspace{4cm}\times| \wt A(k-\ell,\eta-\xi) i(k-\ell)\wh{\Omega^*}(\eta-\xi,k-\ell)| \   |{A\wh\Omega^*(k,\eta)}| d\eta d\xi\\
&=:T_1+T_2.
\end{align}
By rewriting $B(\pa_v-t\pa_z)\de_t^{-1}\Omega_\nq$ as $\frac{(\pa_v-t\pa_z)}{\de_L}\de_L(B\de_t^{-1}\Omega_\nq)-\frac{1}{\de_L}\de_L((\pa_v B)\de_t^{-1}\Omega_\nq)$, and invoking the commutator estimate \eqref{cm_sh_b} with appropriately chosen $\mathfrak{S}$ (i.e., $\mathfrak{S}=\de_L(B\de_t^{-1}\Omega_\nq)$ or $\de_L((\pa_v B)\de_t^{-1}\Omega_\nq$) as described in Lemma \ref{lem:Acm_sh}, we apply the elliptic estimate \eqref{ell_est_t} to obtain the bound 
\begin{align}|T_1|\lesssim& (\|A\de_L(B\de_t^{-1}\Omega_\nq)\|_{L^2}+\|A\de_L((\pa_vB)\de_t^{-1}\Omega_\nq)\|_{L^2}))\|A\oms\|_{L^2}^2 \lesssim\|A\om_\nq\|_{L^2}\|A\oms\|_{L^2}^2.
\end{align}
For the $T_2$ term, we estimate it with a new commutator property of the $\zeta_k$-multiplier. Thanks to the mean value theorem,
\begin{align}
|\zeta_k(t)-\zeta_{k-\ell}(t)|\leq \exp\lf\{\delta\nu^{1/3}t(1+\max\{|k|^{2/3},|k-\ell|^{2/3}\})\rg\}\delta\nu^{1/3}t\lf||k|^{2/3}-|k-\ell|^{2/3}\rg|.
\end{align}Direct computation yields the following relations:
\begin{align}
\max\{|k|^{2/3},|k-\ell|^{2/3}\}\leq |k-\ell|^{2/3}+|\ell|^{2/3},\quad||k|^{2/3}-|k-\ell|^{2/3}|\leq |\ell|^{2/3}.
\end{align}Hence, 
\begin{align}|\zeta_k-\zeta_{k-\ell}|\leq \zeta_{k-\ell}\zeta_\ell \delta \nu^{1/3} |\ell|^{2/3}t.\label{zeta_com_ell}
\end{align}Combining this with the bootstrap assumption \eqref{Hypotheses_sh}, we have that
\begin{align}
T_2&\lesssim \nu^{1/3}t\lf\| A\sqrt{\frac{-\pa_t\W_I}{\W_I}}\om_\nq\rg\|_{L^2}\|A\sqrt{-\de_L}\oms_\nq\|_{L^2}\underbrace{\|A\oms\|_{L^2}}_{\lesssim\ep \nu^{1/3}}\lesssim \ep(\nu^{1/6}t )^2CK_I[\om_\nq]+\ep\nu\|A\sqrt{-\de_L}\oms_\nq\|_{L^2}^2. 
\end{align}
For $t\leq\nu^{-1/6}$, these two terms are bounded as in \eqref{E_12est_sh}. Hence, combining the $(T_1,T_2)$-decomposition and the estimates established above, we come to the conclusion that the $\mf E_1$-term is consistent with \eqref{E_12est_sh}. 

\noindent
{\bf Step \# 2:} We estimate the the $\mf E_2$ term for $t\leq \nu^{-1/6}$: 
\begin{align}
\mf E_2&\lesssim  \bigg| \sum_{k,\ell}\iint \zeta_k(\wt A(k,\eta)-\wt A(k-\ell,\eta-\xi))\lf(i\ell \wh{(B\de_t^{-1}\Omega_\nq)}(\ell,\xi) i(\eta-\xi-(k-\ell)t) \wh\Omega^*(\eta-\xi,k-\ell)\rg) \\
&\hspace{3cm}\times  \overline{A\wh\Omega^*(k,\eta)} d\eta d\xi\bigg|\\
&\quad +\sum_{k,\ell}\iint   |\zeta_k-\zeta_{k-\ell}| |\ell \wh{(B\de_t^{-1}\Omega_\nq)}(\ell,\xi) | \\
&\hspace{4cm}\times| \wt A(k-\ell,\eta-\xi) (\eta-\xi-(k-\ell)t) \wh{\Omega^*}(\eta-\xi,k-\ell)| \   |{A\wh\Omega^*(k,\eta)}| d\eta d\xi\\
&=:T_3+T_4.
\end{align} To treat the $T_3$ term, one applies the estimate \eqref{cm_sh_a} to obtain that 
\begin{align}|T_3|\lesssim(1+ t)\|A\de_L(B\de_t^{-1}\Omega_\nq)\|_{L^2} \|A\oms\|_{L^2}^2 \lesssim (1+t)\|A\om\|_{L^2}\|A\oms\|_{L^2}^2.
\end{align}
The estimate of the $T_4$ is similar to the estimate of $T_2$. We invoke the \eqref{zeta_com_ell} to obtain that \begin{align}
T_4\lesssim& \nu^{1/3}t\lf\| A\sqrt{\frac{-\pa_t\W_I}{\W_I}}\om_\nq\rg\|_{L^2}\lf(\|A\sqrt{-\de_L}\oms\|_{L^2}\|A\oms_\nq\|_{L^2}+\|A\sqrt{-\de_L}\oms_\nq\|_{L^2}\|A\oms\|_{L^2}\rg)\\
\lesssim &\ep(\nu^{1/6}t )^2CK_I[\om_\nq]+\ep\nu\|A\sqrt{-\de_L}\oms\|_{L^2}^2\lesssim\ep CK_I[\om_\nq]+\ep\nu\|A\sqrt{-\de_L}\oms\|_{L^2}^2,\quad\forall t\leq \nu^{-1/6}. 
\end{align}
This concludes the proof. }
\end{proof}
\begin{proof}[Proof of Lemma \ref{lem:NLa}] Combining the decomposition \eqref{E} and the Lemma \ref{lem:E_3_proof} and \ref{lem:E_4_proof} yields the short-time estimate \eqref{NLa:est:short}. Combining the decomposition \eqref{E} and the Lemma \ref{lem:E_1_proof}, \ref{lem:E_2_proof} and \ref{lem:E_3_proof} yields the long-time estimate \eqref{NLa:est:long}. 
\end{proof}

\subsection{The Estimate of $\mathfrak{NL}_{\mathfrak{0a}}$}
We define $\mathfrak{NL}_{\mathfrak{0a}}$ as 
\begin{align}
  \mathfrak{NL}_{\mathfrak{0a}} := \int  A\oms A (   \po U^1 \p_z \Omega^*) \, dV .
\end{align}
In this section, we estimate the $\mathfrak{NL}_{\mathfrak{0a}}$ terms in \eqref{Energy_Ev}.
\begin{lemma} For $T_{*} \in (0, \nu^{-1/6}]$ (or $T_{*}\in[\nu^{-1/6}, c_*\nu^{-1}]$, respectively), under the bootstrap assumptions of Proposition \ref{pro:short} (or Proposition \ref{pro:long}, respectively), the following estimate holds given that the viscosity $\nu$ and the parameter $\ep$ is chosen small enough,
    \begin{align}\label{NL_0a}
    \mathfrak{NL}_{\mathfrak{0a}}\lesssim \ep \nu^{1/3}\|A|\pa_z|^{1/3}\Omega^*_\nq\|_{L^2}^2. 
    \end{align}
    As a result of the bootstrap assumption \eqref{Hypotheses_sh} (\eqref{Hypotheses}, respectively), the following estimate holds for $\mathcal{I} = [0, T_*], [\nu^{-1/6}, T_*]$:
    \begin{align}
        \int_{\mathcal{I}} |\mathfrak{NL}_{\mathfrak{0a}}|dt\lesssim \ep^3\nu^{2/3}.
    \end{align}
\end{lemma}
\begin{proof}
We first observe there is a cancellation relation 
\begin{align}
\int B(\pa_v-t\pa_z) \de_t^{-1}\oms_{\mathbbm{o}} \ \pa_z A\oms_\nq  \ A \oms_\nq dV=\int B(\pa_v-t\pa_z)\de_t^{-1}\oms_{\mathbbm{o}} \pa_z\left(\frac{A \oms_\nq}{2}\right)^2dV=0.
\end{align}
Hence, we can rewrite the term $\mathfrak{NL}_{\mathfrak{0a}}$ as follows:
\begin{align}
\mathfrak{NL}_{\mathfrak{0a}} =&2\bigg|\int \bigg( A(-B\pa_v\de_t^{-1}\oms_{\mathbbm{o}}\  \pa_z \oms_\nq)+B\pa_v\de_t^{-1} \oms_{\mathbbm{o}}\ \pa_z A \oms_\nq \bigg) A \oms_\nq\  dV\bigg|\\
=&C\bigg|\sum_{k\neq0}\iint \left(\mf M(k,\eta)\lan k,\eta\ran^s-\mf M (k,\xi)\lan k,\xi\ran^s\right)\ \wh{(B\pa_v \de_t^{-1} \oms_{\mathbbm{o}})}(\eta-\xi)\\
&\quad\quad\quad\quad \quad\quad\times \left( k e^{\delta\nu^{1/3}(|k|^{2/3}+1)t}\wh  \oms(k,\xi)\right)\ \overline{{A \wh\oms(k,\eta)}} \ d\eta d\xi \bigg|.\label{commutator_form}
\end{align}
Now we invoke the commutator estimate \eqref{com_est} and the Sobolev product estimate to obtain that,
\begin{align} 
\mathfrak{NL}_{\mathfrak{0a}}&\lesssim\bigg|\sum_{k\neq0}\iint \left(\lan k,\xi\ran^s+\lan\eta-\xi\ran^s \right)|\eta-\xi||\wh{(B\pa_v \de_t^{-1} \oms_{\mathbbm{o}})}(\eta-\xi)|\left|  e^{\delta\nu^{1/3}(|k|^{2/3}+1) t }\wh  \oms( k,\xi)\right|\\
&\hspace{2cm}\times | {A \wh\oms( k,\eta)}|d\eta d\xi \bigg|\\
  &\lesssim \|\lan \pa_v\ran^s \pa_v(B\pa_v \de_t^{-1} \oms_{\mathbbm{o}})\|_{L^2}\|A\oms_\nq\|_{L^2}^2 \lesssim \lf\|\lan \pa_v\ran^s \lf( {B}^ {-1} \underbrace{(B\pa_v)^2 \de_t^{-1} }_{=\mathbbm{1}}\oms_{\mathbbm{o}}\rg)\rg\|_{L^2}\|A\oms_\nq\|_{L^2}^2\\ & \lesssim \|\lan \pa_v\ran^s  (B^{-1}\oms_{\mathbbm{o}})\|_{L^2}\|A\oms_\nq\|_{L^2}^2 .
\end{align}

Now we apply the Kato-Ponce inequality \eqref{KP} \myb{ and a variant of its corollary \eqref{KP_2}} to obtain that
\begin{align}
\mathfrak{NL}_{\mathfrak{0a}}\lesssim (\|B^{-1}\|_{L^\infty}+\|B\|_{L^\infty}+\|B\pa_vB\|_{H^s})\|\lan\pa_v\ran^s\oms_{\mathbbm{o}}\|_{L^2}\|A\oms_\nq\|_{L^2}^2.
\end{align}
 Thanks to the bootstrap assumption  $\|\om_{\mathbbm{o}}^*\|_{L^\infty_t H^s} \lesssim \epsilon \nu^{1/3}$
 , this term is controlled as follows
\begin{align}
\mathfrak {NL}_{\mathfrak{0a}} \leq  C\ep \nu^{1/3}\|A|\pa_z|^{1/3} \Omega^*_\nq\|_2^2.\label{T_om_nq_12}
\end{align}
This concludes the proof of the lemma. 
\end{proof}


\subsection{The Estimate of $\mathfrak{NL}_{\mathfrak{l}}$}
We define $\mathfrak{NL}_{\mathfrak{l}}$ as 
\begin{align}
    \mathfrak{NL}_{\mathfrak{l}} := \int A\oms \ A\lf(B \na_L^\perp \Delta_t^{-1} \pn \Omega \cdot \na_L F \rg)\, dV.
\end{align}

In this section, we prove the following lemma.
\begin{lemma}\label{lem:NLl}
    \noindent {\bf a) Short time:} For $T_{*} \in (0, \nu^{-1/6}]$, under the bootstrap assumption of Proposition \ref{pro:short}, the following estimate holds as long as $\nu$ is chosen small enough: 
        \begin{align}
        \label{NLl:est:short}
           \int_0^{T_*}|\mathfrak{NL}_{\mathfrak{l}}|\, dt  &\lesssim   \nu^{-1/6} \norm{A \Omega}_{L_t^{\infty}L_{}^2} \norm{A \p_zF }_{L_t^\infty L_{}^2}  \norm{A \Omega^*}_{L_t^{\infty}L_{}^{2}} \\
    &+ \nu^{-1/3}\norm{A \Omega}_{L_t^{\infty}L_{}^2}
    \norm{\lb t\rb^{-1} A\nabla_L F}_{L_t^{\infty}L_{}^2}\norm{A \Omega^*}_{L_t^{\infty}L_{}^2}.\notag
        \end{align}
        
        \noindent 
        {\bf b) Long time: }For $T_*\in [\nu^{-1/6}, c_*\nu^{-1}]$
        under the bootstrap assumption of Proposition \ref{pro:long}, the following estimate holds as long as $\nu$ is chosen small enough:
\begin{align}
\label{NLl:est:long}
 \int_{\nu^{-1/6}}^{T_*}|\mathfrak{NL}_{\mathfrak{l}}| dt &\lesssim  \norm{A \pn \Omega}_{L_t^2L_{}^2} \norm{A \p_z F}_{L_t^2L_{}^2}\norm{A\Omega^*}_{L_t^{\infty}L_{}^2}\\
    &+ \norm{A \sqrt{ - \frac{\p_t \mathcal{W}_I}{\mathcal{W}_I} } \pn \Omega}_{L_t^2L_{}^2}  \nu^{1/6} \norm{A \nabla_L F}_{L_t^2L_{}^2} \norm{A\Omega^*}_{L_t^{\infty}L_{}^2}  \notag\\    &+ \norm{A \sqrt{ - \frac{\p_t \mathcal{W}_I}{\mathcal{W}_I} } \pn \Omega}_{L_t^2L_{}^2}\nu^{-1/3}\norm{AF}_{L_t^{\infty}L_{}^2}  \norm{A \sqrt{ - \frac{\p_t \mathcal{W}_E}{\mathcal{W}_E} } \Omega^*}_{L_t^2L_{}^2}. \notag
        \end{align}

Hence, as a result of the bootstrap assumption \eqref{Hypotheses_sh} (\eqref{Hypotheses}, respectively), the following estimate holds for $\mathcal{I} = [0, T_*], [\nu^{-1/6}, T_*]$:
\begin{align}
    \int_{\mathcal{I}} |\mathfrak{NL}_{\mathfrak{l}}| \, dt \lesssim \epsilon^3 \nu^{2/3}.
\end{align}

\end{lemma}

\begin{proof}
    
We decompose the $\mathfrak{NL}_{\mathfrak{l}}$ term in \eqref{Energy_Ev} as follows
\begin{align}
    \mathfrak{NL}_{\mathfrak{l}} &= \lb  A(B \na_L^\perp \Delta_t^{-1} \pn \Omega \cdot \na_L F), A \Omega^* \rb = \sum_{j = 1}^2 \mathfrak{E}_{lj} ,\notag\\
    \mathfrak{E}_{l1} &:= -\lb  A(B (\p_v -t \p_z) \Delta_t^{-1} \pn \Omega  \p_z F), A \Omega^* \rb ,\notag \\
    \mathfrak{E}_{l2} &:= \lb  A(B  \p_z\Delta_t^{-1} \pn \Omega  (\p_v -t \p_z) F), A \Omega^* \rb .\notag
\end{align}
To bound $\mathfrak{E}_{l1}$ for short time, we use \eqref{A_product_rule_Hs} and  Lemma \ref{ellip:perturb} to obtain
\begin{align}
    \label{El1:short}
    &\int_0^{T_*}|\lb  A(B (\p_v -t \p_z) \Delta_t^{-1} \pn \Omega  \p_z F), A \Omega^* \rb|\, dt  \\
    &\lesssim  
    \int_0^{T_*} \norm{A B (\p_v - t\p_z) \Delta_t^{-1} \pn \Omega }_{L_{}^2} \norm{A \p_z 
    F}_{L_{}^2} \norm{A\Omega^*}_{L_{}^2} \, dt \notag\\
    &\lesssim \nu^{-1/6} \norm{A \Omega}_{L_t^{\infty}L_{}^2} \norm{A \p_zF }_{L_t^\infty L_{}^2}  \norm{A \Omega^*}_{L_t^{\infty}L_{}^{2}}
    \notag. 
     \end{align}
For the long time estimate, we again use  \eqref{A_product_rule_Hs}, Lemma \ref{ellip:perturb},  Proposition \ref{linCK}, and Cauchy-Schwarz in time to get
\begin{align}
\label{El1:long}
    &\int_{\nu^{-1/6}}^{T_*} |\lb  A(B (\p_v -t \p_z) \Delta_t^{-1} \pn \Omega  \p_z F), A \Omega^* \rb|\, dt  \\
    &\lesssim  
    \int_{\nu^{-1/6}}^{T_*} \norm{A B (\p_v - t\p_z) \Delta_t^{-1} \pn \Omega }_{L_{}^2} \norm{A \p_z 
    F}_{L_{}^2} \norm{A\Omega^*}_{L_{}^2} \, dt \notag\\
    &\lesssim   \norm{A \pn \Omega}_{L_t^2L_{}^2} \norm{A \p_z F}_{L_t^2L_{}^2}\norm{A\Omega^*}_{L_t^{\infty}L_{}^2}. \notag
\end{align}
Now we bound $\mathfrak{E}_{l2}$. For the short time estimate, \eqref{A_product_rule_Hs} and Lemma 
 \ref{ellip:perturb} yield
\begin{align}
    \label{EL2:short}
    &\int_0^{T_*} |\lb  A(B  \p_z\Delta_t^{-1} \pn \Omega  (\p_v -t \p_z) F), A \Omega^* \rb| \, dt \\
    &\lesssim \int_0^{T_*} \norm{A (B \p_z \Delta_t^{-1} \Omega) }_{L_{}^2} \norm{A(\p_v -t \p_z) F}_{L_{}^2 } \norm{ A \Omega^*}_{L_{}^2} \, dt \notag \\
    &\lesssim  \nu^{-1/3} \norm{A\Omega}_{L_t^{\infty}L_{}^2} \norm{\lb t\rb^{-1} A(\p_v -t \p_z) F}_{L_t^{\infty}L_{}^2}
    \norm{A \Omega^*}_{L_t^{\infty}L_{}^2}\int_0^{T_*} \lb t\rb \, dt
    \notag \\
    &\lesssim 
    \nu^{-1/3}\norm{A \Omega}_{L_t^{\infty}L_{}^2}
    \norm{\lb t\rb^{-1} A\nabla_L F}_{L_t^{\infty}L_{}^2}\norm{A \Omega^*}_{L_t^{\infty}L_{}^2}.  \notag
\end{align}
For the long time we estimate this term as in Lemma \ref{lem:E_2_proof}:
\begin{align}
\label{El2:long}
    &\int_{\nu^{-1/6}}^{T_*} |\lb  A(B  \p_z\Delta_t^{-1} \pn \Omega  (\p_v -t \p_z) F), A \Omega^* \rb| \, dt \\
    &\lesssim  
\norm{A \sqrt{ - \frac{\p_t \mathcal{W}_I}{\mathcal{W}_I} } \pn \Omega}_{L_t^2L_{}^2}  \nu^{1/6} \norm{A \nabla_L F}_{L_t^2L_{}^2} \norm{A\Omega^*}_{L_t^{\infty}L_{}^2}  \notag\\    
&+ \norm{A \sqrt{ - \frac{\p_t \mathcal{W}_I}{\mathcal{W}_I} } \pn \Omega}_{L_t^2L_{}^2}\nu^{-1/3}\norm{AF}_{L_t^{\infty}L_{}^2}  \norm{A \sqrt{ - \frac{\p_t \mathcal{W}_E}{\mathcal{W}_E} } \Omega^*}_{L_t^2L_{}^2}.   
    \notag
\end{align}
To conclude, we combine  \eqref{El1:short} and \eqref{EL2:short} to obtain  for the short time and long time estimate respectively
\begin{align}
    \int_0^{T_*} |\mathfrak{NL}_{\mathfrak{l}}| \, dt &\lesssim 
    \nu^{-1/6} \norm{A \Omega}_{L_t^{\infty}L_{}^2} \norm{A \p_zF }_{L_t^\infty L_{}^2}  \norm{A \Omega^*}_{L_t^{\infty}L_{}^{2}} \\
    &+ \nu^{-1/3}\norm{A \Omega}_{L_t^{\infty}L_{}^2}
    \norm{\lb t\rb^{-1} A\nabla_L F}_{L_t^{\infty}L_{}^2}\norm{A \Omega^*}_{L_t^{\infty}L_{}^2},  \notag
\end{align}
and 
\begin{align}
    &\int_{\nu^{-1/6}}^{T_*} |\mathfrak{NL}_{\mathfrak{l}}| \, dt \lesssim  \norm{A \pn \Omega}_{L_t^2L_{}^2} \norm{A \p_z F}_{L_t^2L_{}^2}\norm{A\Omega^*}_{L_t^{\infty}L_{}^2}\\
    &+ \norm{A \sqrt{ - \frac{\p_t \mathcal{W}_I}{\mathcal{W}_I} } \pn \Omega}_{L_t^2L_{}^2}  \nu^{1/6} \norm{A \nabla_L F}_{L_t^2L_{}^2} \norm{A\Omega^*}_{L_t^{\infty}L_{}^2}  \notag\\    &+ \norm{A \sqrt{ - \frac{\p_t \mathcal{W}_I}{\mathcal{W}_I} } \pn \Omega}_{L_t^2L_{}^2}\nu^{-1/3}\norm{AF}_{L_t^{\infty}L_{}^2}  \norm{A \sqrt{ - \frac{\p_t \mathcal{W}_E}{\mathcal{W}_E} } \Omega^*}_{L_t^2L_{}^2}, \notag
\end{align}
which concludes the proof. 
\end{proof}

\subsection{The Estimates of $\mathfrak{L}_{1}$ - $\mathfrak{L}_{3}$}
We defined $\mathfrak{L}_{1}$ - $\mathfrak{L}_{3}$ as follows
\begin{align}
    \mathfrak{L}_1 :=& \int A( (B' - B_0')   \Delta_t^{-1}\p_z\pn (F + \Omega^*)  )\, A \pn \Omega^* \, dV,\\
    \mathfrak{L}_2 := & \int A(  B_0'(\Delta_t^{-1} - \Delta_0^{-1}) \p_z\pn (F + \Omega^*))\, A\Omega^*\, dV,\\
    \mathfrak{L}_3 :=  &\int A\nu ( \Delta_{0}   - \tilde{\Delta}_t)  F A  \Omega^* \, dV . 
\end{align}
    
We have the following lemma.

        


\begin{lemma}
\label{lem:linear:per}
 For $T_{*} \in (0, \nu^{-1/6}]$ (or $T_{*}\in[\nu^{-1/6}, c_*\nu^{-1}]$, respectively), under the bootstrap assumptions of Proposition \ref{pro:short} (or Proposition \ref{pro:long}, respectively), the following estimate holds on the interval $\mathcal{I} = [0, T_*]$, $[\nu^{-1/6}, T_*]$ for $\nu$, $c_*$, and $\ep$ chosen small enough,
 \begin{align}
\int_{ \mathcal{I}} |\mathfrak{L}_1|+ | \mathfrak{L}_2|+ | \mathfrak{L}_3|dt\lesssim  \int_{\mathcal{ I}}\gamma CK_I[\oms]+\gamma\nu\|A\sqrt{-\de_L}\oms\|_{L^2}^2+C_\gamma\nu\|A\pn\oms\|_{L^2}^2\,dt.
 \end{align}
Hence, as a result of the bootstrap assumption \eqref{Hypotheses_sh} (\eqref{Hypotheses}, respectively), the following estimate holds
\begin{align}
    \int_{\mathcal{I}} |\mathfrak{L}_1|+ | \mathfrak{L}_2|+ | \mathfrak{L}_3|\, dt \lesssim (\gamma + C_\gamma \nu^{2/3} )\epsilon^2 \nu^{2/3}.
\end{align}
\myb{\myr{HS}: Are we starting from $0$? Or do we need to distinguish between $0,$ and $\nu^{-1/6}?$} 
\end{lemma}
\begin{proof}
For $\mathfrak{L}_1 - \mathfrak{L}_3$, we will need to use the smallness of the coefficients appearing in these terms. In what follows $\gamma \in (0,1)$ will be a free parameter. We will fix this parameter when we close the bootstrap in Section \ref{sec:bootstrap} by choosing $c_*$ sufficiently small to guarantee the needed smallness for $\gamma$. We will only consider the case where $T_* \in [\nu^{-1/6}, c_*\nu^{-1}]$, since the other case is easier.
We begin with $\mathfrak{L}_1$:
\begin{align}
\label{linear1}
    \int_{\nu^{-1/6}}^{T_*}\mathfrak{L}_1 \, dt =&\int_{\nu^{-1/6}}^{T_*}\int  A( (B' - B_0')  \p_z \Delta_t^{-1}\pn (F + \Omega^*)  ) A \pn \Omega^*  \, dV dt \\
     &= \int_{\nu^{-1/6}}^{T_*} \int  (-\Delta_L)^{1/2} A( (B' - B_0')  \p_z \Delta_t^{-1}\pn (F + \Omega^*)  ) (-\Delta)^{-1/2}A \pn \Omega^* \,dV dt \notag\\
     &\lesssim \gamma \norm{ A \sqrt{- \frac{\p_t \mathcal{W}_I}{\mathcal{W}_I}} \Omega^*}_{L_t^2L_{}^2}^2 \notag
\end{align}
where in the last line we used Lemma \ref{ellip:perturb} and \eqref{F:CK}.
We can bound $\mathfrak{L}_2$ exactly like $\mathfrak{L}_1$:
\begin{align}
\label{linear2}
   \mathfrak{L}_2 =  \int_{\nu^{-1/6}}^{T_*}\int | A(  B_0'(\Delta_t^{-1} - \Delta_0^{-1}) \pn \p_z(F + \Omega^*)) A\Omega^* |\, dV dt \lesssim \gamma \norm{ A \sqrt{- \frac{\p_t \mathcal{W}_I}{\mathcal{W}_I}} 
  \Omega^*}_{L_t^2L_{}^2}^2.
\end{align}
Finally,  \eqref{B:difference} and \eqref{F:CK}  imply 
\begin{align}
\label{linear3}
   \mathfrak{L}_3 =&\int_{\nu^{-1/6}}^{T_*} \int  A\nu (\tilde{\Delta}_t  - \Delta_{0})  F A  \Omega^* \, dV dt  =  \int_{\nu^{-1/6}}^{T_*} \int  A\nu (\tilde{\Delta}_t  - \Delta_{0}) \pn F A \pn  \Omega^*  \, dV dt\\
    \lesssim 
    &\gamma \nu \int_{\nu^{-1/6}}^{T_*}\norm{\nabla_L A\pn F}_{L_{}^2} \norm{A\nabla_L \pn\Omega^*}_{L_{}^2} \, dt  + \nu \int_{\nu^{-1/6}}^{T_*} \norm{A \p_v\pn F  }_{L_{}^2} \norm{A \pn\Omega^*}_{L_{}^2} \, dt \\
    &\lesssim \gamma \norm{ A \sqrt{- \frac{\p_t \mathcal{W}_I}{\mathcal{W}_I}}\pn\Omega^*}_{L_t^2L_{}^2}^2  + \gamma \nu \norm{A \nabla_L \pn\Omega^* }_{L_t^2L_{}^2}^2   +  C_\gamma \nu \norm{A \pn \Omega^*}_{L_t^2L_{}^2}^2, 
\end{align}
which concludes the proof.
\end{proof}

\subsection{The Estimate of $\po U^1$}
\label{u1:0:mode}
We prove 
\begin{lemma}
    \label{lem:u1:0mode}
   There exists $\epsilon_0, \nu_0, c_* > 0$ such that for $\epsilon \in (0, \epsilon_0), \nu \in (0, \nu_0)$
\begin{align}
        \label{low_freq_U1}
           \sup_{t \in [0, c_* \nu^{-1}]} \norm{\po U^1(t)}_{L_v^2} \leq     4\epsilon^2 \nu^{2/3}.
        \end{align}
\end{lemma}

\begin{proof}
The proof for the short-time is contained in the proof of the long-time bootstrap so we focus on the latter. The proof proceeds similarly to \cite{BMV16} section 8. From \eqref{velocity:shear}, in the original $(x,y)$ variables, we have
\begin{align}
    \p_t \po u_1 + \po ( \pn u \cdot \nabla \pn u_1  ) - \nu \p_y^2 \po u_1 = 0.
\end{align}
Using the change of variables \eqref{chg_of_v}
gives 
\begin{align}
\label{0mode1}
  \p_t \po U^{(1)}  - \nu B^2 \p_v^2 \po U^{(1)}  = - \po ( B\nabla^{\perp} \pn \Delta_t^{-1} \Omega \cdot \nabla \pn( 
   (\p_v - t\p_z)\Delta_t^{-1} \Omega    ))
\end{align}
where we have used that $\nabla_L^{\perp} f \cdot \nabla_L g = \nabla^{\perp} f \cdot \nabla g$. 
Testing  \eqref{0mode1} with $\po U^{(1)}$ yields 
\begin{align}
\label{U1:energy}
    &\frac{1}{2}\frac{d}{dt}\norm{\po U^{(1)}  }_{L_v^2}^2 - \nu \int  B^2 \p_v^2 \po U^{(1)}, \po U^{(1)}  dV \\  
   &=- \int    \nabla^{\perp} \pn \Delta_t^{-1} \Omega \cdot \nabla \pn( 
   (\p_v - t \p_z)\Delta_t^{-1} \Omega    )  \po U^{(1)} dV \notag.
\end{align}
First, we write
\begin{align}
    \nabla^{\perp} \pn \Delta_t^{-1} \Omega \cdot \nabla \pn( 
   (\p_v - t \p_z)\Delta_t^{-1} \Omega    )  = &-\p_v \pn \Delta_t^{-1} \Omega \p_z \pn 
   (\p_v - t \p_z)\Delta_t^{-1} \Omega \\
   &+\p_z \pn \Delta_t^{-1} \Omega \p_v \pn (\p_v - t\p_z) \Delta_t^{-1} \Omega.
\end{align}
Using Holder's inequality and the embedding $H^{1/2}(\TT \times \R) \to L^4(\TT \times \R) $, we can bound the first term as  
\begin{align}
     &\int  |  \nabla^{\perp} \pn \Delta_t^{-1} \Omega \cdot \nabla \pn( 
   (\p_v - t \p_z)\Delta_t^{-1} \Omega    )  \po U^{(1)} | \, dV \\ 
   &\lesssim \norm{\p_v \pn \Delta_t^{-1} \Omega}_{H^{1/2}} \norm{\p_z \pn 
   (\p_v - t \p_z)\Delta_t^{-1} \Omega}_{H^{1/2}} \norm{\po U^{(1)}}_{L_v^2} \\ 
   &+\norm{\p_z \pn \Delta_t^{-1} \Omega}_{H^{1/2}} \norm{\p_v \pn 
   (\p_v - t \p_z)\Delta_t^{-1} \Omega}_{H^{1/2}} \norm{\po U^{(1)}}_{L_v^2} 
\end{align}

Using Lemma \ref{ellip:perturb} and the inviscid damping inequalities for $s \geq 0$, $\beta \in [0,2], \alpha \in [0,1]$ 
\begin{align}
    \norm{\Delta_L^{-1} \pn f }_{H^s} &\lesssim \frac{1}{\lb t \rb^\beta} \norm{ |\p_z|^{-\beta} \pn f}_{H^{s + \beta } } \\
    \norm{\Delta_L^{-1/2} \pn f }_{H^s} &\lesssim \frac{1}{\lb t \rb^\alpha} \norm{ |\p_z|^{-\alpha} \pn f}_{H^{s + \alpha }}
\end{align}
we have 
\begin{align}
    \norm{\p_v \pn 
   (\p_v - t \p_z)\Delta_t^{-1} \Omega}_{H^{1/2}} +  \norm{\p_v \pn \Delta_t^{-1} \Omega}_{H^{1/2}} &\lesssim \frac{1}{\lb t\rb^{1/2}} \norm{\pn \Omega}_{H^2}\\
   \norm{\p_z \pn \Delta_t^{-1} \Omega}_{H^{1/2}}   + \norm{\p_z \pn 
   (\p_v - t \p_z)\Delta_t^{-1} \Omega}_{H^{1/2}} 
   &\lesssim \frac{1}{\lb t\rb }\norm{\pn \Omega}_{H^2}.
\end{align}
Therefore, 
\begin{align}
\label{NLb1}
      \int \nabla^{\perp} \pn \Delta_t^{-1} \Omega \cdot \nabla \pn( 
   (\p_v - t \p_z)\Delta_t^{-1} \Omega    )  \po U^{(1)} dV 
   &\lesssim 
   \frac{1}{\lb t \rb^{3/2}} \norm{\pn\Omega}_{H^2}^2 \norm{\po U^{(1)}}_{L_v^2}.
   \end{align}
The bootstrap assumptions  \eqref{Hypotheses_sh} and \eqref{Hypotheses} and Proposition \ref{linCK} imply that   \eqref{NLb1} leads to an integrable forcing term. Next, we deal with the dissipation. Observe that 
\begin{equation}
  - \int  B^2 \p_v^2 \po U^{(1)} \po U^{(1)} dV  =   \norm{B\p_v \po U^{(1)} }_{L_v^2}^2 -  2 \int B\p_vB \p_v\po U^{1} \po U^{(1)} dV.   
  \end{equation}
For $\mu \in (0,1)$ we can bound the last term 
as 
\begin{align}
  \int B\p_vB \p_v\po U^{1} \po U^{(1)} dV
  &\lesssim   \norm{B\p_vB}_{L_v^\infty} \norm{\p_v \po U^{(1)}}_{L_v^2} \norm{\po U^{(1)}}_{L_v^2}  \\
  & \lesssim \mu \norm{B\p_v \po U^{(1)}}_{L_v^2}^2  + C_{\mu} \norm{\p_vB}_{L_v^\infty}^2\norm{\po U^{(1)}}_{L_v^2}^2.  
\end{align}
Taking $\mu$ sufficiently small allows us to write \eqref{U1:energy} as 
\begin{align}
     \frac{1}{2} \frac{d}{dt} \norm{\po U^{(1)}}_{L_v^2}^2 \lesssim \nu \norm{\po U^{(1)}}_{L_v^2}^2 + \frac{1}{\lb t \rb^{3/2}} \norm{\pn\Omega}_{H^2}^2 \norm{\po U^{(1)}}_{L_v^2}.
\end{align}
Using the bootstrap assumptions \eqref{Hypotheses_sh} and \eqref{Hypotheses} then imply \eqref{low_freq_U1} by choosing $c_*$ and $\epsilon$ sufficiently small. 
\end{proof}

\subsection{Proof of the Proposition \ref{pro:short} and Proposition \ref{pro:long}}
\label{sec:bootstrap}

We first choose the parameter $K$ in \eqref{multipliers} large enough depending on the shear profile derivative $B$, such that the estimate in \eqref{D_est} holds. Then we choose $c_*$ sufficiently small to ensure \eqref{low_freq_U1} and Lemma \ref{lem:linear:per} hold. Finally, we choose $\epsilon, \nu$ sufficiently small depending on the bootstrap constants and various implicit constants in Section \ref{sec:energy} to conclude the proof of Proposition \ref{pro:short} and \ref{pro:long}. Theorem \ref{main} then follows.\myb{?}
\section{Analysis of the Linear Equation}
\label{sec:linear}
Taking a Fourier transform in $z$ of \eqref{Flin} yields  \begin{align}
\label{F_k}
    &\p_t F_k - ikB_0'  \Delta_{0}^{-1} F_k  - \nu \Delta_{0} F = ik B_0' \Delta_{0}^{-1}  \Omega_{k}^* , \quad F_k(0, v) = 0. 
\end{align}

Proposition \ref{linCK} will be a consequence of the following proposition which gives control on the non-local term $ik B_0'\Delta_0^{-1} F_k$ in terms of the 
 the auxiliary profile $\Omega_k^*$.
\begin{proposition}
\label{prop:linear:bounds:k}
Assume that $k \in \mathbb{Z} \setminus\{ 0\}$ and that $F_k$ satisfies \eqref{F_k}. 
Then there exists $\nu_0 > 0$ such for $\nu \in (0, \nu_0)$, the following bounds hold:
\begin{subequations}
\label{Fbounds}
    \begin{align}
    \label{linpro3}
         \| A   F_k(t)\|_{L_t^{\infty}[0, T]L_v^2} &\lesssim \lf\|A   (-\Delta_{L})^{-1/2}\Omega_k^* \rg\|_{L_t^2[0, T ]L_v^2}, \\
    \label{linpro3.4}
\norm{ \lb t \rb^{-1} A(-\Delta_L)^{1/2}F_k(t)}_{L_t^{\infty}[0, T]L_v^2} &\lesssim \norm{A  (-\Delta_L)^{-1/2}\Omega_k^*}_{L_t^2[0,T]L_v^{2}}, \\
   \label{linpro3.5}
    \norm{k(-\Delta_L)^{-1/2} AF_k }_{L_t^2[0,T]L_v^2} &\lesssim   \lf\|A (-\Delta_L)^{-1/2}\Omega_k^* \rg\|_{L_t^2[0, T ]L_v^2}. 
  \end{align}
\end{subequations}
\end{proposition}
Recall that for $k \neq 0$, $\p_t \mathcal{W}_I(t, k ,\eta) \approx \frac{k^2}{k^2 + (\eta -kt)^2} $, so we can rephrase \eqref{Fbounds} in terms of  $\mathcal{W}_I$ as 
\begin{align}
\label{F:WI}
    \norm{ k A F_k}_{L_t^{\infty}[0,T] L_v^2} +  
    \norm{ \lb t \rb^{-1} k A F_k}_{L_t^{\infty}[0,T] L_v^2} 
    + \norm{ k A \sqrt{  -\frac{\p_t  \mathcal{W}_I   }{ \mathcal{W}_I    }     } F_k }_{L_t^{2}[0,T] L_v^2} \lesssim  \norm{A \sqrt{  -\frac{\p_t  \mathcal{W}_I   }{ \mathcal{W}_I    }     } \Omega_k^* }_{L_t^2[0,T]L_v^2  }.
\end{align}
We now  prove Proposition  \ref{linCK} assuming Proposition \ref{prop:linear:bounds:k}.
\begin{proof}
The bound for $\norm{ \lb t \rb^{-1} A  |\p_z| (-\Delta_L)^{1/2} F }_{L_t^{\infty}[0, T]L_{}^2  }   $ is a direct consequence of \eqref{linpro3.4} so we focus on the remaining terms. We rewrite \eqref{Flin} as 
\begin{align}
\label{F:forced}
    \p_t F   - \nu \Delta_{0} F = B_0' \p_z \Delta_{0}^{-1} \pn ( F + \Omega^*).  
\end{align}
The inner product  of the the right hand side of \eqref{F:forced} with $A F$
can be bounded as   
\begin{align}
\label{F:rhs}
    & \int A   F \, A[B' \Delta_0^{-1}  (F+ \Omega^*  )] \, dV     = \int  A  \p_z (-\Delta_L)^{-1/2} F \,  A(-\Delta_L)^{1/2}[B'  \Delta_0^{-1}  ( F + \Omega^*  )] \, dV \\
    &\leq \norm{ A \p_z (-\Delta_L)^{-1/2} F}_{L_{}^2} \norm{A (-\Delta_L)^{1/2}   [B'  \Delta_0^{-1}  (F + \Omega^* )] }_{L_{}^2}  . \notag 
\end{align}
As a consequence of  Lemma \ref{lem:ellip} and \eqref{linpro3.5}, time integrating \eqref{F:rhs} yields 
\begin{align}
 \int_0^T  \int A  F \, A\p_z[B' \Delta_0^{-1}  (F+ \Omega^*  )] \, dV  dt  &\lesssim  \norm{A (-\Delta_L)^{-1/2} \Omega^*}_{L^2[0,T]L_{}^2 }^2.
\end{align}
From \eqref{F:forced} we compute the time evolution of $\norm{A |\p_z| F }_{L_{}^2}^2$ which gives 
\begin{align}
\label{lin:energy}
    &\frac{1}{2} \frac{d}{dt} \norm{A|\p_z| F}_{L_{}^2}^2 + \sum_{\iota \in \nu, I, E}CK_{\iota}[A|\p_z| F]        \\ 
    &= \int A  |\p_z| F A\p_z(B' \Delta_0)^{-1}  (\Omega^* + F) dV  + \int A \Delta_0 |\p_z| F A|\p_z|F d V.   \notag 
\end{align}
Following the argument for $D_{\neq}$ in Lemma \ref{lem:Diffusion} and Lemma \ref{lem:Con_Low_reg_1} combined with 
\eqref{F:rhs} implies the desired result upon time integrating \eqref{lin:energy}. 
\end{proof}
 The rest of this section will be dedicated to Proposition \ref{prop:linear:bounds:k}.

 A key step in proving 
 this Proposition  is 
 having a representation formula for $F_k$. The following result follows from \cite{J23} Proposition 1.2.
 \begin{lemma}
 \label{lem:rep}
    Assume that $b(y)$ satisfies Assumption \ref{assum:b},
      $k \in \mathbb{Z} \cap [1, \infty)$, $f_k(t,v): [0, \infty) \times \R \to \mathbb{C}$ satisfies $f_k \in C^{\infty}((0, \infty) \times \R )\cap C([0, \infty), L^2(\R)  )$, with initial data $f_k(0, v) = f_{0k}(v) \in C_0^{\infty}(\R)$, and  $f_k$ solves
         \begin{align}
         &\p_t f_k + ivk f_k -ik B_0'\Delta_{B_0}^{-1}  f_k  - \nu \Delta_{B_0} f_k = 0, \quad f_k(0, v) = f_{0k}, \\
         \label{db}
         &\Delta_{B_0} := B_0^2 \p_v^2 - B_0'\p_v - k^2. 
         \end{align}  
     Then there exists $\nu_0 , \delta_{lin} >0$ such that for $\nu \in (0, \nu_0)$ we  have for $t > 0$
   \begin{align}
   \label{fkrep}
       f_k(t,v) = -\frac{1}{2 \pi}  e^{- \delta_{lin} \nu^{1/3} |k|^{2/3} t - \nu k^2 t }  \int_\R e^{-ikwt} \Omega_{k, \nu} (v,w) \, dw ,
   \end{align}
     where $\Omega_{k,\nu}$ satisfies for $v, w\in \R, \epsilon := \nu/k$,
     \begin{align}
     \label{Oss0}
         \epsilon \Delta_{B_0} \Omega_{k, \nu} (v,w) + \delta_{lin} \epsilon^{1/3} \Omega_{k, \nu}(v,w) -i(v-w) \Omega_{k, \nu} + B_0'(v) \Delta_{B_0}^{-1} \Omega_{k,\nu}(v,w)  = f_{0k}(v).
     \end{align}
 \end{lemma}
\begin{remark}
    In this section and this section only, we abuse notation and let $\epsilon = \nu/k$. There will be no risk of confusion with $\epsilon$ as the small prefactor for the size of the initial data since the linear analysis is independent of the size of the data.
\end{remark}
 \begin{remark}
     The assumptions on the smoothness and compact support are qualitative and can be removed by an approximation argument. 
 \end{remark}
 \begin{remark}
    The key difference between Lemma \ref{lem:rep} and Proposition 1.2 in \cite{J23} is the factor of $e^{-\delta_{lin} \nu^{1/3}|k|^{2/3}}t$ in \eqref{fkrep} and $\delta_{lin} \epsilon^{1/3}$ in \eqref{Oss0}. However, such an extension follows using the methods in \cite{J23} and choosing $\delta_{lin}$ sufficiently small.  
 \end{remark}
With Lemma \ref{lem:rep} we can now state our representation formula for $F_k$. 
 \begin{lemma}
   Assume that   $k \in \mathbb{Z} \cap [1, \infty)$ and  let $F_k$ be a solution of \eqref{F_k}. There exists $\nu_0, \delta_{lin} > 0$ such that for $\nu \in (0, \nu_0)$ and    $t > 0$ we have 
\begin{align}
    \label{rep:for:duh}
     F_k(t,v) =  -\frac{1}{2\pi } \int_0^t  e^{- (\delta_{lin} \nu^{1/3} |k|^{2/3} + \nu k^2    ) (t - \tau)  } \int_\R e^{ikt(v-w)  } 
     \Omega_{k, \nu, \tau}(v,w)  \, d w\,d\tau  
   \end{align}
   where 
 $\Omega_{k, \nu, \tau }(v,w)$ is defined as the solution to 
\begin{align}
 \label{Oss:shift}
   &\epsilon \Delta_{B_0} \Omega_{k, \nu ,\tau} (v,w) + \delta_{lin} \epsilon^{1/3} \Omega_{k, \nu, \tau} - i(v-w)\Omega_{k, \nu ,\tau }(v,w) + ik B_0' \Delta_{B_0}^{-1} \Omega_{k, \nu, \tau }(v,w) \\
   &=  ikB_0'e^{-ik\tau (v-w)} \Delta_{0}^{-1}  \Omega_k^*(\tau, v).  \notag
     \end{align}
 \end{lemma}
\begin{remark}
    Because $F(t, z, v)$ is real-valued, we can restrict to only considering $k \geq 1$ by noting that $F_{-k}(t, v) = \overline{F_{k}}(t,v)$. 
\end{remark}
\begin{proof}
Observe that if we define $f_k$  as 
\begin{align}
    f_k(t,v) := e^{-iktv}F_k(t,v)
\end{align}
then $f_k$ satisfies 
\begin{align}
\label{forced_lin}
    &\p_t f_k + iktv f_k - ikB_0' \Delta_{B_0}^{-1}f_k  - \nu \Delta_{B_0} f_k = e^{-iktv }ik B_0' \Delta_{0 }^{-1}  \Omega_k^* , \quad f_k(0, v) = 0.   
\end{align}
Using the representation formula from Lemma \ref{lem:rep} and Duhamel's formula then implies  \eqref{rep:for:duh}. 
\end{proof}

In order to better analyze the degeneracy at $v =w $ in \eqref{Oss:shift}, we consider the new unknown $\Upsilon_{k, \tau}(v, w) := \Omega_{k, \nu, \tau } (v + w, w)$ which satisfies

\begin{align}
\label{ups:oss}
&\epsilon (B_0(v + w))^2 \p_v^2 \Upsilon_{k, \tau} (v,w) + \epsilon  B_0'(v + w) \p_v \Upsilon_{k, \tau}(v, w) + \delta_{lin}\epsilon^{1/3} \Upsilon_{k, \tau}(v,w)  \notag\\
    &- i v \Upsilon_{k, \tau}(v, w)  + i B_0'(v  +w) \Theta_{k, \tau}(v,w) = 
    e^{-ik\tau v  }B_0'(v + w)X_k(\tau, v +w)\\
    & (B_0(v + w))^2 \p_v^2 \Theta_{k, \tau} (v,w) +   B_0'(v + w) \p_v \Theta_{k, \tau} (v,w) - k^2 \Theta_{k, \tau}(v,w) = \Upsilon_{k, \tau}(v,w)
    \end{align}
\begin{align}
    \label{Xrhs}
          X_k(\tau, w) := ik  \Delta_0^{-1} \Omega_k^*(\tau, w) \\
\end{align}

We can rewrite \eqref{rep:for:duh} in terms of $\Upsilon_{k, \tau}$ to get 
\begin{align}
    \label{rep:renorm}
    F_k(t,v) &=  -\frac{1}{2\pi } \int_0^t  e^{- (\delta_{lin} \nu^{1/3} |k|^{2/3} + \nu k^2    ) (t - \tau)  } \int_\R e^{ikt(v-w)  } 
     \Upsilon_{k, \tau}(v -w, w)  \, d w\,d\tau  \\
    &= -\frac{1}{2\pi } \int_0^t  e^{- (\delta_{lin} \nu^{1/3} |k|^{2/3} + \nu k^2    ) (t - \tau)  } \int_\R e^{iktw'  } 
     \Upsilon_{k, \tau}(w', v - w')  \, d w'\,d\tau. \notag
\end{align}
Taking the Fourier transform of \eqref{rep:renorm} with respect to $v$ gives
\begin{align}
    \label{Frep1}
    \widehat{F_k}(t, \xi) = -\frac{1}{2\pi } \int_0^t  e^{- (\delta_{lin} \nu^{1/3} |k|^{2/3} + \nu k^2    ) (t - \tau)  } \widehat{\Upsilon}_{k, \tau}(\xi - kt, \xi) 
 \,d\tau. 
\end{align}


The following proposition is a slight modification of Theorem 1.3 in \cite{J23} and follows by the methods developed therein. 
\begin{proposition}
\label{shift}
Assume that $s \geq 0$, $k \in \mathbb{Z} \cap [1, \infty)$, $\tau \in [0,t]$,  and let $(\Upsilon_{k, \tau}(v,w), \Theta_{k, \tau} (v,w))$ satisfy \eqref{ups:oss}. Then there exists $\epsilon_0 > 0$ such that for $\epsilon \in (0, \epsilon_0)$ the following bounds hold: 
\begin{align}
\label{Y:bdd}
    \norm{  \lb k, \eta  \rb^s \lb k, \eta - k \tau  \rb \sup_{\xi \in \R}  |\widehat{\Upsilon}_{k, \tau} (\xi, \eta )|}_{L_\eta^2} \lesssim   \norm{ \lb k, \xi \rb^s \lb k, \xi - k\tau \rb \widehat{X}_k(\tau ,\xi )}_{L_\xi^2}.
\end{align}

\end{proposition}

With Proposition \ref{shift}, we can now prove Proposition \ref{prop:linear:bounds:k}. 
\begin{proof}
  First we will prove \eqref{linpro3}. We can ignore the presence of any ghost multipliers in the Fourier multiplier $A$ and only consider the exponential enhanced dissipation factor and the Sobolev regularity. Using \eqref{Frep1} and choosing $\delta$ so that  $2\delta \leq \delta_{lin}$ yields 
\begin{align}
\label{lindamp.5}
    &|\zeta_k(t) \lb k, \xi \rb^s |\widehat{F}_k(t,  \xi)|  
    \lesssim  \zeta_k(t) \int_0^t e^{- (\delta_{lin} \nu^{1/3} |k|^{2/3} + \nu k^2    ) (t - \tau)}   \lb k, \xi \rb^s |\widehat{\Upsilon}_{k, \tau}(\xi - kt, \xi)| \, d \tau  \\
    &\leq \left(\int_0^t \frac{1}{\lb k, \xi - k\tau \rb^2 }  \, d \tau \right)^{1/2} \left( \int_0^t  \zeta_k^2(\tau)  \lb k, \xi - k \tau \rb^2   \lb k, \xi \rb^{2s} \sup_{\gamma \in \R}|\widehat{\Upsilon}_{k, \tau}(\gamma, \xi)|^2      \, d \tau    \right)^{1/2} \notag\\
    &\leq \frac{1}{|k|} \left( \int_0^t \zeta_k^2(\tau) \lb k, \xi - k \tau \rb^2  \lb k, \xi \rb^{2s} \sup_{\gamma \in \R}|\widehat{\Upsilon}_{\tau, k}(\gamma, \xi)|^2       \, d \tau    \right)^{1/2}. \notag
\end{align}
Applying the $L_\xi^2$ norm, Fubini's theorem,  Proposition \ref{shift}, and Lemma \ref{lem:ellip}  imply
\begin{align}
\label{lindamp.7}
    \norm{A F_k(t,\xi)}_{L_\xi^2}^2 &\lesssim \frac{1}{|k|^2} \int_0^t\zeta_k^2(\tau) \norm{   \lb k , \xi - k\tau \rb   \lb k, \xi \rb^{s} \sup_{\gamma \in \R}|\widehat{\Upsilon}_{k,\tau}(\gamma, \xi)|   }_{L_\xi^2}^2 d\tau  \\
    &\lesssim  \int_0^t \zeta_k^2(\tau) \norm{\lb k , \xi \rb^{s}   \lb k, \xi - k\tau\rb^{-1} \widehat{\Omega_k^*}  (\tau , \xi )     }_{L_\xi^2}^2 \,  d \tau. \notag
\end{align}
This concludes the proof of \eqref{linpro3}. 
The proof of \eqref{linpro3.5} is similar: Using \eqref{lindamp.5} we have 
    \begin{align}
    \label{lindamp1}
         &\lb k, \xi - t k \rb^{-1} \lb k, \xi \rb^s \zeta_k(t) \widehat{F}_k(t,\xi)| \\ 
         &\lesssim \frac{1}{|k| \lb k, \xi - kt\rb}\left( \int_0^{T} \zeta_k^2(\tau) \lb k, \xi -k \tau\rb^s   \lb k, \xi \rb^{2s} \sup_\gamma|\widehat{\Upsilon}_{k, \tau} (\gamma ,\xi)|^2       \, d \tau    \right)^{1/2}\notag
    \end{align}
    where we are assuming that $t\leq T$.
    Crucially, the only $t$ dependence on the right hand side of \eqref{lindamp1} comes from the factor  $ \lb k, \xi - kt \rb$. Squaring both sides of \eqref{lindamp1} and integrating with respect to $t$ over $[0,T]$ and $\xi$ over $\R$ then implies \eqref{linpro3.5} by following \eqref{lindamp.7}.  Finally, we prove \eqref{linpro3.4}. 
    For $\tau \in [0, t]$ we have $\lb k , \xi - kt \rb \lesssim \lb k, \xi - \tau k\rb + tk$.
    Therefore, by \eqref{Frep1} 
    \begin{align}
        &\lb k, \xi - kt \rb \lb k, \xi \rb^s\zeta_k(t) \widehat{F}_k(t,\xi)\\
        &\lesssim  \int_0^t \zeta_k^2(\tau) \lb k , \xi - k\tau \rb \lb k, \xi \rb^s \sup_{\gamma \in \R} |\widehat{\Upsilon}_{k, \tau}( \gamma, \xi)|  \, d \tau + |kt|\int_0^t \zeta_k^2(\tau) \lb k, \xi \rb^s \sup_{\gamma \in \R} |\widehat{\Upsilon}_{k, \tau}( \gamma, \xi)|   \, d \tau \notag.
    \end{align}
    For the first term, applying the $L_\xi^2$ norm, using \eqref{Y:bdd}, and Cauchy-Schwarz in time  yields 
    \begin{align}
       &\norm{ \int_0^t \zeta_k^2(\tau)  \lb k ,\xi - k\tau \rb \lb k, \xi \rb^s  \sup_{\gamma \in \R} |\widehat{\Upsilon}_{k, \tau}( \gamma, \xi)|  \, d \tau}_{L_\xi^2} 
      \lesssim t^{1/2} \norm{A \lb k, \xi - k\tau \rb^{-1} \Omega_k^{*}(\tau, \xi) }_{L^2[0,t]L_\xi^2} . 
    \end{align}
    For the second term  we bound as in \eqref{lindamp.5} and \eqref{lindamp.7} to obtain 
    \begin{align}
        \norm{|kt|\int_0^t \zeta_k(\tau) \lb k, \xi \rb^s \sup_{\gamma \in \R} |\widehat{\Upsilon}_{k, \tau}( \gamma, \xi)|  \, d \tau }_{L_{\xi}^2} \lesssim |t|\norm{A \lb k, \xi - k\tau\rb^{-1} \Omega_k^{*}(\tau, \xi) }_{L^2[0,t]L_\xi^2} 
    \end{align}
    which completes the proof of \eqref{linpro3.4}. The lemma is now proved.
\end{proof}


\appendix

\section{Pointwise Frequency bounds for Green's function}
We are interested in studying the Green's function $\mathcal{G}_k(v,w)$ which satisfies the equation
\begin{align}
\label{Gequation}
    B_0^2(v) \p_v^2 \G(v,w) + B_0(v)\p_vB_0(v) \p_v \G(v,w) - k^2 \G(v,w) = B_0(w)\delta(v-w).
\end{align}
\begin{lemma}
\label{lem:G:bdds}
    For $k \in \mathbb{Z} \setminus \{ 0\} $ we can decompose $\G(v,w)$ in the following way
    \begin{align}
    \label{Gdecomp}
        \G(v,w) = \chi(w)\mathcal{G}_{1,k}(v-w) + \mathcal{G}_{2,k}(v,w),
    \end{align}
where  $\chi$ satisfies 
\begin{align}
\label{cut:bdd}
    \norm{\chi}_{L^{\infty}(\R)} + \sup |e^{\lb \xi \rb^{1/2}} \widehat{\p_v \chi }(\xi)| \lesssim 1,
\end{align}
\begin{equation}
\label{G1}
    |\widehat{\mathcal{G}_{1 ,k}}(\xi)| \lesssim \frac{1}{k^2 + \xi^2 },
\end{equation}
\begin{equation}
\label{G2}
    |\widehat{\mathcal{G}_{2, k}}(\xi, \eta)| \lesssim \frac{1}{(k^2 + \xi^2) \lb \xi + \eta \rb^{ \ceil{s} +2  }}.
\end{equation}
    Furthermore, for $\mathfrak{C}_1, \mathfrak{C}_2$  
    satisfying \eqref{cut:bdd} we have that 
\begin{equation}
\label{G3}
    |(\mathfrak{C}_1(\cdot)\mathcal{G}_{2,k}(\cdot, *)\mathfrak{C}_2(*))^{\wedge}(\xi, \eta)| \lesssim \frac{1}{(k^2 + \xi^2)\lb \xi + \eta\rb^{\ceil{s} + 2}}.
\end{equation}
\end{lemma}


\begin{proof}
Since $\text{supp } \p_y^2b(y) \subseteq [-1/\sigma_0, 1/ \sigma_0 ]$ , we have $\p_v B_0(v) \subseteq [-L, L]  $ 
for some $L >1 $. 
Define two partitions of unity $ \tilde{\chi}_{*}, \chi_{*}, * \in \{ -, 0, +  \}$ satisfying the following conditions
\begin{align}
    &1 = \chi_-(v) + \chi_0(v) + \chi_+(v), \\
    &\chi_0(v) \equiv 1 \quad  v \in [-1 -L, L + 1],\quad  \text{supp }\chi_0 \subseteq [-2 - L , L + 2],  \\
    &\chi_-(v) = \chi_+(-v),\\
\end{align}
 and 
\begin{align}
    &1 = \tilde{\chi}_-(v) + \tilde{\chi}_0(v) + \tilde{\chi}_+(v), \\
    &\tilde{\chi}_0(v) \equiv 1 \quad  v \in [-3 -L, L + 3],\quad  \text{supp }\tilde{\chi}_0 \subseteq [-4 - L , L + 4],  \\
    &\tilde{\chi}_-(v) = \tilde{\chi}_+(-v).
\end{align}
The cutoffs satisfy the following  properties:
\begin{enumerate}
    \item Both $\tilde{\chi}_*$ and $\chi_*$ satisfy \eqref{cut:bdd},
    \item The supports of $\chi_{\pm}$ are disjoint from the support of $\p_vB_0(v)$,
    \item The support of $\chi_0$ is disjoint from $\tilde{\chi}_{\pm}$.
\end{enumerate}
Using these cutoffs, we define 
\begin{align}
    \mathcal{G}^*(v,w) := \G(v,w) \tilde{\chi}_0(w), 
\end{align}
and 
$\mathcal{G}_{\alpha \beta}, \alpha  \in \{-, 0, +\}, \beta \in \{-, +\}$ as 
\begin{align}
    \mathcal{G}_{\alpha \beta} (v,w) := \chi_\alpha(v) \G(v,w) \tilde{\chi}_\beta(w). 
\end{align}
Since we closely follow \cite{J20} Appendix A.2, we will not provide full details and advise the interested reader to look there.  Starting from the identity \eqref{green:helmholtz}
we deduce that 
\begin{align}
    \mathcal{G}_k(v,w) &= \frac{e^{-|k||v-w|E(v,w)}  }{|k|},\quad\
    E( v,w) = \int_0^1 (b^{-1})'(  w + s(v-w) ) \, ds.
\end{align}
Crucially, $E$ is as smooth as $(b^{-1})'$ in both $v$ and $w$. 
Using the product and composition properties for Gevrey spaces (see \cite{J20}) we can deduce that 
\begin{equation}
\label{G0w}
   |\p_w^{l} \G(v + w,w)| \lesssim  e^{-\mu |k v|}l!^2M^l, \quad l \in \mathbb{Z} \cap [1, \infty)
\end{equation}
where $M > 1$ and  $\mu$ depends on $\sigma_0$ from Assumption \ref{assum:b}. Crucially, the exponential decay in $v$ is uniform in the number of derivatives we take with respect to $w$.  From the equation
\begin{align}
\label{GFourier}
    \widehat{\mathcal{G}^*}(\xi , \eta) 
    = \int_{\R^2} \G(v + w,w) \tilde{\chi}_0(w) e^{-i\xi v - i(\xi + \eta) w} \, dvdw,
\end{align}
the bound \eqref{G0w}, integrating by parts twice with respect to  $v$, and integrating by parts $\ceil{s} + 2$  times with respect to $w$ we can deduce that 
\begin{equation}
\label{G0}
    |\widehat{\mathcal{G}^*}(\xi , \eta)| \lesssim \frac{1}{(k^2 + \xi^2)\lb \xi + \eta\rb^{\ceil{s} + 2} }.
\end{equation}

Next, we turn our attention to $\mathcal{G}_{\alpha \beta}$. We observe that $\mathcal{G}_{\alpha \beta}$ satisfies the equation
\begin{align}
    &B_0^2(v) \p_v^2 \mathcal{G}_{\alpha \beta}(v,w) + B_0(v)\p_vB_0(v) \p_v \mathcal{G}_{\alpha \beta}(v,w) - k^2 \mathcal{G}_{\alpha \beta}(v,w) \\
    &= 
   \chi_{\alpha}(w)\tilde{\chi}_{\beta}(w)B_0(w)\delta(v-w) 
    \\
    &+ \tilde{\chi}_{\beta}(w) (2B_0^2(v)\p_v\chi_\alpha(v)   \p_v \G(v,w) + B_0^2(v)\p_v^2\chi_\alpha(v)  \G(v,w) - B_0(v) \p_v B_0(v)\p_v \chi_\alpha(v) \G(v,w)).
\end{align}
 The cases $\mathcal{G}_{+\beta}$ and $\mathcal{G}_{-\beta}$ are symmetric. Similarly, the treatment of    $\mathcal{G}_{0 +}$ and  $\mathcal{G}_{0-}$ are symmetric. The case of $\mathcal{G}_{++}$ is harder than $\mathcal{G}_{+-}$ so we will only consider the cases $\mathcal{G}_{++}$ and $\mathcal{G}_{0+}$. First consider the case $\mathcal{G}_{++}$: on the support of $\chi_+(v)$, $\p_vB_0(v) = 0$ and hence $B_0(v)$ is a constant which we denote as $B_+$. Also, $\chi_+ \equiv 1$ on the support of $\tilde{\chi}_+$ so we can simplify \eqref{Gequation} as 
    \begin{align}
     \p_v^2 \mathcal{G}_{++}(v,w)  - (k/B_+)^2\mathcal{G}_{++}(v,w) &=  B_+^{-1}\tilde{\chi}_{+}(w) \delta(v-w) 
    \\
    &+ \tilde{\chi}_{+}(w)  (2\p_v\chi_+(v)  \p_v \G(v,w) + \p_v^2\chi_+(v) \G(v,w)). 
\end{align}
Using $G_{k/B_+}(v,w)$ to represent the kernel for the operator $\p_v^2 - (k/B_+)^2$ we can represent $\mathcal{G}_{++}$ as 
\begin{align}
\label{G++}
    \mathcal{G}_{++}(v,w) &= B_+^{-1} \tilde{\chi}_{+}(w) G_{k/B_+}(v-w) + I_+(v,w),\\
    I_+(v,w) &:= \int_\R G_{k/B_+}(v-\rho) \tilde{\chi}_{+}(w) (2\p_{\rho}\chi_+(\rho)  \p_{\rho} \G(\rho,w) + \p_{\rho}^2\chi_+(\rho) \G(\rho,w) )  d \rho.  
\end{align}
We have that $B_+^{-1} \tilde{\chi}_{+}(w) G_{k/B_+}(v-w)$ is in a form compatible with \eqref{Gdecomp} and \eqref{G1} so we turn our attention to $I_+(v,w)$. We claim 
\begin{equation}
    |\widehat{I_+}(\xi, \eta)| \lesssim  \frac{1}{(k^2 + \xi^2)\lb  \xi + \eta \rb^{\ceil{s}  + 2 }}.
\end{equation}
Proceeding as in \eqref{GFourier} and \eqref{G0w}, we can obtain that 
\begin{equation}
\label{Ikernel}
    |(\p_\rho\chi_+(\cdot)  \p_\rho \G(\cdot,\cdot) \tilde{\chi}_{+}(\cdot))^{\wedge}(\xi, \eta)|  +  |(\p_\rho^2\chi_+(\cdot) \G(\cdot,\cdot) 
 \tilde{\chi}_+(\cdot))^{\wedge}(\xi, \eta)| \lesssim \frac{1}{\lb \xi + \eta \rb^{\ceil{s} + 2} }.
\end{equation}
Using that $I_+$ is a convolution of $G_{k/B_+}$ and \eqref{Ikernel} then implies 
\begin{align}
    |\widehat{I_+}(\xi, \eta)| \lesssim 
    \frac{ 1}{(k^2 + \xi^2) \lb \xi + \eta \rb^{\ceil{s} + 2}},
\end{align}
which completes the proof of case of $\mathcal{G}_{++}$. 

We now consider $\mathcal{G}_{0+}$: Using the disjoint support of $\chi_0(w)$ and $\chi_+(w)$, $\mathcal{G}_{0+}$ solves the equation 
\begin{align}
     &B_0^2(v) \p_v^2 \mathcal{G}_{0+}(v,w) + B_0(v)\p_vB_0(v) \p_v \mathcal{G}_{0 +}(v,w) - k^2 \mathcal{G}_{0 +}(v,w) \\
    &= \tilde{\chi}_+(w)(2B_0^2(v)\p_v\chi_0(v)  \p_v \G(v,w) + B_0^2(v)\p_v^2\chi_0(v) \G(v,w) - B_0(v) \p_v B_0(v)\p_v \chi_0(v) \G(v,w))
\end{align}
which implies 
\begin{align}
\label{G0+}
    \mathcal{G}_{0+}(v,w) 
    &= \int_\R \G(v, \rho) \tilde{\chi}_+(w) \left[ 2B_0^2(\rho)\p_\rho\chi_0(\rho)  \p_\rho \G(\rho,w) + B_0^2(\rho)\p_v^2\chi_0(\rho) \G(\rho,w)\right] \, d \rho  \\
    &-\int_\R \G(v, \rho) \tilde{\chi}_+(w) B_0(\rho) \p_\rho B_0(\rho)\p_\rho \chi_0(\rho) \G(\rho,w)\,  d \rho 
\end{align}
Each term in \eqref{G0+} is of the form 
\begin{equation}
\label{G0+1}
    I_0(v,w) := \int_\R \G(v, \rho) \mathcal{C}_0(\rho) H(\rho, w) \tilde{\chi}_+(w) d\rho 
\end{equation}
where $\mathcal{C}_0(\rho)$ Gevrey-2  regular and compactly supported away from the support of $\tilde{\chi}_+(w)$, $H(\rho, w)$ is either $\G(\rho,w)$ or  $\p_\rho \G(\rho ,w)$. Since $|\rho - w|> 0$ on the support of $\mathcal{C}_0(\rho) \tilde{\chi}_+(w)$, $H(\rho+ w, w)$ is actually smooth with respect to $\rho$ but we will not need this fact. Due to the localization of the cutoffs to compact sets and the exponential decay of $\G(v, \rho)$ in $v$ and $\G(\rho, w)$ in $w$, the Fourier transform of $\mathcal{G}_{0+}$ is bounded. For the higher frequencies we consider the change of variables $v \to v+ w$ and shift $\rho \to \rho + w $  in \eqref{G0+1} to get
\begin{align}
    I_0(v + w, w) = \int_\R \G(v + w, \rho + w) \mathcal{C}_0(\rho + w) H(\rho + w, w) \chi_+(w) d\rho  
\end{align}
Using the product rule and the previously established bounds for $\mathcal{G}^*$, we can deduce that 
\begin{align}
    |\widehat{I}_0(\xi, \eta)| \lesssim 
    \frac{ 1}{(k^2 + \xi^2)\lb \xi + \eta \rb^{\ceil{s} +2 }}. 
\end{align}
By the product rule, \eqref{G3} follows immediately from the bounds for $\mathcal{G}_{2, k}$ and \eqref{cut:bdd}.  
\end{proof}

\section{Proof of Commutator Lemmas}
\label{app:com}

We first provide a proof of Lemma \ref{lem:C:bdd}.
\begin{proof}[Proof of Lemma \ref{lem:C:bdd}]
 Since $\mathfrak{C}$ is independent of $z$,  it suffices to prove 
 \begin{align}
     \norm{\lb\p_v\rb^s (\mathfrak{C} f_k )}_{L_v^2} \lesssim \norm{ \lb\p_v\rb^s f_k}_{L_v^2}.
 \end{align}
Furthermore, since $(\lb \p_v \rb^s)^{\wedge} = (1 + |\xi|^2)^{s/2} \lesssim 1 + |\xi|^{s} $ and 
\begin{align}
    \norm{\mathfrak{C} f_k }_{L_v^2} \lesssim \norm{\mathfrak{C}}_{L_v^{\infty}} \norm{f_k}_{L_v^2},
\end{align}
it suffices to prove
\begin{align}
    \norm{|\p_v|^s(\mathfrak{C} f_k )}_{L_v^2} \lesssim \norm{\lb \p_v\rb^s f_k}_{L_v^2}. 
\end{align}
This follows from the fractional Leibniz rule 
\begin{align}
     \norm{|\p_v|^s( f g  )}_{L_v^2} \lesssim \norm{f}_{L_v^{\infty}} \norm{|\p_v|^sg}_{L_v^2} + \norm{|\p_v|^sf}_{L_v^{2}} \norm{g}_{L_v^{\infty}},
 \end{align}
 see  \cite{Li19kato}. To prove \eqref{C:equiv}, we write 
 \begin{align}
     f_k = \frac{1}{\mathfrak{C}} \mathfrak{C} f_k, 
 \end{align}
and repeat the above argument with $\mathfrak{C}^{-1}$ in place of $\mathfrak{C}$. This concludes the lemma. 
\end{proof}

Next, we prove Lemma \ref{lem:com:half}.

\begin{proof}[Proof of Lemma \ref{lem:com:half}]
We only prove \eqref{halfhelm} as \eqref{helm} is similar. We work mode by mode. Following the proof of Lemma \ref{lem:ellip}, for $k \in \mathbb{Z} \setminus\{0\}$ we have that 
\begin{align}
\label{hcom1}
    [\mathfrak{C}, (-\Delta_L)^{-1/2}  ] f_k(v) = \int_\R e^{-ikt(v-w)} (\mathfrak{C}(v) - \mathfrak{C}(w)) K_{0}(k(v -w)) f_k(w) \,   dw, 
\end{align}
where $K_{0}$ is a of modified Bessel function of the second kind: 
\begin{equation}
\label{bessel}
    K_{0}(v) :=  \int_{\R}    \frac{1}{(1 + |\xi|^2)^{1/2}}  e^{i v \xi } \, d \xi, \quad v \neq 0,
\end{equation}
and we have used the equality
\begin{align}
    \int_\R \frac{1}{(k^2 + \xi^2)^{1/2}} e^{i v\xi } \,d \xi  = 
   \int_\R \frac{1}{(1 + \xi^2)^{1/2} } e^{i kv\xi }. 
\end{align}
Define
    \begin{align}
        H(v,w) :=  (\mathfrak{C}(v) - \mathfrak{C}(w)) K_{0}(k(v -w)).
    \end{align}
    we can write $\widehat{H}(\xi, \eta)$ as  
    \begin{align}
         \widehat{H}(\xi,\eta) &= \int_{\R^2}  wK_{0}(kw)  \int_0^1 \mathfrak{C}'( v - (1-s)w) ds\,   e^{-i (\xi + \eta ) v + i \eta  w}\, dv dw.
    \end{align}
Note that 
\begin{align}
    |\widehat{H}(\xi,\eta)| \lesssim  \int_{\R^2} |wK_{0}(kw)| |\mathfrak{C}'(v)| \, dv dw \lesssim \frac{1}{k^2},
\end{align}
 where we have used $vK_{0}(v), \p_v\mathfrak{C}(v) \in L^1(\R) $. Integrating by parts once with respect to $w$ and $\ceil{s} + 2$ with respect to $v$ implies  
 \begin{align}
 \label{hcom2}
     |\widehat{H}(\xi ,\eta)|  \lesssim \frac{\norm{\p_v\mathfrak{ C}}_{W^{ \ceil{s} +  3,1}}    }{(|k| + |\eta|) (1 + |\xi + \eta|)^{\ceil{s} + 2}}. 
 \end{align}

Taking the Fourier transform with respect to $v$ \eqref{hcom1}  
\begin{align}
     ([\mathfrak{C}, (-\Delta_L)^{-1/2}  ]  f_k)^{\wedge}(\xi) = \int_\R \widehat{H}(\xi - kt, kt - \eta) \widehat{f_k}(\eta)\, d \eta.   
\end{align}
Combined with \eqref{hcom2} this implies 
\begin{align}
\lb k, \xi \rb^s |([\mathfrak{C}, (-\Delta_L)^{-1/2}  ]  f_k)^{\wedge}(\xi)|  \lesssim \norm{\p_v\mathfrak{ C}}_{W^{ \ceil{s} +  3,1}} \int_\R \frac{ \lb k, \xi \rb^s|\widehat{f_k}(\eta)|}{(|k| + |\eta - kt|)\lb \xi - \eta \rb^{\ceil{s} +2 }}\,  d \eta.
   \end{align}
Young's inequality for convolutions then implies the desired result.
\end{proof}
We now prove Lemma \ref{lem:com:ghost}.
\begin{proof}[Proof of Lemma \ref{lem:com:ghost}]
    We only do the proof for $\mathcal{W}_E$ as the others are similar. Since $\mathcal{W}_E$ is bounded below, it suffices to prove the inequality with  $f$ on the right-hand side. Following the proof of Lemma \ref{lem:com:half}
    we write 
    \begin{align}
          [\mathfrak{C}, \mathcal{W}_E ] f_k(v) = \int_\R e^{-ikt(v-w)} (\mathfrak{C}(v) - \mathfrak{C}(w)) T_{k}(v -w) f_k(w) \,   dw,
    \end{align}
where
\begin{align}
    T_k (v) = C\sum_{\ell\neq 0} \frac{\text{sign}\ell  }{|\ell|^2} \frac{1}{iv } \exp\{-K(1+|k-\ell|+|\ell|)|v|\}, \quad v \neq 0.
\end{align}
We have that 
\begin{align}
    \left| \frac{\mathfrak{C}(v) - \mathfrak{C}(w)   }{ v- w} \right| \lesssim \norm{\mathfrak{C}}_{W^{1, \infty}}.  
\end{align}
Furthermore, $\norm{\exp(-K |v|)}_{L_v^1} \leq 1/K$ so Young's convolution inequality implies 
\begin{align}
     \norm{[\mathfrak{C}, \mathcal{W}_E ] f_k}_{L_v^2} \lesssim \frac{\norm{f}_{L_v^2}   }{K},
\end{align}
which completes the proof. 
 \end{proof}
Finally, we prove Lemma \ref{lem:B:diff}.
\begin{proof}[Proof of Lemma \ref{lem:B:diff}]
To prove \eqref{B:difference}, we have from the definition of $B(t,v)$ and $v$ that 
\begin{align}
\label{B:evo}
    \p_t B(t,v) = \nu (B(t,v))^2\p_v^2B(t,v). 
\end{align}
So, 
\begin{align}
    B(t,v)- B(0,v) = -t\nu \int_0^1 (B(st ,v ))^2 \p_v^2B(st, v)  \, ds.
\end{align}
Proceeding as in the proof of Lemma \ref{lem:C:bdd} and using \eqref{B:smooth} yields the desired result.
\end{proof}

\section{Fourier Multipliers}\label{sec:multiplier}
In this section, we summarize the properties of the Fourier multipliers that we employ. First, we collect some basic properties of the multiplier $\mf M$ defined in \eqref{M_N_Omega}. The main part of the proof can be found in \cite{He23}.
\begin{lemma}\label{lem:multipliers}
The following  properties  for $\mf M$ hold:
\begin{align}
\label{M_property_common}
\mf M(t,k,\eta)&=\pi^3,\quad |k|\notin(0,\nu^{-1/2}], \\
\frac{27}{8}\pi^3\geq& \mf M(t,k,\eta)\geq \frac{\pi^3}{8}, \\
- {\pa_t \mf M(t,k,\eta)} \geq& \frac{\pi^2}{4K}\frac{|k|^2}{|k|^2+|\eta-kt|^2},\quad k\neq 0,\\
 | {\pa_\eta \mf M(t,k,\eta)}  |\leq &\frac{12\pi^2}{K|k|},\quad k\neq 0.
\end{align}

Moreover, the multipliers $\mf M$ have the enhanced dissipation properties
\begin{align}\frac{2}{9K\pi^2}\nu^{1/3} {|k|^{2/3}}\leq&-\frac{\pa_t \mf M}{ \mf M}(t,k,\eta)+\nu (|k|^2+|\eta-kt|^2).\label{M_property_ED}
\end{align}
\end{lemma} 

The product rule for the multiplier $\mf M$ is contained in the next lemma:
\begin{lemma} Consider the multipliers $\mf M$. For two functions $f,g\in H^s(\Torus\times \rr), \, s> 1$, we have the following product rule
\begin{align}\label{M_product_rule_Hs}
\|\mf M(fg)\|_{H^s}\leq &C\|\mf M f\|_{H^s}\|\mf M g\|_{H^s}.
\end{align}
Similarly, we have the following product rule for $A,\, s> 1$,
\begin{align}\label{A_product_rule_Hs}
\|A (fg)\|_{L^2}\leq C\|A f\|_{L^2}\|Ag\|_{L^2}.
\end{align}
\end{lemma}

In the proof, the following commutator estimate is needed:
\begin{lemma}[Commutator estimates] The following commutator estimate concerning $\mf M$ is satisfied
\begin{align}\n
&|\mf M (t,k,\eta)\lan k,\eta\ran^{s}-\mf M(t,k,\xi)\lan k,\xi\ran^s |\\
&\leq C |\eta-\xi|(\lan \eta-\xi\ran^{s-1}+\lan k,\xi\ran^{s-1})+\frac{1}{K}\frac{|\eta-\xi|}{|k|}({\lan \eta-\xi\ran^{s}}+\lan k,\xi\ran^{s}),\quad k \neq 0.\label{com_est}
\end{align}
\end{lemma}
\begin{proof}

Here the difference $|\mf M (t,k,\eta)\lan k,\eta\ran^{s}-\mf M(t,k,\xi)\lan k,\xi\ran^s|$ can be decomposed as follows:
\begin{align}\label{Commutator_1_2}
\quad \quad \mf M&(t,k,\eta)\lan k,\eta\ran^s-\mf M(t,k,\xi)\lan k,\xi\ran^s\\
=&\mf M(t,k,\eta)\left((1+k^2+\eta^2)^{s/2}-(1+k^2+\xi^2)^{s/2}\right)\\
&+(\mf M(t,k,\eta)-\mf M(t,k, \xi))(1+k^2+ \xi^2)^{s/2}\mathbf{1}_{ k\neq 0}=:\mathcal{T}_1+\mathcal{T}_2.\\
\end{align}
For the first term in \eqref{Commutator_1_2}, one applies the mean value theorem to obtain that there exists $\theta\in[0,1]$, such that the  following estimate holds
\begin{align}
|\mathcal{T}_1|=&\bigg|\mf M(k,\eta)\frac{s}{2}(1+k^2+((1-\theta)\eta+\theta\xi)^2)^{\frac{s}{2}-1})2((1-\theta)\eta+\theta\xi)(\eta-\xi)\bigg|\\
\leq&C\bigg((1+k^2+ \xi^2)^{\frac{s-1}{2}}+(1+k^2+\eta^2)^{\frac{s-1}{2}}\bigg)|\eta-\xi|\leq C\frac{|\eta-\xi|}{|k|}(\lan k,\xi\ran^{s-1}+\lan  \eta-\xi\ran^{s-1}).
\end{align}
To estimate the $\mathcal{T}_2$ term in \eqref{Commutator_1_2}, we apply the property \eqref{M_property_common} and the mean value theorem to obtain that
\begin{align}
|\mathcal{T}_2|\leq\frac{C|\eta-\xi|}{\myr{K}|k|}(1+k^2+\xi^2)^{\frac{s}{2}}\mathbf{1}_{k\neq 0}.
\end{align}
Combining the two estimates and \eqref{Commutator_1_2}, we obtain \eqref{com_est}. 
The proof of the lemma is finished.
\end{proof}

We will need another family of commutator estimates for the short time. 
 \begin{lemma}\label{lem:Acm_sh}
The following commutator relation holds for $t\in[0, \nu^{-1/6}]$,
\begin{subequations}\label{cm_sh}
\begin{align}\n
& \sum_{k,\ell}\iint\zeta_k\lf|\wt A(t,k,\eta)-\wt A(t,k-\ell, \eta-\xi)\rg|\frac{|\ell||\eta-\xi-(k-\ell)t| }{\ell^2+|\xi-\ell t|^2}\lf|\wh{\mf S_\nq}(\ell,\xi)\rg|\lf|\wh {\oms }(k-\ell,\eta-\xi)\rg|\lf|A\overline{\wh {\oms}(k,\eta)} \rg|d\eta d\xi  \\
  &\hspace{3cm}\lesssim (1+t)\|A\mf S_\nq\|_{L^2}\|A\oms\|_{L^2}^2;\label{cm_sh_a}\\
\n &\sum_{k,\ell}\iint\zeta_k\lf|\wt A(t,k,\eta)-\wt A(t,k-\ell, \eta-\xi)\rg|\frac{(|\xi-\ell t|+1)  |k-\ell|}{\ell^2+|\xi-\ell t|^2}\lf|\wh{\mf S_\nq }(\ell,\xi)\rg|\lf| \wh{\oms} (k-\ell,\eta-\xi)\rg|\lf|A\overline{\wh{\oms}(k,\eta)} \rg|d\eta d\xi\\
&\hspace{3cm}\lesssim {\|A\mf S_\nq\|_{L^2}\|A\oms\|_{L^2}^2}.\label{cm_sh_b}
\end{align}
\end{subequations}
Here, the $\zeta_k,\, A,\, \wt A$ multipliers are defined in \eqref{zeta} and \eqref{A_N_Omega}.
\myb{ Right now, the regularity threshold is $H^{2}$.}
 \end{lemma}
 \begin{remark}
We recall that in the time interval $[0,\nu^{-1/6}]$, the multiplier $\mf M$ within $A$ only contains the $\W_I$ component. Moreover, in the actual application, the $\mf S_\nq$ will be $\de_L(B\de_t^{-1}\oms_\nq)$,\,$\de_L((\pa_v B)\de_t^{-1}\oms_\nq)$, $\de_L(B\de_t^{-1}F_\nq)$,  or $\de_L((\pa_v B)\de_t^{-1}F_\nq)$.
 \end{remark}
 
 \begin{proof}

Now we look at the commutator term
\begin{align}
&\eqref{cm_sh_a}_{l.h.s.} \\
&\leq\sum_{k\in \mathbb{Z},\ell\neq 0}\iint \zeta_k\lf( \W_I(k,\eta)\lf|\lan k,\eta\ran^s-\lan k-\ell,\eta-\xi\ran^s\rg|+\lan k-\ell,\eta-\xi\ran^s\lf| \W_I(k,\eta)-\W_I(k-\ell,\eta-\xi)\rg|\rg)\\
&\hspace{2cm}\times\frac{|\ell||\eta-\xi-(k-\ell)t|}{\ell^2+|\xi-\ell t|^2}|\wh{\mf S_\nq}(\ell,\xi)|  |\wh{\oms}  (k-\ell,\eta-\xi)||A\wh\oms (k,\eta)| d\eta d\xi\\
&=:T_{a;1}+T_{a;2}.\label{cm_sh_a_p1}
\end{align}
The estimate of $T_1$ follows from the argument in \cite{WZ23}. We observe that thanks to the range of $s\geq 2$ and the mean value theorem,
\begin{align}
&\lf|\lan k\ran^s-\lan k-\ell\ran^s\rg|=\lf|\int_0^1 \frac{d}{d\theta}\lf(\theta(1+k^2)+
(1-\theta)(1+|k-\ell|^2)\rg)^{s/2} d\theta\rg|\\
&=\frac{s}{2}\lf|\int_0^1 \lf(\theta(1+k^2)+
(1-\theta)(1+|k-\ell|^2)\rg)^{s/2-1} ( k^2-(k-\ell)^2) d\theta\rg|\\
&\lesssim \lf|\int_0^1 \lf(
\lan k-\ell\ran^{s -2}+\lan \ell\ran^{s -2}\rg) |2k-\ell||\ell| d\theta\rg|\lesssim (\lan k-\ell\ran^{s -1}+\lan \ell\ran^{s -1}) |\ell| .
\end{align}
A similar argument applied to the $\eta,\eta-\xi$ variable yields that,
\begin{align}
\lf|\lan k,\eta\ran^s-\lan k-\ell,\eta-\xi\ran^s\rg|\lesssim \lf(
\lan k-\ell,\eta-\xi\ran^{s -1}+\lan \ell, \xi\ran^{s -1}\rg) ( |\ell|+ |\xi|) .\label{cm_sh_a_p2}
\end{align}Here we use the compact notation \eqref{mul_bracket}. Hence, the $T_{a;1}$ term can be estimated as follows
\begin{align}
    |T_{a;1}|\lesssim &\sum_{k\in \mathbb{Z},\ell\neq 0}\iint \zeta_k\lf(\lan k-\ell,\eta-\xi\ran^{s -1}+\lan \ell,\xi\ran^{s -1}\rg)\frac{|\ell|(|\ell| +|\xi| )|\eta-\xi-(k-\ell)t|}{\ell^2+|\xi-\ell t|^2}|\wh{\mf S_\nq}(\ell,\xi)| \\
    &\hspace{2cm}\times|\wh{\oms}  (k-\ell,\eta-\xi)||A\wh\oms (k,\eta)| d\eta d\xi\\
    \lesssim& t \sum_{k\in \mathbb{Z},\ell\neq 0}\iint \zeta_k \lan k-\ell,\eta-\xi\ran^{s}\lan\ell,\xi\ran  \frac{|\wh{\mf S_\nq}(\ell,\xi)|}{\sqrt{\ell^2+|\xi-\ell t|^2}}  |\wh{\oms}  (k-\ell,\eta-\xi)||A\wh\oms (k,\eta)| d\eta d\xi\\
    &+t\sum_{k\in \mathbb{Z},\ell\neq 0}\iint  \zeta_k  \lan k-\ell,\eta-\xi\ran\lan \ell,\xi\ran^{s } \frac{ |\wh{\mf S_\nq}(\ell,\xi)|}{\sqrt{\ell^2+|\xi-\ell t|^2}} |\wh{\oms}  (k-\ell,\eta-\xi)||A\wh\oms (k,\eta)| d\eta d\xi\\
    \lesssim& t\|A\mf S_\nq\|_{L^2}\|A\oms\|_{L^2}^2. \label{cm_sh_a_p3}
\end{align}
This is consistent with the result \eqref{cm_sh_a}. 

Next, decompose the $T_{a;2}$-term in \eqref{cm_sh_a_p1} in parameter regimes (with $(k,\eta,\ell,\xi)\in (\mathbb{Z}\times\rr)^2$): 
\begin{align}\label{Rnge}
   a)&\quad R_1:=\{|\eta-\xi-(k-\ell)t|\leq \max\{{100}|\xi-\ell t|,1\} \,\, \text{or}\,\, |\eta-kt |\leq 1\} \cap \{|k||k-\ell|\neq 0\};\\
   b)& \quad R_2:=\{|\eta-\xi-(k-\ell)t|> \max\{{100}|\xi-\ell t|,1\} \,\, \text{and}\,\, |\eta-kt |> 1\}\cap\{|k||k-\ell|\neq 0\};\\
   c)& \quad R_3:=\lf\{|\eta-\xi-t/2|\leq \max\lf \{10 \lf|\xi-\ell t+t/2\rg|,1\rg\}\rg\} \cap\{|k|\neq 0, |k-\ell|=0\};\\
   d)& \quad R_4:=\lf\{|\eta-\xi-t/2|>\max\lf \{10 \lf|\xi-\ell t+t/2\rg|,1\rg\}\rg\} \cap\{|k|\neq 0, |k-\ell|=0\};\\
    e)& \quad R_5:=\lf\{ |\eta-\xi+\ell t|\leq \max\lf \{10 \lf|\xi-\lf(\ell+{1}/{2}\rg)t{}\rg|,1\rg\}\rg\} \cap\{|k|= 0, |k-\ell|\neq0\};\\
   f)& \quad R_6:=\lf\{ |\eta-\xi+\ell t|> \max\lf \{10 \lf|\xi-\lf(\ell+{1}/{2}\rg)t{}\rg|,1\rg\}\rg\} \cap\{|k|= 0, |k-\ell|\neq0\}.
\end{align}
We further define for $j \in \{1,2,..., 6\}$
\begin{align}
    T_{a;2 j}:=&\sum_{k\in \mathbb{Z},\ell\neq 0}\iint\zeta_k\mathbbm{1}_{(k,\eta,\ell,\xi)\in R_j} \lan k-\ell,\eta-\xi\ran^s\lf| \W_I(k,\eta)-\W_I(k-\ell,\eta-\xi)\rg| 
\\
&\hspace{3cm}\times\frac{|\ell||\eta-\xi-(k-\ell)t|}{\ell^2+|\xi-\ell t|^2}|\wh{\mf S_\nq}(\ell,\xi)|  |\wh{\oms}  (k-\ell,\eta-\xi)||A\wh\oms (k,\eta)| d\eta d\xi
\end{align}
We focus on the estimate of the $T_{a;21},T_{a;22}, T_{a;23}$ and $T_{a;24}$ terms. The estimates of $T_{a;25}$ and $T_{a;26}$ are similar to those of $T_{a;23}$ and $T_{a;24}$. The first term can be estimated as follows
\begin{align}
    T_{a;21}\lesssim&  
    \sum_{k,\ell}\iint\zeta_k\frac{|\ell|(|\xi-\ell t|+1)}{\ell^2+|\xi-\ell t|^2}|\wh{\mf S_\nq}(\ell,\xi)| \lan k-\ell,\eta-\xi\ran ^s|\wh{\oms }(k-\ell,\eta-\xi)|A\wh{\oms }(k,\eta)|d\xi d\eta\\
    \lesssim & \|A\mf S_\nq\|_{L^2}\|A\oms\|_{L^2}^2.
\end{align}
This is consistent with the estimate \eqref{cm_sh_a}.

Now in $R_2$\eqref{Rnge}, we have that for any $s\in[0,1]$, the following relation holds
\begin{align}\label{Rnge_r}
    |s(\eta-kt)+(1-s)((\eta-\xi)-(k-\ell)t)|^2\approx ((\eta-\xi)-(k-\ell)t)^2\geq 1, \quad |k||k-\ell|\neq 0. 
\end{align}  
 We rewrite the $\W_I$ as follows
\begin{align}
\W_I(t,k,\eta)=\pi+\arctan\lf(\frac{1}{K}\frac{1}{k}(\eta-kt)\rg) .
\end{align} 
Consider the function that interpolates between $\W_I(t,k,\eta)$ and $\W_I(t,k-\ell,\eta-\xi)$: 
\begin{align*}
    \mathfrak{L}(\theta):=\pi+\arctan\lf(\frac{1}{K}\frac{ \theta(\eta-kt)+(1-\theta)((\eta-\xi)-(k-\ell)t)  }{\theta k+(1-\theta)(k-\ell)}\rg),\quad \theta\in[0,1].
\end{align*}
If $k(k-\ell)<0$, then for some $\theta\in[0,1]$, the denominator will be zero, and we might have a jump in the $\mf L$ function. To avoid this jump, we use the parameter constraint \eqref{Rnge_r}. Direct derivative yields that as long as $\theta k+(1-\theta)(k-\ell)\neq 0$, we have that

\begin{align}
     &\frac{d}{d\theta}\lf(\arctan\lf(\frac{1}{K}\frac{\theta(\eta-kt)+(1-\theta)((\eta-\xi)-(k-\ell)t) }{\theta k+(1-\theta)(k-\ell)}\rg)\rg) \\
&=  \frac{1}{K+K^{-1}\lf(\frac{\theta(\eta-kt)+(1-\theta)((\eta-\xi)-(k-\ell)t)}{\theta k+(1-\theta)(k-\ell)}\rg)^2}\\
&\quad \times \lf(\frac{-\ell[\theta(\eta-kt)+(1-\theta)((\eta-\xi)-(k-\ell)t)]}{(\theta k+(1-\theta)(k-\ell))^2} + \frac{\xi-\ell t}{(\theta k+(1-\theta)(k-\ell)) }\rg) \\
&=  \frac{   -\ell [\theta(\eta-kt)+(1-\theta)((\eta-\xi)-(k-\ell)t)] +  (\theta k+(1-\theta)(k-\ell))  (\xi-\ell t) }{K( \theta k+(1-\theta)(k-\ell))^2+K^{-1}\lf(  \theta(\eta-kt)+(1-\theta)((\eta-\xi)-(k-\ell)t) \rg)^2} .
\end{align}
Now we observe that thanks to the constraint \eqref{Rnge_r}, the denominator is never zero, and hence we can safely extend the expression to the point where $\theta k+(1-\theta)(k-\ell)= 0$. Moreover, we can take higher derivatives in $\theta$ without difficulty. Now, thanks to the range constraint \eqref{Rnge_r},
\begin{align}&|\W_I(k,\eta)-  \W_I(k-\ell,\eta-\xi)|\leq \lf|\int_0^1\frac{d\mf L(\theta)}{d\theta} d\theta\rg|\\
&\lesssim \int_0^1 \frac{  |\ell| +|\xi-\ell t|    }{\lf(( \theta k+(1-\theta)(k-\ell))^2+\lf(  \theta(\eta-kt)+(1-\theta)((\eta-\xi)-(k-\ell)t) \rg)^2\rg)^{1/2}} d\theta\\
&\lesssim \frac{|\ell|+|\xi-\ell t|}{1+\mathbbm{1}_{k(k-\ell)>0}\min\{|k|,|k-\ell|\}+|(\eta-\xi)-(k-\ell)t|}.\label{Difference}\end{align}

 With this estimate, we estimate the $T_{a;22}$-term as follows
 \begin{align}
     |T_{a;22}|\lesssim &  
    \sum_{k,\ell}\iint\zeta_k\frac{|\ell|(|\xi-\ell t|+|\ell|)|(\eta-\xi)-(k-\ell)t|}{(1+|(\eta-\xi)-(k-\ell)t|)(\ell^2+|\xi-\ell t|^2)}|\wh{\mf S_\nq}(\ell,\xi)|\\
    &\hspace{3cm}\times\lan k-\ell,\eta-\xi\ran ^s|\wh{\oms_\nq}(k-\ell,\eta-\xi)|A\wh{\oms_\nq}(k,\eta)|d\xi d\eta\\
    \lesssim & \|A\mf S_\nq\|_{L^2}\|A\oms\|_{L^2}^2.
 \end{align}

In the $R_3$ region, we have that 
\begin{align*}
    |T_{a;23}|\lesssim&\sum_{k\neq 0,\ell\neq 0}\iint\zeta_k\lan \eta-\xi\ran^s  
\frac{|\ell|(|\eta-\xi-t/2|+t/2)}{\ell^2+|\xi-\ell t|^2}|\wh{\mf S_\nq}(\ell,\xi)|  |\wh{\oms}  (0,\eta-\xi)||A\wh\oms (k,\eta)| d\eta d\xi\\
\lesssim&\sum_{k\neq 0,\ell\neq 0}\iint\zeta_k\lan \eta-\xi\ran^s  
\frac{|\ell|(1+|\xi-\ell t|+t)}{\ell^2+|\xi-\ell t|^2}|\wh{\mf S_\nq}(\ell,\xi)|  |\wh{\oms}  (0,\eta-\xi)||A\wh\oms (k,\eta)| d\eta d\xi\\
    \lesssim& (1+t) \|A\mf S_\nq\|_{L^2}\|A\oms\|_{L^2}^2.
\end{align*}
Finally, we treat the $R_4$ region. Since
 \begin{align}|\eta-\xi-t/2|\geq \max\lf \{10 \lf|\xi-\ell t+t/2\rg|,1\rg\},\quad k=\ell,\label{Rnge_r_3}\end{align} 
we have that for any $\theta\in[0,1]$, 
\begin{align}
    &\lf(\theta\lf(\eta-kt\rg)+(1-\theta)(\eta-\xi-t/2)\rg) ^2=\lf(\theta\lf(\xi-kt+t/2\rg)+ (\eta-\xi- t/2)\rg) ^2\approx(\eta-\xi- t/2)^2\geq 1.
\end{align}
Now we consider the function that interpolates between $\W_I (t,k,\eta)$ and $\W_I^\circ(t,0,\eta-\xi)$: 
\begin{align*}
    \mathfrak{L}_2(\theta):=\pi+\arctan\lf(\frac{1}{K}\frac{ \theta(\eta-kt)+(1-\theta)(\eta-\xi-t/2)  }{\theta k+(1-\theta)(1/2)}\rg),\quad \theta\in[0,1].
\end{align*}
If $k<0$, then for some $\theta\in[0,1]$, the denominator will be zero, and we might have a jump in the $\mf L_2$ function. To avoid this jump, we use the parameter constraint \eqref{Rnge_r_3}. Direct derivative yields that as long as $\theta k+(1-\theta)(1/2)\neq 0$, we have that
\begin{align}
  \lf|\frac{d}{d\theta}\mf L_2(\theta)\rg|  &= \lf|\frac{d}{d\theta}\lf(\arctan\lf(\frac{1}{K}\frac{\theta(\eta-kt)+(1-\theta)(\eta-\xi-t/2) }{\theta k +(1-\theta)/2}\rg)\rg)\rg| \\
&=  \Bigg|\frac{1}{K+K^{-1}\lf(\frac{\theta(\eta-kt)+(1-\theta)(\eta-\xi- t/2)}{\theta k+(1-\theta)/2}\rg)^2}\\
&\quad \times \lf(\frac{-(k-1/2)[\theta(\eta-kt)+(1-\theta)(\eta-\xi- t/2)]}{(\theta k+(1-\theta)/2)^2} + \frac{\xi-\ell t+t/2 }{(\theta k+(1-\theta) /2 ) }\rg) \Bigg|\\
&=  \lf|\frac{   -(k-1/2) [\theta(\eta-kt)+(1-\theta)(\eta-\xi-t/2)] +  (\theta k+(1-\theta)/2)  (\xi-\ell t+ t /2) }{K( \theta k+(1-\theta)/2)^2+K^{-1}\lf(  \theta(\eta-kt)+(1-\theta)(\eta-\xi- t/2) \rg)^2}\rg|.
\end{align}
By applying the same argument as before, we can apply the assumption \eqref{Rnge_r_3} to get that 
\begin{align}
     |\W_I(k,\eta)-  \W_I^\circ(0,\eta-\xi)|& \lesssim \lf|\frac{ |\ell-1/2| +  |\xi-\myr{ \ell t+ t/2}| }{1+\lf|\eta-\xi- t/2 \rg|}\rg| .
\end{align}
We observe that 
\begin{align}
   &T_{a;24}\lesssim  \\
   &\sum_{k\neq 0,\ell\neq 0}\iint\zeta_k \frac{ |\ell| +  |\xi-  \ell t|+ t/2} {1+\lf|\eta-\xi- t/2 \rg|}\frac{|\ell|(|\eta-\xi-t/2|+t/2)}{\ell^2+|\xi-\ell t|^2}|\wh{\mf S_\nq}(\ell,\xi)|  |\lan \eta-\xi\ran^s\wh{\oms}  (0,\eta-\xi)||A\wh\oms (k,\eta)| d\eta d\xi\\
   &\lesssim \sum_{k\neq 0,\ell\neq 0}\iint\zeta_k( 1+ t/2) \frac{(1+t/2)}{\sqrt{\ell^2+|\xi-\ell t|^2}}|\ell||\wh{\mf S_\nq}(\ell,\xi)|  |\lan \eta-\xi\ran^s\wh{\oms}  (0,\eta-\xi)||A\wh\oms (k,\eta)| d\eta d\xi\\
     &\lesssim \sum_{\ell\neq 0}\iint\zeta_\ell \frac{(1+t/2)^2}{\lan\xi\ran(1+|\xi/\ell- t|)}\lan\xi\ran|\wh{\mf S_\nq}(\ell,\xi)|  |\lan \eta-\xi\ran^s\wh{\oms}  (0,\eta-\xi)||A\wh\oms (\ell,\eta)| d\eta d\xi\\
   &\lesssim (1+t)\lf(\sum_{\ell\neq 0}\int\zeta_\ell \frac{(1+t)}{\lan\xi\ran (1+|\xi/\ell- t|)} \lan\xi\ran| \wh{\mf S}(\ell,\xi)| d\xi\rg)\|A\oms_{\mathbbm{o}}\|_{L^2} \|A\oms_\nq\|_{L^2} .
\end{align}
Now we distinguish between two cases:  $|\xi/\ell -t|<t/2$ and $|\xi/\ell- t|\geq t/2$. In the first case, $\lan\xi\ran(1+|\xi/\ell- t|)\gtrsim 1+|\ell|t$, 
\begin{align}
&\lf|\sum_{\ell\neq 0}\int \zeta_\ell\frac{(1+t)}{ \lan \xi\ran(1+|\xi/\ell- t|)} \lan \xi\ran|\wh{\mf S}(\ell,\xi)| d\xi\rg|\\
&\lesssim \left(\sum_{\ell\neq 0}\frac{1}{|\ell|^2}\int \frac{1}{(|\ell|^2+\xi^2)^{3/2}}d\xi\right)^{1/2}\lf(\sum_{\ell\neq 0} \int  \zeta_\ell^2(|\ell|^2+|\xi|^2)^{3/2}\lan \xi\ran^2 | \wh{\mf S}(\ell,\xi)|^2 d\xi\rg)^{1/2} 
 \lesssim\|A\mf S_\nq\|_{L^2}.
\end{align}In the second case, we have that 
\begin{align}
    \lf(\sum_{\ell\neq 0}\int \zeta_\ell\frac{(1+t)}{  (1+|\xi/\ell- t|)}  | \wh{\mf S}(\ell,\xi)| d\xi\rg)\lesssim\|A\mf S_\nq\|_{L^2}.
\end{align}
As a result,
\begin{align}
   T_{a;24}\lesssim(1+t)\lf\| A{\mf S}_\nq\rg\|_{L^2}\|A\oms\|_{L^2} ^2.
\end{align}

\noindent
{\bf Step \# 2: Proof of \eqref{cm_sh_b}.}
Next, we consider \eqref{cm_sh_b}
\begin{align}
& \eqref{cm_sh_b}_{l.h.s.} \\
&\leq\sum_{k\in \mathbb{Z},\ell\neq 0}\iint\zeta_k\lf( \W_I(k,\eta)\lf|\lan k,\eta\ran^s-\lan k-\ell,\eta-\xi\ran^s\rg|+\lan k-\ell,\eta-\xi\ran^s\underbrace{\lf| \W_I(k,\eta)-  \W_I(k-\ell,\eta-\xi)\rg|}_{=:I}\rg)\\
&\hspace{2cm}\times\underbrace{\frac{(|\xi-\ell t|+1) |k-\ell|}{\ell^2+|\xi-\ell t|^2}}_{=:J}|\wh{\mf S_\nq}(\ell,\xi)|  |\wh{\oms}  (k-\ell,\eta-\xi)||A\wh\oms (k,\eta)| d\eta d\xi\\
&=:T_{b;1}+T_{b;2}.
\end{align}
{The $T_{b;1}$ term is treated as  $T_{a;1}$. We invoke \eqref{cm_sh_a_p2} and Young's convolution inequality to obtain that
\begin{align}
    |T_{b;1}|\lesssim &\sum_{k\in \mathbb{Z},\ell\neq 0}\iint \zeta_k\lf(\lan k-\ell,\eta-\xi\ran^{s -1}+\lan \ell,\xi\ran^{s -1}\rg)\frac{(|\ell| +|\xi| )(1+|\xi-\ell t|)|k-\ell|}{\ell^2+|\xi-\ell t|^2}|\wh{\mf S_\nq}(\ell,\xi)| \\
    &\hspace{2cm}\times|\wh{\oms}  (k-\ell,\eta-\xi)||A\wh\oms (k,\eta)| d\eta d\xi\\
    \lesssim&   \sum_{k\in \mathbb{Z},\ell\neq 0}\iint \zeta_\ell\lan\ell,\xi\ran  \frac{|\wh{\mf S_\nq}(\ell,\xi)|}{\sqrt{\ell^2+|\xi-\ell t|^2}}  \zeta_{k-\ell}\lan k-\ell,\eta-\xi\ran^{s} |\wh{\oms_\nq } (k-\ell,\eta-\xi)||A\wh\oms (k,\eta)| d\eta d\xi\\
    & +  \sum_{k\in \mathbb{Z},\ell\neq 0}\iint \zeta_\ell \lan\ell,\xi\ran^s  \frac{|\wh{\mf S_\nq}(\ell,\xi)|}{\sqrt{\ell^2+|\xi-\ell t|^2}}  \zeta_{k-\ell}|k-\ell||\wh{\oms_\nq}  (k-\ell,\eta-\xi)||A\wh\oms (k,\eta)| d\eta d\xi\\
    \lesssim&\|\zeta_k |k|^{-1}\lan k,\eta\ran \wh {\mf S_\nq}\|_{\ell_k^1L_\eta^1}\| A\oms_\nq\|_{L^2}\|A\oms\|_{L^2}+\|\zeta_k |k|^{-1}\lan k,\eta\ran^s \wh {\mf S_\nq}\|_{\ell_k^1L_\eta^2}\| \zeta_k|k|\wh{\oms_\nq}\|_{\ell_k^2L_\eta^1}\|A\oms\|_{L^2}\\
    \lesssim& \|A\mf S_\nq\|_{L^2}\|A\oms\|_{L^2}^2. 
\end{align}}
For the $T_{b;2}$-term, we observe that if $|k-\ell|=0$, there is nothing to prove ($J=0$). If $k=0$, then we can simply use that the multiplier is bounded ($I\lesssim 1$) and $J\lesssim 1$ to handle. The remaining case is $|k||k-\ell|\neq0$. To absorb the $|k-\ell|$, we use two observations. If $k(k-\ell)>0,\, \ell\neq0$,  then we can apply the the commutator estimate \eqref{com_WIz_bd} and the observation $\frac{|k-\ell|}{\min\{|k|,|k-\ell|\}|\ell|}\leq 2$. On the other hand, if $k(k-\ell)<0$, then $|\ell|\geq \max\{|k|,|k-\ell|\}$. Combining all these considerations, we obtain that 
\begin{align*}
T_{b;2}\leq & \sum_{k,\ell} \iint\zeta_k\frac{|\ell|(|\xi-\ell t|+|\ell|)}{\ell^2+|\xi-\ell t|^2}|\wh{\mf S_\nq}(\ell,\xi)| \lan k-\ell,\eta-\xi\ran ^s|\wh\oms(k-\ell,\eta-\xi)||A\wh\oms(k,\eta)|d\xi d\eta\\
&+\sum_{k,\ell}\mathbbm{1}_{k(k-\ell)>0,\, \ell\neq0}\iint\zeta_k\frac{|\ell|(|\xi-\ell t|)}{\ell^2+|\xi-\ell t|^2}\frac{|k-\ell|}{\min\{|k|,|k-\ell|\}|\ell|}|\wh{\mf S_\nq}(\ell,\xi)|\\
&\hspace{3cm}\times\lan k-\ell,\eta-\xi\ran ^s|\wh\oms(k-\ell,\eta-\xi)||A\wh\oms(k,\eta)|d\xi d\eta\\
&+\sum_{k,\ell}\mathbbm{1}_{k(k-\ell)<0,\, \ell\neq0}\iint\zeta_k\frac{|\ell|(|\xi-\ell t|+|\ell|)}{\ell^2+|\xi-\ell t|^2}|\wh{\mf S_\nq}(\ell,\xi)|\\
&\hspace{3cm}\times \frac{|k-\ell|}{|\ell|}\lan k-\ell,\eta-\xi\ran ^s|\wh\oms(k-\ell,\eta-\xi)||A\wh\oms(k,\eta)|d\xi d\eta\\
    \lesssim &\|A \mf S_\nq\|_{L^2}\|A\oms\|_{L^2}^2.
\end{align*}\myb{\myr{...}}
 This concludes the proof of the lemma. 
 \end{proof}

In the work \cite{WZ23}, the authors identify an efficient multiplier $\W_E$ that encodes sufficient information about the echo cascade in the system. We summarize two key lemmas from their work. 

\begin{lemma}\label{lem:W_E}
    The following time derivative estimate holds (see \cite{WZ23} for proof). 
    \begin{align}
        -\pa_t \W_E(t,k,\eta)\geq  \sum_{\ell\in\mathbb Z\backslash \{0\}}\frac{1}{|\ell|}\frac{1+|k-\ell|+|\ell|}{K(1+|k-\ell|+|\ell|)^2+K^{-1}(\eta-\ell t)^2}>0,\quad 0\leq |k|<\nu^{-1/2}.
    \end{align}
\end{lemma}
Another key lemma is the trilinear estimate Lemma \ref{lem:Tri_est}. 
\begin{proof}[Proof of Lemma \ref{lem:Tri_est}]
We decompose the integrand in \eqref{Tri_est} into the High-Low and Low-High contribution as follows (with the notation \eqref{zeta})
\begin{align}
    &\sum_{k,\ell} \iint \zeta_k\lan\xi,\ell\ran ^s |\ell ||\wh{(B\de_t^{-1}\mathcal F_\nq)}(\ell,\xi) ||(\eta-\xi-(k-\ell)t) \mathcal G(\eta-\xi,k-\ell)| \  |A\mathcal H(k,\eta)| d\eta d\xi  \\
    &+\sum_{k,\ell} \iint \zeta_k\lan\eta-\xi,k-\ell\ran ^s |\ell ||\wh{(B\de_t^{-1}\mathcal F_\nq)}(\ell,\xi) ||(\eta-\xi-(k-\ell)t) \mathcal G(\eta-\xi,k-\ell)| \  |A\mathcal H(k,\eta)| d\eta d\xi \\
   & =:T_{1}+T_{2}. 
\end{align}
\noindent
{\bf Step \# 1: The High-Low interaction $T_1$.} 
First of all, we estimate the $T_1$-term that encodes the echo effect (High-Low interaction)
\begin{align}
T_{1}&=\sum_{k,\ell}\int\iint  \zeta_k\lan \ell,\xi\ran^s |i\ell\wh{(B\de_t^{-1} \mathcal F_\nq)}(\ell,\xi) ||(\eta-\xi)-(k-\ell)t||\mathcal G(k-\ell,\eta-\xi)|A \mathcal H(k,\eta)|d\eta d\xi d\tau\\
    &\lesssim  \Bigg(\sum_{k,\ell}\iint   \zeta_k^2(\ell^2+(\xi-\ell t)^2) \lan \ell,\xi\ran^{2s} | \ell|^2\lf|\wh{(B\de_t^{-1} \mathcal F_\nq)}(\ell,\xi)\rg|^2 \\
    &\hspace{2cm}\times\lan k-\ell,\eta-\xi\ran^2 \bigg|(\eta-\xi-(k-\ell)t)\mathcal G(k-\ell,\eta-\xi)\bigg|^2 d\eta d\xi \Bigg)^{1/2}\\
    &\quad\times\lf(\sum_{k,\ell}\iint \frac{|A\mathcal H(k,\eta)|^2}{(\ell^2+(\xi-t\ell)^2) (1+ (k-\ell)^2+(\eta-\xi)^2)}  d\eta d\xi \rg)^{1/2} \\
    &\lesssim F_1^{1/2} F_2^{1/2}.\label{Factor_12}
\end{align}
The first factor $F_1$ can be estimated with the fact that $|k|^{2/3}\leq |k-\ell|^{2/3}+|\ell|^{2/3}$, the Fubini equality (or Young's convolution inequality), product estimate \eqref{zeta_product} and elliptic estimate  \eqref{ell_est_t} as follows
\begin{align}\label{F_1_est}
    F_1\leq& \sum_{\ell\neq 0}\int   \zeta_\ell^2(\ell^2+(\xi-\ell t)^2) \lan \ell,\xi\ran^{2s} | \ell|^2\lf|\wh{(B\de_t^{-1} \mathcal F_\nq)}(\ell,\xi)\rg|^2 \\
    &\hspace{2cm}\times\lf(\sum_{k} \int e^{2\delta\nu^{1/3}|k-\ell|^{2/3}t}\lan k-\ell,\eta-\xi\ran^2 \bigg|(\eta-\xi-(k-\ell)t)\mathcal G(k-\ell,\eta-\xi)\bigg|^2 d\eta\rg) d\xi \\
    \leq &C\myr{\lf(\lf\|\pa_z^2A(B\de_t^{-1}\mathcal F_\nq)\rg\|_{L^2}^2 +\lf\|(\pa_v-t\pa_z)\pa_z A ( B\de_t^{-1}\mathcal F_\nq)\rg\|_{L^2}^2\rg)}\sup_{t\in[0,T]}( \|t e^{\delta\nu^{1/3} |\pa_z|^{2/3} t}\mathcal G_\nq\|_{\myr{H^2}}+\|\mathcal G_{\mathbbm{o}}\|_{\myr{H^2}})\\
    \leq &C \lf(\lf\|\frac{|\pa_z|^2}{\de_L}A  \de_L(B\de_t^{-1}\mathcal F_\nq)\rg\|_{L^2}^2 +\lf\|\frac{\pa_z(\pa_v-t\pa_z)}{\de_L}A  \de_L(B\de_t^{-1}\mathcal F_\nq)\rg\|_{L^2}^2\rg) \\
    &\quad\times(e^{-\delta\nu^{1/3}t}\|t e^{\delta\nu^{1/3}(|\pa_z|^{2/3}+1)t}\mathcal G_\nq\|_{L^\infty([0,T];\myr{H^{2}})}+\|\mathcal G_{\mathbbm{o}}\|_{L^\infty([0,T];\myr{H^{2}})})\\
    \leq& C\lf\|\sqrt{\frac{-\pa_t \W_I}{\W_I}}A\mathcal F_\nq\rg\|_{L^2}^2(e^{-\delta\nu^{1/3}t}\|t e^{\delta\nu^{1/3}(|\pa_z|^{2/3}+1)t}\mathcal G_\nq\|_{L^\infty([0,T];\myr{H^{2}})}+\|\mathcal G_{\mathbbm{o}}\|_{L^\infty([0,T];\myr{H^{2}})}).
\end{align}
By the Poisson kernel estimate from \cite{WZ23}, we have that the second factor is explicit, i.e.,
\begin{align}
F_2&= \sum_{k,\ell}\iint \frac{|A\mathcal H(k,\eta)|^2}{(\ell^2+(\xi-t\ell)^2) (1+ (k-\ell)^2+(\eta-\ell t-(\xi-\ell t))^2)}  d\eta d\xi\\ &= \sum_{k,\ell}\int\frac{\pi|A\mathcal H(k,\eta)|^2}{|\ell|(1+|k-\ell|^2)^{1/2}}\frac{|\ell|+(1+|k-\ell|^2)^{1/2}}{(|\ell|+(1+|k-\ell|^2)^{1/2})^2+(\eta-\ell t)^2}d\eta\\
&\leq C\sum_{k}\int\pi|A\mathcal H(k,\eta)|^2\sum_{\ell\neq 0}\lf(\frac{1}{|\ell|(1+|k-\ell|^2)^{1/2}}\frac{ 1+|k-\ell|+|\ell|}{(1+|k-\ell|+|\ell|)^2+(\eta-\ell t)^2} \rg)d\eta\\
&\leq C\sum_{k}\int|A\mathcal H(k,\eta)|^2(-\pa_t \W_E) d\eta.
\end{align}
Here, the last estimate is due to Lemma \ref{lem:W_E}. Combining the two estimates about $F_1,\ F_2$, we have the result.

\noindent
{\bf Step \# 2: The Low-High interaction $T_2$.} 
For the Low-High interaction, we obtain that 
\begin{align}
    T_{2}=&\sum_{k,\ell}\iint |i\ell\wh{B\de_t^{-1} \mathcal F}(\ell,\xi) |(\eta-\xi)-(k-\ell)t||A\mathcal G(k-\ell,\eta-\xi)|A(k,\eta)\mathcal H(k,\eta)|d\eta d\xi \\
    \leq& C\myr{\lf\|ik\wh{(B\de_t^{-1}\mathcal F)}\rg\|_{\ell_k^1L^1_\eta}}\|A\sqrt{-\de_L}\mathcal G\|_{L^2}\|A\mathcal H\|_{L^2}\leq C\lan t\ran^{-1}\lf\|\sqrt{\frac{-\pa_t\W_I}{\W_I}}\mathcal F_\nq\rg\|_{H^2}\|A\sqrt{-\de_L}\mathcal G\|_{L^2}\|A\mathcal H\|_{L^2}
   .
\end{align}
Here, in the last line, we have used the estimate
\begin{align*}
&\lf\|ik\wh{(B\de_t \mathcal{F})}\rg\|_{\ell_k^1 L_\eta^1}=C\sum_{k\in\mathbb{Z}\backslash\{0\}}\int \lf|\frac{ik}{|k|^2+|\eta-kt|^2}(\de_L(B\de_t^{-1}\mathcal{F}))^\wedge(k,\eta)\rg|d\eta\\
&\lesssim \lf(\sum_{k\neq 0}\int \frac{1}{|k|^2\lan  \eta\ran^2}d\eta\rg)^{1/2}\lf(\sum_{k\in\mathbb{Z}\backslash\{0\}}\int\underbrace{\frac{1}{\lan \eta\ran^2(1+|t-\eta/k|^2) }}_{\lesssim 1/\lan t\ran^2}\frac{\lan \eta\ran^4|k|^2}{|k|^2+|\eta-kt|^2}|\de_L(B\de_t^{-1}\mathcal{F}_\nq)|^2d\eta\rg)^{1/2}\\
&\lesssim \frac{1}{\lan t\ran}\lf\|\sqrt{\frac{-\pa_t \W_I}{\W_I}}\de_L(B\de_t^{-1}\mathcal{F}_\nq)\rg\|_{H^2} \underbrace{\lesssim}_{\eqref{ell_est_t}} \frac{1}{\lan t\ran}\lf\|\sqrt{\frac{-\pa_t \W_I}{\W_I}} \mathcal{F}_\nq \rg\|_{H^2}.
\end{align*}
This concludes the proof. 
\end{proof}


\end{document}